\documentclass[11 pt]{article}

\usepackage{amsmath,amssymb,amsthm,amscd,amsfonts}
\usepackage{epsfig,pinlabel}
\usepackage[hmargin=1.25in, vmargin=1in]{geometry}
\usepackage[font=small,format=plain,labelfont=bf,up,textfont=it,up]{caption}
\usepackage{stmaryrd}
\usepackage{verbatim}
\usepackage{type1cm}
\usepackage{url}
\usepackage{calc}
\usepackage[all,cmtip]{xy}
\usepackage{enumitem}
\usepackage{microtype}

\newcommand{\urlwofont}[1]{\urlstyle{same}\url{#1}}

\newcommand{\nc}{\newcommand}
\nc{\nt}{\newtheorem}
\nc{\dmo}{\DeclareMathOperator}

\theoremstyle{plain}
\newtheorem{theorem}{Theorem}[section]
\newtheorem{maintheorem}{Theorem}
\newtheorem{proposition}[theorem]{Proposition}
\newtheorem{lemma}[theorem]{Lemma}

\newtheorem{corollary}[theorem]{Corollary}

\theoremstyle{definition}

\theoremstyle{remark}
\newtheorem*{remark}{Remark}

\DeclareMathOperator{\Ker}{ker}

\DeclareMathOperator{\Mod}{Mod}
\newcommand\Torelli{\mathcal{I}}
\dmo{\SMod}{SMod}
\dmo{\PMod}{PMod}
\dmo{\SHomeo}{SHomeo}
\dmo{\SI}{\mathcal{SI}}
\dmo{\SSp}{SSp}
\dmo{\PSp}{PSp}

\DeclareMathOperator{\Sp}{Sp}
\DeclareMathOperator{\GL}{GL}
\DeclareMathOperator{\SL}{SL}

\newcommand\R{\ensuremath{\mathbb{R}}}
\newcommand\C{\ensuremath{\mathbb{C}}}
\newcommand\Z{\ensuremath{\mathbb{Z}}}
\newcommand\Q{\ensuremath{\mathbb{Q}}}

\newcommand\Figure[3]{
\begin{figure}[t]
\centering
\centerline{\psfig{file=#2,scale=60}}
\caption{#3}
\label{#1}
\end{figure}}

\nc{\p}[1]{\bigskip \noindent {\bf #1.}}
\nc{\margin}[1]{\marginpar{\raggedright \scriptsize #1}}

\nc{\PartialIBases}{\mathfrak{IB}}
\nc{\PartialIBasesEx}{\widehat{\mathfrak{IB}}}
\nc{\PartialBases}{\mathfrak{B}}
\nc{\Building}{\mathfrak{T}}
\nc{\height}{\ensuremath{\text{ht}}}
\nc{\Poset}{\mathfrak{P}}
\nc{\Field}{\mathbb{F}}
\nc{\Link}{\ensuremath{\text{Link}}}
\nc{\Star}{\ensuremath{\text{Star}}}
\nc{\Aut}{\ensuremath{\text{Aut}}}

\nc{\SymTorelli}{\ensuremath{\mathcal{SI}}}
\nc{\BTorelli}{\ensuremath{\mathcal{BI}}}
\dmo{\Braid}{\ensuremath{B}}
\dmo{\PureBraid}{\ensuremath{PB}}
\nc{\Hyper}{\ensuremath{\iota}}

\nc{\BigFreeProd}{\mathop{\mbox{\Huge{$\ast$}}}}

\nc{\Quotient}{\ensuremath{\mathcal{Q}}}
\nc{\QuotientEx}{\ensuremath{\widehat{\mathcal{Q}}}}

\nc{\Presentation}[2]{\ensuremath{\text{$\langle #1$ $|$ $#2 \rangle$}}}
\nc{\SpGen}{\ensuremath{S_{\text{Sp}}}}
\nc{\SpRel}{\ensuremath{R_{\text{Sp}}}}
\nc{\QGen}{\ensuremath{S_{\mathcal{Q}}}}
\nc{\QRel}{\ensuremath{R_{\mathcal{Q}}}}
\nc{\PBs}{\ensuremath{T}}
\nc{\Qs}{\ensuremath{\overline{s}}}

\dmo{\PB}{PB}
\nc{\BIredg}{\mathcal{BI}_{2g+1}^{\text{red}}}
\nc{\BI}{\mathcal{BI}}
\dmo{\D}{D}
\dmo{\Stab}{Stab}
\dmo{\Surger}{Surger}
\nc{\I}{\mathcal{I}}
\renewcommand{\C}{\mathcal{C}}

\nc{\spanmap}{span}

\nc{\genbygen}[2]{\premonoid{#1}{#2}}
\nc{\premonoid}[2]{#1 \circledcirc #2}
\nc{\monoid}[2]{#1 \odot #2}

\nc{\lax}[1]{(#1)_{\pm}}

\title{\vspace{-30pt}Generators for the hyperelliptic Torelli group and the kernel of the Burau representation at $t=-1$}

\author{Tara Brendle, Dan Margalit, and Andrew Putman\footnote{The first author is supported in part by EPSRC grant EP/J019593/1. The second author is supported in part by a Sloan Fellowship and an NSF CAREER award.  The third author is supported in part by NSF grant DMS-1005318.}\vspace{-6pt}}

\begin{document}

\newcounter{enumi_saved}

\maketitle

\vspace{-23pt}
\begin{abstract}
We prove that the hyperelliptic Torelli group is generated by Dehn twists about separating curves that are preserved by the hyperelliptic involution.  This verifies a conjecture of Hain.  The hyperelliptic Torelli group can be identified with the kernel of the Burau representation evaluated at $t=-1$ and also the fundamental group of the branch locus of the period mapping, and so we obtain analogous generating sets for those.  One application is that each component in Torelli space of the locus of hyperelliptic curves becomes simply connected when curves of compact type are added.
\end{abstract}

\section{Introduction}
\label{section:intro}

In this paper, we find simple generating sets for three closely related groups:
\begin{enumerate}\setlength{\itemsep}{0ex}
 \item the hyperelliptic Torelli group $\SymTorelli_g$, that is, the subgroup of the mapping class group consisting of elements that commute with some fixed hyperelliptic involution and that act trivially on the homology of the surface;
\item the fundamental group of $\widetilde{\mathcal{H}}_g$, the branch locus of the period mapping from Torelli space to the Siegel upper half-plane; and
\item the kernel of $\beta_n$, the Burau representation of the braid group evaluated at $t=-1$ (the representation $\beta_n$ is sometimes known as the integral Burau representation).
\end{enumerate}
The group $\SymTorelli_g$, the space $\widetilde{\mathcal{H}}_g$, and the representation $\beta_n$
arise in many places in algebraic geometry, number theory, and topology;   
see, e.g., the work
of A'Campo \cite{ACampo}, Arnol'd \cite{ArnoldHyperelliptic}, Band--Boyland \cite{BandBoyland}, Funar--Kohno \cite{FunarKohno}, Gambaudo--Ghys \cite{GambaudoGhys}, Hain \cite{HainConjecture}, Khovanov--Seidel \cite{KhovanovSeidel}, Magnus--Peluso \cite{MagnusPeluso},
McMullen \cite{McMullenHodge}, Morifuji \cite{Morifuji}, Venkataramana \cite{Venkataramana}, and Yu \cite{YuCohenLenstra}.

\p{Hyperelliptic Torelli group}
Let $\Sigma_{g}$ be a closed oriented surface of genus $g$ and let $\Mod_{g}$ be its mapping class group, that is, the group of isotopy classes of orientation-preserving homeomorphisms of $\Sigma_{g}$.

Let $\Hyper : \Sigma_{g} \rightarrow \Sigma_{g}$ be a hyperelliptic
involution; see Figure~\ref{figure:introduction}.  By definition a hyperelliptic involution is an order two homeomorphism of $\Sigma_g$ that acts by $-I$ on $H_1(\Sigma_g;\Z)$, and it is a fact there is a unique hyperelliptic involution up to conjugacy by homeomorphisms of $\Sigma_g$; we fix one once and for all.  The hyperelliptic mapping class group $\SMod_{g}$ is the subgroup of $\Mod_{g}$ 
consisting of mapping classes that can be represented by homeomorphisms that commute with $\Hyper$.
The Torelli group $\Torelli_{g}$ is the kernel of the action of $\Mod_{g}$ on $H_1(\Sigma_{g};\Z)$, and the
hyperelliptic Torelli group $\SymTorelli_{g}$ is $\SMod_{g} \cap \Torelli_{g}$.

A simple closed curve $x$ in $\Sigma_{g}$  is {\em symmetric} if $\Hyper(x) = x$, in which case the Dehn twist $T_x$ is in $\SMod_{g}$. If $x$ is a separating curve, then $T_x \in \Torelli_{g}$; see Figure~\ref{figure:introduction}.

\begin{maintheorem}
\label{maintheorem:mainmod}
For $g \geq 0$, the group $\SymTorelli_{g}$ is generated by Dehn twists about symmetric separating curves.
\end{maintheorem}

The first two authors proved that Theorem~\ref{maintheorem:mainmod} in fact implies the stronger result that $\SymTorelli_g$ is generated by Dehn twists about symmetric separating curves that cut off subsurfaces of genus 1 and 2 \cite[Proposition 1.5]{BrendleMargalitPoint}.  

Theorem~\ref{maintheorem:mainmod} was conjectured by Hain \cite[Conjecture 1]{HainConjecture} and is also listed as a folk conjecture by Morifuji \cite[Section 4]{Morifuji}.
Hain has informed us that he has proven the case $g=3$ 
of Theorem~\ref{maintheorem:mainmod}.  His proof uses special properties of the Schottky locus in genus $3$.

When we first encountered Hain's conjecture, it appeared to us to be overly optimistic.  There is a well-known generating set for $\Torelli_g$, namely, the set of bounding pair maps and Dehn twists about separating curves; see Figure~\ref{figure:introduction}.  There is no reason to expect that an infinite-index subgroup of $\Torelli_g$ should be generated by the elements on this list lying in the subgroup. 
Additionally, there are several other natural elements of $\SymTorelli_g$, and it was not at first clear how to write those in terms of Hain's proposed generators.  Consider for instance the mapping class $[T_uT_{u'},T_vT_{v'}] \in \SymTorelli_g$ indicated in Figure~\ref{figure:introduction}.  Eventually, it turned out this element is a product of six Dehn twists about symmetric separating curves \cite{BrendleMargalitFactor}, but the curves are rather complicated looking. 

\Figure{figure:introduction}{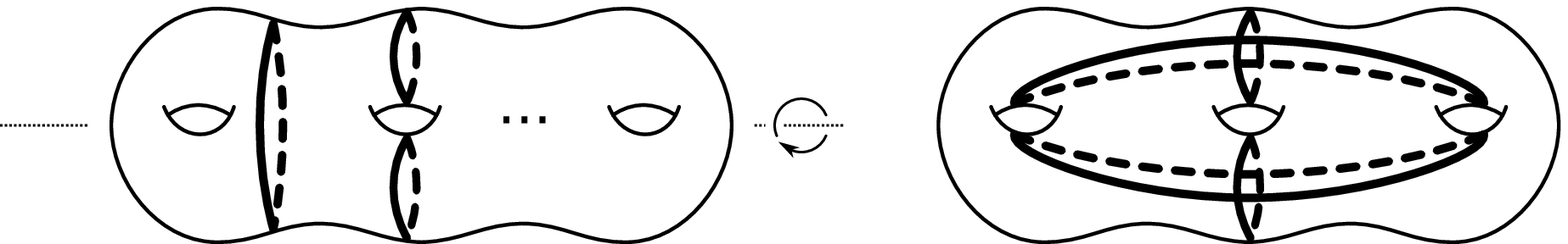}{The hyperelliptic involution $\Hyper$ rotates the surface 180 degrees 
about the indicated axis.  The mapping
class $T_x$ is a Dehn twist about a symmetric separating curve.   A bounding pair map, such as $T_y T_z^{-1}$, is the difference of two Dehn twists about disjoint, nonseparating,  homologous simple closed curves.  The mapping classes $T_u T_{u'}$ and $T_v T_{v'}$ are in $\SMod_g$ and
their actions on $H_1(\Sigma_g;\Z)$ commute because $\hat\imath(u,v)=\hat\imath(u',v')=0$, so $[T_u T_{u'},T_v T_{v'}] \in \SymTorelli_{g}$.}

\p{Branch locus of the period map}
Hain \cite{HainConjecture} observed that Theorem~\ref{maintheorem:mainmod} has an interpretation in terms of the period map.  Let $\mathcal{T}_g$ be Teichm\"{u}ller space and $\mathfrak{h}_g$ 
the Siegel upper half-plane.  The period map $\mathcal{T}_g \rightarrow \mathfrak{h}_g$ takes
a Riemann surface to its Jacobian.  It factors through the Torelli space $\mathcal{T}_g / \Torelli_g$, which is an Eilenberg--MacLane space
for $\Torelli_g$.  The induced map $\mathcal{T}_g / \Torelli_g \rightarrow \mathfrak{h}_g$ is a $2$-fold
branched cover onto its image.  The branch locus is the subspace $\widetilde{\mathcal{H}}_g \subset \mathcal{T}_g / \Torelli_g$
consisting of points that project to the hyperelliptic locus $\mathcal{H}_g$ in the moduli space of curves.
The space $\widetilde{\mathcal{H}}_g$ is not connected, but its components
are all homeomorphic and have fundamental group $\SymTorelli_{g}$.  Thus, Theorem~\ref{maintheorem:mainmod} gives generators for $\pi_1(\widetilde{\mathcal{H}}_g)$.

Let $\widetilde{\mathcal{H}}_g^c$ be the space obtained by adjoining hyperelliptic curves of compact type to
$\widetilde{\mathcal{H}}_g$.  Theorem~\ref{maintheorem:mainmod} has the following corollary.  

\begin{maintheorem}
\label{maintheorem:periodmap}
For $g \geq 0$, each component of $\widetilde{\mathcal{H}}_g^c$ is simply connected.
\end{maintheorem}

See Hain's paper \cite{HainConjecture} for the details on how to derive Theorem~\ref{maintheorem:periodmap} from Theorem~\ref{maintheorem:mainmod}.  The main idea is that when we add to $\widetilde{\mathcal{H}}_g$ a hyperelliptic curve of compact type obtained by degenerating a symmetric separating simple closed curve in a hyperelliptic curve, the effect on $\pi_1(\widetilde{\mathcal{H}}_g)$ is to kill the generator of $\pi_1(\widetilde{\mathcal{H}}_g)$ given by the corresponding Dehn twist.  

\p{Kernel of the Burau representation}
The (unreduced) Burau representation \cite{BirmanBraids} is an important representation of the braid group $\Braid_{n}$ to $\GL_{n}(\Z[t,t^{-1}])$.  
Let $\beta_{n} : \Braid_{n} \rightarrow \GL_{n}(\Z)$ be the representation obtained by substituting $t=-1$ into the Burau representation.  
Denote the kernel of $\beta_n$ by $\BTorelli_{n}$ (the notation stands for ``braid Torelli group'').

We identify $\Braid_{n}$ with the mapping class group of a disk $\D_n$ with $n$ marked points, that is, the group of isotopy classes of homeomorphisms of $\D_n$ preserving the set of marked points and fixing $\partial \D_n$ pointwise.  For most purposes, we will regard the marked points as punctures.  For instance, curves (and homotopies of curves) are not allowed to pass through the marked points.  When we say that a simple closed curve is essential in $\D_n$, we mean that it is not homotopic to a marked point, an unmarked point, or the boundary.

\begin{maintheorem}
\label{maintheorem:mainbraid}
For $n \geq 1$, the group $\BTorelli_{n}$ is generated by squares of Dehn twists about curves in $\D_n$ surrounding odd numbers of marked points.
\end{maintheorem}

Just like for Theorem~\ref{maintheorem:mainmod}, our proof gives more, namely that 
$\BTorelli_{n}$ is generated by squares of Dehn twists about curves surrounding exactly 3 or 5 marked points.  As pointed out to us by Neil Fullarton, both types of twists are needed.  Indeed, the abelianization homomorphism $\Braid_{n} \to \Z$ maps the square of a Dehn twist about a curve surrounding $2k+1$ marked points to $8k^2+4k$.  Thus the group generated by squares of Dehn twists about curves surrounding 3 marked points maps to $12\Z$ and the group generated by squares of Dehn twists about 5 marked points maps to $40\Z$.  Since $\gcd(12,40)=4$ the image of the group generated by squares of both types of Dehn twists---hence the image of $\BTorelli_{n}$---contains $4\Z$.

\p{Hyperelliptic Torelli vs Burau} 
We now explain the relationship between Theorems~\ref{maintheorem:mainmod} and~\ref{maintheorem:mainbraid}.  
This requires defining the hyperelliptic Torelli group for a surface with boundary.

Let $\Sigma_g^1$ be the surface obtained from $\Sigma_g$ by deleting the interior of an embedded $\Hyper$-invariant disk.  There is an induced 
hyperelliptic involution of $\Sigma_g^1$ which we also call $\Hyper$.  Let $\Mod_g^1$ be the group of  
isotopy classes of homeomorphisms of $\Sigma_g^1$ that fix $\partial \Sigma_g^1$ pointwise and
let $\SMod_g^1$ be the subgroup of $\Mod_g^1$ consisting of mapping classes that can be represented 
by homeomorphisms that commute with $\Hyper$.  Observe that unlike for $\Mod_g$, the map
$\Hyper$ does not correspond to an element of $\Mod_g^1$.
Finally, let $\Torelli_g^1$ be the kernel of the action of $\Mod_g^1$ on $H_1(\Sigma_g^1;\Z)$ and let
$\SymTorelli_g^1 = \SMod_g^1 \cap \Torelli_g^1$.

The involution $\Hyper$ fixes $2g+1$ points of $\Sigma_{g}^{1}$. Regarding the images of these points in $\Sigma_{g}^{1} / \Hyper$ as marked points, 
we have $\Sigma_{g}^{1} / \Hyper \cong \D_{2g+1}$.  There is a homomorphism $L : \Braid_{2g+1} \rightarrow \SMod_{g}^{1}$
which lifts a mapping class through the branched cover $\Sigma_{g}^{1} \rightarrow \Sigma_{g}^{1} / \Hyper$.
Birman--Hilden \cite{BirmanHilden} proved that $L$ is an isomorphism.  Under this isomorphism, a Dehn twist about a curve $c$ surrounding an odd number of marked points maps to a half-twist about the (connected) preimage of $c$ in $\Sigma_g^1$; in particular, the Dehn twist about $\partial\D_{2g+1}$ acts by $-I$ on $H_1(\Sigma_g^1)$.  Similarly, a Dehn twist about a curve $c$ surrounding an even number of marked points maps to the product of the Dehn twists about the two components of the preimage of $c$.  

The representation $\beta_{2g+1}$ decomposes into two irreducible representations.  One is  the 1-dimensional trivial representation, and the other is conjugate to the composition
$$\Braid_{2g+1} \stackrel{L}{\longrightarrow} \SMod_{g}^{1} \hookrightarrow \Mod_{g}^{1} \longrightarrow \Sp_{2g}(\Z),$$
where the map $\Mod_{g}^{1} \rightarrow \Sp_{2g}(\Z)$ is the standard representation arising from the action of
$\Mod_{g}^{1}$ on $H_1(\Sigma_{g}^{1};\Z)$.  The map $L$ therefore restricts to an isomorphism $\BTorelli_{2g+1} \cong \SymTorelli_{g}^{1}$.
Under this isomorphism, squares of Dehn twists about curves surrounding odd numbers of marked points map to Dehn twists about 
symmetric separating curves.  The case $n=2g+1$ of Theorem~\ref{maintheorem:mainbraid} is therefore equivalent to the 
statement that $\SymTorelli_g^1$ is generated by Dehn twists about symmetric separating curves.  The first two authors proved \cite[Theorem 4.2]{BrendleMargalitPoint}
that the kernel of the natural map $\SymTorelli_g^1 \to \SymTorelli_g$ is generated by the Dehn twist about $\partial \Sigma_g^1$, so this
is equivalent to Theorem~\ref{maintheorem:mainmod}.

We can also relate Theorem~\ref{maintheorem:mainbraid} for even numbers of punctures to the mapping
class group by extending Theorem~\ref{maintheorem:mainmod} 
to the case of a surface with two boundary components.  Briefly, let $\Sigma_g^2$ be the compact surface of genus $g$ with two boundary components obtained by
removing the interiors of two disks in $\Sigma_g$ that are interchanged by $\Hyper$.  Again, there
is an induced hyperelliptic involution of $\Sigma_g^2$ which we will also call $\Hyper$.  The
homeomorphism $\Hyper$ interchanges the two boundary components of $\Sigma_g^2$.
We can define $\SMod_g^2$ as before.  The Torelli group $\Torelli_g^2$ is the kernel of the action of $\SMod_g^2$ on $H_1(\Sigma_g^2,P;\Z)$, where $P$ is a pair of points, one on each boundary component of $\Sigma_g^2$.  The hyperelliptic Torelli group $\SI_g^2 = \SMod_g^2 \cap \Torelli_g^2$ is then isomorphic to $\BTorelli_{2g+2}$.  The $n=2g+2$ case of Theorem~\ref{maintheorem:mainbraid} translates to the fact that $\SI_g^2$ is generated by Dehn twists about symmetric separating curves.  See \cite{BrendleMargalitPoint} for more details.

\p{Prior results}
Theorem~\ref{maintheorem:mainmod} was previously known for $g \leq 2$.  It is a classical fact that $\Torelli_{g}=1$ for $g \leq 1$, so $\SymTorelli_{g}$ is trivial in 
these cases.  When $g=2$, all essential curves in $\Sigma_g$ are homotopic to symmetric curves.  This implies that $\SMod_2 = \Mod_2$ and $\SymTorelli_{2} = \Torelli_{2}$ (see, e.g., \cite[Section 9.4.2]{FarbMargalitPrimer}).  The group $\Torelli_2$ is 
generated by Dehn twists about separating curves; in fact, Mess \cite{Mess} proved that $\Torelli_2$ is a free group on an infinite set of Dehn 
twists about separating curves (McCullough--Miller \cite{McCulloughMiller} previously showed $\Torelli_2$ was infinitely generated).  This implies that $\SymTorelli_{2}$ is generated by Dehn twists about symmetric separating curves.

Theorem~\ref{maintheorem:mainmod} was known for $n \leq 6$.   The group $\BTorelli_{n}$ is trivial for $n \leq 3$.  Smythe showed \cite{Smythe} that $\BTorelli_4 \cong F_\infty$.  He also identified the  generating set from Theorem~\ref{maintheorem:mainbraid}.  The group $\BTorelli_4$ is isomorphic to the stabilizer in $\SymTorelli_2$ of a nonseparating simple closed curve, and so Smythe's theorem can be considered as a precursor to Mess's theorem.   Next, for $g \geq 1$ the first two authors proved \cite[Theorem 4.2]{BrendleMargalitPoint} that $\BTorelli_{2g+1}$ is isomorphic to $\SymTorelli_g \times \Z$, and so by Mess's theorem $\BTorelli_5$ is isomorphic to $F_\infty \times \Z$ and further it satisfies Theorem~\ref{maintheorem:mainbraid}.  The first two authors also proved \cite[Theorem 1.2]{BrendleMargalitPoint} that $\BTorelli_{2g+2}$ is isomorphic to $(\BTorelli_{2g+1}/\Z) \ltimes F_\infty$ and that each element of the $F_\infty$ subgroup is a product of squares of Dehn twists about curves surrounding odd numbers of marked points and so again by Mess's theorem we obtain that $\BTorelli_6$ is isomorphic to $F_\infty \ltimes F_\infty$ and that it also satisfies Theorem~\ref{maintheorem:mainbraid}. 

Aside from our Theorem~\ref{maintheorem:mainmod}, little is known about $\SymTorelli_{g}$ when $g \geq 3$.  Letting $H = H_1(\Sigma_g;\Z)$, 
Johnson \cite{JohnsonHomo, JohnsonSecond} constructed a $\Mod_{g}$-equivariant homomorphism $\tau : \Torelli_g \rightarrow (\wedge^3 H) / H$ 
and proved that $\Ker(\tau)$ is precisely the subgroup $\mathcal{K}_g$ of $\Torelli_g$ generated by Dehn twists about (not-necessarily-symmetric) separating curves.  Since $\Hyper$ acts by $-I$ on $(\wedge^3 H)/H$, it follows that $\SymTorelli_{g} < \mathcal{K}_g$.  Despite the fact that $\SymTorelli_{g}$ has infinite index in $\mathcal{K}_g$, Childers \cite{ChildersThesis} showed that these groups have the same image in the abelianization of $\Torelli_g$.

Birman \cite{BirmanSiegel} and Powell \cite{PowellTorelli} showed that $\Torelli_g$ is generated by bounding pair maps and Dehn 
twists about separating curves; other proofs were given by Putman \cite{PutmanCutPaste} and by Hatcher--Margalit \cite{HatcherMargalit}.   
One can find bounding pair maps $T_y T_z^{-1}$ such that $\Hyper$ exchanges $y$ and $z$ (see Figure~\ref{figure:introduction}); however, these 
do not lie in $\SymTorelli_{g}$ since $\Hyper T_y T_{z}^{-1} \Hyper^{-1} = T_z T_y^{-1}$.  In fact, since no power of a bounding pair map is in $\Ker(\tau)$,
there are no nontrivial powers of bounding pair maps in $\SymTorelli_{g}$.

With Childers, the first two authors  proved that $\SymTorelli_{g}$ has cohomological dimension $g-1$ and that $H_{g-1}(\SymTorelli_{g};\Z)$ has infinite rank \cite{BrendleMargalitChilders}.  This implies $\SymTorelli_{3}$ is not finitely presentable.  It is not known, however, whether $\SymTorelli_{g}$, or even
$H_1(\SymTorelli_{g};\Z)$, is finitely generated for $g \geq 3$.

\p{Approach of the paper} 
The simplest proofs that the mapping class group is generated by Dehn twists or that the Torelli group is generated by separating twists and bounding pair maps rely on the connectivity of certain complexes of curves.   One natural complex in our setting has vertices in bijection with the $\SymTorelli_g$-orbit of the isotopy class of a symmetric nonseparating curve and edges for curves with the minimal possible intersection number.  However, we do not know if this complex is connected, so our proof requires a new approach.  

First, to prove Theorems~\ref{maintheorem:mainmod} and~\ref{maintheorem:mainbraid}, it suffices to prove Theorem~\ref{maintheorem:mainbraid} for $n = 2g+1$.  Indeed, we already said that the $n=2g+1$ case of Theorem~\ref{maintheorem:mainbraid} is equivalent to the genus $g$ case of
Theorem~\ref{maintheorem:mainmod} and the first two authors proved \cite[Theorems 1.4 and 4.2]{BrendleMargalitPoint} that the $n=2g+1$ case of Theorem~\ref{maintheorem:mainbraid} implies the $n=2g+2$ case of
Theorem~\ref{maintheorem:mainbraid}. 

As we explain in Section~\ref{section:outline}, we have $\BTorelli_{2g+1} \subseteq \PureBraid_{2g+1}$ and
there is an isomorphism $\PureBraid_{2g+1}/\BTorelli_{2g+1} \cong \Sp_{2g}(\Z)[2]$. 
Theorem~\ref{maintheorem:mainbraid} in this case is thus equivalent to the assertion that 
$\PureBraid_{2g+1} / \Theta_{2g+1} \cong \Sp_{2g}(\Z)[2]$, where $\Theta_{2g+1}$ is the group generated by 
squares of Dehn twists about curves surrounding odd numbers of marked points.

This isomorphism can be viewed as a finite presentation for $\Sp_{2g}(\Z)[2]$
since $\PureBraid_{2g+1}$ is finitely presented and $\Theta_{2g+1}$ has a 
finite normal generating set.  There are several known presentations for 
$\Sp_{2g}(\Z)$.  Also, there are standard tools for obtaining finite presentations of finite-index subgroups of
finitely presented groups (e.g.\ Reidemeister--Schreier).  However, they all explode
in complexity as the index of the subgroup grows.  And even if we had some finite presentation for $\Sp_{2g}(\Z)[2]$, there is no reason to hope that such a presentation would be equivalent to the one given by the (purported)  isomorphism $\PureBraid_{2g+1} / \Theta_{2g+1} \cong \Sp_{2g}(\Z)[2]$.  

What we do instead is to apply a theorem of the third author in order to obtain an infinite presentation of $\Sp_{2g}(\Z)[2]$ with a certain amount of symmetry, and then introduce a new method for changing this presentation into the finite presentation $\PureBraid_{2g+1} / \Theta_{2g+1}$.  The tools we construct should be useful in other contexts.  In fact, they have already been used by the last two authors to give finite presentations for certain congruence subgroups of $\SL_n(\Z)$ which are reminiscent of the standard presentation for $\SL_n(\Z)$; see \cite{MargalitPutman}.

\p{Acknowledgments}
We would like to thank Joan Birman, Matt Day, Benson Farb, Neil Fullarton, Richard Hain, Sebastian Hensel, Allen Hatcher, Curt McMullen, and Takayuki Morifuji for helpful comments and conversations.  We are also grateful to Susumu Hirose, Yasushi Kasahara, Masatoshi Sato, and Wataru Yuasa for pointing out mistakes in an earlier draft and to Tom Church for providing invaluable comments.   We are especially indebted to our referees who provided us with many corrections and insightful suggestions.


\section{Outline of paper}
\label{section:outline}

Recall from the introduction that to prove Theorems~\ref{maintheorem:mainmod}, \ref{maintheorem:periodmap}, and~\ref{maintheorem:mainbraid} it is enough to prove Theorem~\ref{maintheorem:mainbraid} for $n=2g+1$.
Since $\BTorelli_{2g+1}$ is known to satisfy Theorem~\ref{maintheorem:mainbraid} for $0 \leq g \leq 2$, 
we can apply induction with $g=2$ as the base case.  Instead of proving Theorem~\ref{maintheorem:mainbraid} directly, we will work with a mild rephrasing, namely, Proposition~\ref{proposition:main} below.  After stating this proposition,
we give an outline of the proof and a plan for the remainder of the paper.

\p{Background}   Arnol'd \cite{ArnoldHyperelliptic} proved that the kernel of the composition
$$\Braid_{2g+1} \stackrel{\beta_{2g+1}}{\longrightarrow} \Sp_{2g}(\Z) \longrightarrow \Sp_{2g}(\Z/2)$$ 
is exactly $\PureBraid_{2g+1}$.  In particular, the image of $\PureBraid_{2g+1}$ under $\beta_{2g+1}$ lies in  the level $2$ congruence subgroup of $\Sp_{2g}(\Z)$, namely, $\Sp_{2g}(\Z)[2] = \ker(\Sp_{2g}(\Z) \rightarrow \Sp_{2g}(\Z/2))$.  A'Campo \cite{ACampo} then proved that $\beta_{2g+1}(\PureBraid_{2g+1})$ is all of $\Sp_{2g}(\Z)[2]$.  These two results can be summarized in the following commutative diagram.  In the diagram, $S_{2g+1}$ is the symmetric group, and the map $\Braid_{2g+1} \to S_{2g+1}$ is the action on the marked points of $\D_{2g+1}$. 
$$\xymatrix{
1 \ar[r] & \PureBraid_{2g+1} \ar[r] \ar@{->>}[d] & \Braid_{2g+1} \ar[r] \ar_{\beta_{2g+1}}[d] & S_{2g+1}     \ar[r] \ar@{^{(}->}[d] & 1\\
1 \ar[r] & \Sp_{2g}(\Z)[2]   \ar[r]              & \Sp_{2g}(\Z)  \ar[r]               & \Sp_{2g}(\Z/2) \ar[r]        & 1}$$
In particular, we see that $\BTorelli_{2g+1} < \PureBraid_{2g+1}$ and $\PureBraid_{2g+1} / \BTorelli_{2g+1} \cong \Sp_{2g}(\Z)[2]$.

\p{The Main Proposition}  Recall that $\Theta_{2g+1}$ is the group generated by squares of Dehn twists about curves surrounding odd numbers of marked points.  Denote the quotient $\PureBraid_{2g+1} / \Theta_{2g+1}$ by $\Quotient_g$.  Since Dehn twists about symmetric separating curves in $\Sigma_g^1$ correspond to squares of Dehn twists about curves surrounding odd numbers of marked points in $\D_{2g+1}$, we have $\Theta_{2g+1} \leqslant \BTorelli_{2g+1}$ and so there is a further quotient map $\Quotient_g \rightarrow \Sp_{2g}(\Z)[2]$.  The $n=2g+1$ case of Theorem~\ref{maintheorem:mainbraid} is then equivalent to the following.

\begin{proposition}
\label{proposition:main}
For $g \geq 2$, the quotient map $\pi : \Quotient_g \rightarrow \Sp_{2g}(\Z)[2]$ is an isomorphism.
\end{proposition}

The starting point is the following theorem of the first two authors \cite[Theorems 1.3 and 4.2]{BrendleMargalitPoint}, which makes it easy to recognize when certain elements of $\BTorelli_{2g+1}$ lie in $\Theta_{2g+1}$ (or, when the corresponding elements of $\ker \pi$ are trivial).  We say an element of $\BTorelli_{2g+1}$ is {\em reducible} if it fixes the homotopy class of an essential simple closed curve in $\D_{2g+1}$.  

\begin{theorem}
\label{theorem:reducibletotwists}
If $\BTorelli_{2h+1} = \Theta_{2h+1}$ for all $h < g$, then all reducible elements
of $\BTorelli_{2g+1}$ lie in $\Theta_{2g+1}$.
\end{theorem}

This theorem is derived from a version of the Birman exact sequence for $\SymTorelli_g$.  It is used at various points in the proof, specifically Sections~\ref{section:calc1}, \ref{section:calc2},  and~\ref{section:theproofmain}.

To prove Proposition~\ref{proposition:main}, it suffices to construct an inverse map $\phi:\Sp_{2g}(\Z)[2] \rightarrow \Quotient_g$.  Besides Theorem~\ref{theorem:reducibletotwists}, there are two main ingredients to the construction.  We describe them and at the same time give an outline for the rest of the paper.
\begin{enumerate}
\item The first ingredient, Proposition~\ref{prop:sp presentation}, is an infinite presentation for $\Sp_{2g}(\Z)[2]$.  This presentation has two key properties: first, the set of generators is the union of the stabilizers of nontrivial elements of $\Z^{2g}$, and second, the action of $\Sp_{2g}(\Z)$ on $\Sp_{2g}(\Z)[2]$ by conjugation permutes the generators and relations in a natural way.
This presentation is obtained by considering the action of $\Sp_{2g}(\Z)[2]$ on a certain simplicial complex $\PartialIBases_g(\Z)$ which itself admits an $\Sp_{2g}(\Z)$ action.  

The method for constructing such infinite presentations from group actions is due to the third author, who used it to construct an infinite presentation of the Torelli group \cite{PutmanInfinite}.  The theory requires the complexes being acted upon to have certain connectivity properties, and our contribution is to verify these properties.

\item Our second ingredient is an action of $\Sp_{2g}(\Z)$ on $\Quotient_g$ that is compatible (via $\pi$) with the action of $\Sp_{2g}(\Z)$ on $\Sp_{2g}(\Z)[2]$ in a sense made precise by Proposition~\ref{proposition:spaction}.  We construct the action by declaring where each generator of $\Sp_{2g}(\Z)$ sends each generator of $\Quotient_g$ and then checking that all relations in both groups are satisfied.  This step uses a mixture of surface topology and combinatorial group theory.
\end{enumerate}
We deal with the above two ingredients in Sections~\ref{section:presentation} and~\ref{section:action}, respectively.  In Section~\ref{section:theproofmain} we will use Theorem~\ref{theorem:reducibletotwists} to define a homomorphism $(\Sp_{2g}(\Z)[2])_{\vec v} \to \Quotient_g$, where $(\Sp_{2g}(\Z)[2])_{\vec v}$ is the stabilizer of some fixed $\vec v \in \Z^{2g}$.  Then, we will use the fact that $\Sp_{2g}(\Z)[2]$ and $\Quotient_g$ have compatible $\Sp_{2g}(\Z)$ actions to propagate this to a globally-defined map $\Sp_{2g}(\Z)[2] \to \Quotient_g$, thus proving Proposition~\ref{proposition:main}.


\section{An infinite presentation for \boldmath$\Sp_{2g}(\Z)[2]$}
\label{section:presentation}

In this section, we discuss the first ingredient from Section~\ref{section:outline}, a special kind of presentation for the group $\Sp_{2g}(\Z)[2]$.  The presentation will be derived from a general theorem of the third author about obtaining a presentation of a group from its action on a simplicial complex.

\subsection{Setup} 

Our first goal is to give a precise statement of the desired presentation of $\Sp_{2g}(\Z)$, namely
Proposition~\ref{prop:sp presentation} below.  

\p{A presentation theorem} Let $G$ be a group acting without rotations on a simplicial complex $X$; this means that if an element of $G$ preserves some simplex of $X$ then it fixes that simplex pointwise.  For a simplex $\sigma$, write $G_\sigma$ for the stabilizer of $\sigma$.  Also write $X^{(0)}$ for the vertex set of $X$.  There is a homomorphism
$$\psi : \BigFreeProd_{v \in X^{(0)}} G_v \longrightarrow G$$
induced by the various inclusions $G_v \to G$.  If $a \in G$ stabilizes $v \in X^{(0)}$, then we denote by $a_v$ the image of $a$ under the inclusion
$$G_v \to \BigFreeProd_{v \in X^{(0)}} G_v.$$
There are some obvious elements in $\Ker(\psi)$.  First, if $e$ is an edge with vertices $v$ and $v'$ and $a \in G_{e}$ then $a_v a_{v'}^{-1} \in \Ker(\psi)$.  We call these the {\em edge relators}.
Second, we have $b_w a_v b_w^{-1} (b a b^{-1})_{b(v)}^{-1} \in \Ker(\psi)$
for $a \in G_v$ and $b \in G_w$.  We call these the {\em conjugation
relators}.  The following theorem of the third author \cite{PutmanPresentations} states that under certain circumstances these two types of relators suffice to normally generate $\ker(\psi)$.

\begin{theorem}
\label{theorem:presentation}
Let $G$ be a group acting without rotations on a simplicial complex $X$.  Assume that $X$ is $1$-connected
and $X/G$ is $2$-connected.  Then
$$G \cong \left(\BigFreeProd_{v \in X^{(0)}} G_v\right)\Big/R,$$
where $R$ is the normal closure of the edge and conjugation relators.
\end{theorem}

\p{Symplectic bases}
Let $R$ be either $\Z$ or a field and let $\hat \imath$ be the standard symplectic form on $R^{2g}$.  A {\em symplectic basis} for $R^{2g}$ is a pair of $g$-tuples $(\vec{a}_1,\ldots,\vec{a}_g;\vec{b}_1,\ldots,\vec{b}_g)$ of elements of $R^{2g}$ that together form a free basis and satisfy
\[
\hat \imath(\vec{a}_i,\vec{a}_j) = \hat \imath(\vec{b}_i,\vec{b}_j) = 0 \quad \text{and} \quad \hat \imath(\vec{a}_i,\vec{b}_j)=\delta_{ij} \quad \quad (1 \leq i,j \leq g),
\]
where $\delta_{ij}$ is the Kronecker delta.  A {\em partial symplectic basis} is a
tuple $(\vec{a}_1,\ldots,\vec{a}_k;\vec{b}_1,\ldots,\vec{b}_{\ell})$ of elements of $R^{2g}$ so that
there exist $\vec{a}_{k+1},\ldots,\vec{a}_g,\vec{b}_{\ell+1},\ldots,\vec{b}_g \in R^{2g}$  with 
$(\vec{a}_1,\ldots,\vec{a}_g;\vec{b}_1,\ldots,\vec{b}_g)$ a symplectic basis.  We allow $k=0$ or
$\ell=0$ in this notation.

\p{The complex of lax isotropic bases}
A {\em lax vector} in $R^{2g}$ is a pair $\{\vec{v},-\vec{v}\}$, where $\vec v \in R^{2g}$ is nonzero.  We will denote this pair by $\lax{\vec{v}}$, so $\lax{\vec{v}} = \lax{-\vec{v}}$.
The {\em complex of lax isotropic bases} for $R^{2g}$, denoted $\PartialIBases_g(R)$, is the simplicial complex whose
$(k-1)$-simplices are the sets $\{\lax{\vec{a}_1},\ldots,\lax{\vec{a}_k}\}$, where
$(\vec{a}_1,\ldots,\vec{a}_k;)$ is a partial symplectic basis for $R^{2g}$.  By definition,
$\hat \imath(\vec{a}_i,\vec{a}_j) = 0$ for all $1 \leq i,j \leq k$.
Clearly $\Sp_{2g}(\Z)$, hence $\Sp_{2g}(\Z)[2]$, acts on $\PartialIBases_g(\Z)$.

\p{The augmented complex of lax isotropic bases}
As we will explain shortly, the action of $\Sp_{2g}(\Z)[2]$ on $\PartialIBases_g(\Z)$ satisfies
the hypotheses of Theorem~\ref{theorem:presentation} for $g \geq 4$.  For the special case of $g=3$,
we will need a different complex obtained by attaching some cells to $\PartialIBases_g(\Z)$ (the
vertex set of our new complex is the same as for $\PartialIBases_g(\Z)$). 

We make the following definitions.  A simplex $\{\lax{\vec{a}_1},\ldots,\lax{\vec{a}_k}\}$ of $\PartialIBases_g(R)$
will be called a {\em standard simplex}.  If $(\vec{a}_1,\ldots,\vec{a}_k;\vec{b}_1)$ is a partial
symplectic basis, then the set
$\{\lax{\vec{a}_1},\ldots,\lax{\vec{a}_k},\lax{\vec{b}_1}\}$ is a {\em simplex of intersection type}.  If $(\vec{a}_1,\ldots,\vec{a}_k;)$ is a partial symplectic basis, 
the sets
\[\{\lax{\vec{a}_1+\vec{a}_2},\lax{\vec{a}_1},\lax{\vec{a}_2},\ldots,\lax{\vec{a}_k}\} \quad \text{and} \quad 
\{\lax{\vec{a}_1+\vec{a}_2+\vec{a}_3},\lax{\vec{a}_1},\lax{\vec{a}_2},\ldots,\lax{\vec{a}_k}\}\]
will be called {\em simplices of additive type}.  Since $(\vec{a}_1,-\vec{a}_2,\vec{a}_3,\ldots,\vec{a}_k;)$ is also
a partial symplectic basis, the sets
\[\{\lax{\vec{a}_1 - \vec{a}_2},\lax{\vec{a}_1},\lax{\vec{a}_2},\ldots,\lax{\vec{a}_k}\} \quad \text{and} \quad 
\{\lax{\vec{a}_1 - \vec{a}_2 + \vec{a}_3},\lax{\vec{a}_1},\lax{\vec{a}_2},\ldots,\lax{\vec{a}_k}\}\]
are also simplices of additive type.  Similarly, the sets
\[\{\lax{\pm \vec{a}_1 \pm \vec{a}_2},\lax{\vec{a}_1},\lax{\vec{a}_2},\ldots,\lax{\vec{a}_k}\} \quad \text{and} \quad 
\{\lax{\pm \vec{a}_1 \pm \vec{a}_2 \pm \vec{a}_3},\lax{\vec{a}_1},\lax{\vec{a}_2},\ldots,\lax{\vec{a}_k}\}\]
are also simplices of additive type for any choice of signs.  
The \emph{augmented complex of lax isotropic bases}, denoted $\PartialIBasesEx_g(R)$, is the simplicial complex whose
simplices are the standard simplices
and the simplices of additive and intersection type.

\p{The presentation of \boldmath$\Sp_{2g}(\Z)[2]$} The main result of this section is the following.

\begin{proposition}
\label{prop:sp presentation}
Let $g \geq 3$ and let $X_g = \PartialIBases_g(\Z)$ if $g \geq 4$ and $X_g = \PartialIBasesEx_g(\Z)$ if $g = 3$. We have
$$\Sp_{2g}(\Z)[2] \cong \left(\BigFreeProd_{\lax{\vec{v}} \in X_g^{(0)}} \Sp_{2g}(\Z)[2]_{\lax{\vec{v}}}\right)\Big/R$$
where $R$ is the normal closure of the edge and conjugation relators.
\end{proposition}

Proposition~\ref{prop:sp presentation} is a direct consequence of Theorem~\ref{theorem:presentation} and the following three 
propositions; together these propositions establish the conditions of Theorem~\ref{theorem:presentation} for the action of $\Sp_{2g}(\Z)[2]$ on $\PartialIBases_g(\Z)$ if $g \geq 4$ and for the action of $\Sp_6(\Z)[2]$ on $\PartialIBasesEx_3(\Z)$.

\begin{proposition}
\label{proposition:quotient}
Fix some $g \geq 1$.
\begin{enumerate}\setlength{\itemsep}{-2ex}
\item The group $\Sp_{2g}(\Z)[2]$ acts without rotations on $\PartialIBases_g(\Z)$ and we have
an isomorphism of CW complexes:
$\PartialIBases_g(\Z) / \Sp_{2g}(\Z)[2] \cong \PartialIBases_g(\Z/2)$.
\item The group $\Sp_{2g}(\Z)[2]$ acts without rotations on $\PartialIBasesEx_g(\Z)$ and we have
an isomorphism of CW complexes:
$\PartialIBasesEx_g(\Z) / \Sp_{2g}(\Z)[2] \cong \PartialIBasesEx_g(\Z/2)$.
\end{enumerate}
\end{proposition}

\begin{remark}
In the above proposition, recall that if a group $G$ acts without rotations on a simplicial complex $X$,
then $X/G$ is a CW complex in a natural way, but is not necessarily a simplicial complex; for instance, consider
the usual action of $\Z$ on $\R$, where $\R$ is triangulated with vertex set $\Z$.  
The above proposition says that this kind of pathology does not happen for the actions of
$\Sp_{2g}(\Z)[2]$ on $\PartialIBases_g(\Z)$ and on $\PartialIBasesEx_g(\Z)$.
\end{remark}

\begin{proposition}
\label{proposition:partialbasesconnected1}
If $R$ is either $\Z$ or a field, then the complex
$\PartialIBases_g(R)$ is homotopy equivalent to a wedge of $(g-1)$-spheres.  In particular, the complexes $\PartialIBases_g(\Z)$ and $\PartialIBases_g(\Z/2)$ are both 2-connected for $g \geq 4$.
\end{proposition}

\begin{proposition}
\label{proposition:partialbasesconnected2}
The complexes $\PartialIBasesEx_3(\Z)$ and $\PartialIBasesEx_3(\Z/2)$ are $1$-connected and $2$-connected, respectively.
\end{proposition}

Propositions~\ref{proposition:quotient}, \ref{proposition:partialbasesconnected1}, and~\ref{proposition:partialbasesconnected2} are proved in Sections~\ref{section:action and connectivity}, \ref{section:simplicialcomplexes}, and \ref{section:partialbasesexconnectedproof}, respectively.


\subsection{The quotient of the complex of lax isotropic bases}
\label{section:action and connectivity}

In this section, we prove Proposition~\ref{proposition:quotient}, which describes 
the restriction to $\Sp_{2g}(\Z)[2]$ of the actions of $\Sp_{2g}(\Z)$ on $\PartialIBases_g(\Z)$ 
and on $\PartialIBasesEx_g(\Z)$.   
Let $r : \Z^{2g} \rightarrow (\Z/2)^{2g}$ be the standard projection.  
Also observe that in $(\Z/2)^{2g}$ there is no difference between a vector and a lax vector.  To simplify our 
notation, we will write the vertices of $\PartialIBases_g(\Z/2)$ and $\PartialIBasesEx_g(\Z/2)$ simply as vectors.
Observe
that for a lax vector $\lax{\vec{v}}$ of $\Z^{2g}$,
the vector $r(\lax{\vec{v}}) \in (\Z/2)^{2g}$ is well defined since $r(\vec v) = r(-\vec v)$.

The proof of Proposition \ref{proposition:quotient} 
has three ingredients.  The first is Corollary \ref{corollary:norotations} below,
which says that the actions in question are without rotations.  We require a lemma. 

\begin{lemma}
\label{lemma:simplexproject}
Let $\sigma=\{\lax{\vec{v}_0},\ldots,\lax{\vec{v}_k}\}$ be a $k$-simplex of $\PartialIBasesEx_g(\Z)$.  
Then the set of vectors $\{r(\lax{\vec{v}_0}),\ldots,r(\lax{\vec{v}_k})\}$ forms a $k$-simplex of 
$\PartialIBasesEx_g(\Z/2)$ of the same type as $\sigma$.
\end{lemma}

\begin{proof}

Since $r$ preserves the algebraic intersection pairing modulo $2$, 
it follows that $r$ takes each symplectic basis for $\Z^{2g}$ to a symplectic basis for $(\Z/2)^{2g}$.  
If $\sigma$ is a standard simplex or a simplex of intersection type, it follows that 
$\{r(\lax{\vec{v}_0}),\ldots,r(\lax{\vec{v}_k})\}$ forms a $k$-simplex of $\PartialIBasesEx_g(\Z/2)$ 
of the same type as $\sigma$.  
If $\sigma$ is of additive type, then up to reindexing and changing the signs of the $\vec{v}_i$
we can assume that $\{\lax{\vec{v}_1},\ldots,\lax{\vec{v}_k}\}$ is a standard simplex and that
either $\vec{v}_0 = \vec{v}_1 + \vec{v}_2$ or $\vec{v}_0 = \vec{v}_1+\vec{v}_2+\vec{v}_3$.
As above, $\{r(\vec{v}_1),\ldots,r(\vec{v}_k)\}$ is a standard $(k-1)$-simplex of $\PartialIBasesEx_g(\Z/2)$.  
Since $r$ is linear, it follows that $\{r(\vec{v}_0),\ldots,r(\vec{v}_k)\}$ is a 
simplex of additive type in $\PartialIBasesEx_g(\Z/2)$.
\end{proof}

It follows immediately from Lemma~\ref{lemma:simplexproject} that $r$ induces simplicial maps
\[  \zeta : \PartialIBases_g(\Z) \to \PartialIBases_g(\Z/2) \quad  \text{and} \quad \widehat{\zeta} : \PartialIBasesEx_g(\Z) \to \PartialIBasesEx_g(\Z/2) \]
and that both maps take $k$-simplices onto $k$-simplices.

\begin{corollary}
\label{corollary:norotations}
The actions of $\Sp_{2g}(\Z)[2]$ on $\PartialIBases_g(\Z)$ and $\PartialIBasesEx_g(\Z)$ are without rotations.  
\end{corollary}

\begin{proof}

Since $\zeta$ and $\widehat{\zeta}$ are $\Sp_{2g}(\Z)[2]$-invariant, Lemma \ref{lemma:simplexproject} implies
that the vertices of a simplex of $\PartialIBases_g(\Z)$ lie in different
$\Sp_{2g}(\Z)[2]$-orbits, and similarly for $\PartialIBasesEx_g(\Z)$.
\end{proof}

Our second ingredient is Corollary \ref{corollary:liftsimplex} below, which says that the images of $\zeta$
and $\widehat{\zeta}$ contain every simplex of $\PartialIBases_g(\Z/2)$ and $\PartialIBasesEx_g(\Z/2)$, respectively.
This requires the following lemma,
which follows easily from a classical theorem of Newman--Smart \cite[Theorem 1]{NewmanSmart} that 
says that the map $\Sp_{2g}(\Z) \rightarrow \Sp_{2g}(\Z/2)$ is surjective.

\begin{lemma}
\label{lemma:realizepartialbasis}
Let $(\vec{\alpha}_1,\ldots,\vec{\alpha}_k;\vec{\beta}_1,\ldots,\vec{\beta}_{\ell})$ be a partial
symplectic basis for $(\Z/2)^{2g}$.  Then there exists a partial symplectic basis
$(\vec{a}_1,\ldots,\vec{a}_k;\vec{b}_1,\ldots,\vec{b}_{\ell})$ for $\Z^{2g}$ such that
$r(\vec{a}_i) = \vec{\alpha}_i$ and $r(\vec{b}_j) = \vec{\beta}_j$ for $1 \leq i \leq k$
and $1 \leq j \leq \ell$.
\end{lemma}

\begin{corollary}
\label{corollary:liftsimplex}
Let $\sigma$ be a simplex of $\PartialIBasesEx_g(\Z/2)$.
Then there exists some simplex $\widetilde{\sigma}$ of $\PartialIBasesEx_g(\Z)$ such that
$\widehat{\zeta}(\widetilde{\sigma}) = \sigma$.  The analogous statement holds for simplices of $\PartialIBases_g(\Z/2)$.
\end{corollary}

\begin{proof}

The corollary follows from Lemma \ref{lemma:realizepartialbasis} if $\sigma$ is a standard simplex or a simplex of intersection type (in particular the statement for $\PartialIBases_g(\Z/2)$ follows from this).  If $\sigma$ is of additive
type, then write $\sigma = \{\vec{v},\vec{\alpha}_1,\ldots,\vec{\alpha}_k\}$ with
$(\vec{\alpha}_1,\ldots,\vec{\alpha}_k;)$ a partial symplectic basis for $(\Z/2)^{2g}$ and
$\vec{v} = \sum_{i=1}^h \vec{\alpha}_i$ for some $h \in \{2,3\}$.  By Lemma
\ref{lemma:realizepartialbasis} there is a partial symplectic basis
$(\vec{a}_1,\ldots,\vec{a}_k;)$ for $\Z^{2g}$ with $r(\vec{a}_i) = \vec{\alpha}_i$ for
$1 \leq i \leq k$.  Setting $\vec{w} = \sum_{i=1}^h \vec{a}_i$ and 
$\widetilde{\sigma} = \{\vec{w},\vec{a}_1,\ldots,\vec{a}_k\}$, the set $\widetilde{\sigma}$
is a simplex of $\PartialIBasesEx_g(\Z)$ of additive type such that $\widehat{\zeta}(\widetilde{\sigma}) = \sigma$.
\end{proof}

Our third ingredient is Corollary \ref{corollary:movesimplices} below, which shows that two cells
of $\PartialIBases_g(\Z)$ that map to the same simplex of $\PartialIBases_g(\Z/2)$ differ
by an element of $\Sp_{2g}(\Z)[2]$, and similarly for $\PartialIBasesEx_g(\Z)$.  This requires
the following lemma.  For $\vec{v}_1,\ldots,\vec{v}_k \in R^{2g}$, let
$\Sp_{2g}(R,\vec{v}_1,\ldots,\vec{v}_k)$ denote $\{M \in \Sp_{2g}(R) \text{ $|$ }M(\vec{v}_i) = \vec{v}_i \text{ for $1 \leq i \leq k$}\}$.

\begin{lemma}
\label{lemma:stabilizersurject}
Let $(\vec{a}_1,\ldots,\vec{a}_k;\vec{b}_1,\ldots,\vec{b}_{\ell})$ be a partial
symplectic basis for $\Z^{2g}$.  Set $\vec{\alpha}_i = r(\vec{a}_i)$ and $\vec{\beta}_j = r(\vec{b}_j)$
for $1 \leq i \leq k$ and $1 \leq j \leq \ell$.  Then the natural map
\[\psi : \Sp_{2g}(\Z,\vec{a}_1,\ldots,\vec{a}_k,\vec{b}_1,\ldots,\vec{b}_{\ell}) \longrightarrow \Sp_{2g}(\Z/2,\vec{\alpha}_1,\ldots,\vec{\alpha}_k,\vec{\beta}_1,\ldots,\vec{\beta}_{\ell})\]
is surjective.
\end{lemma}

\begin{proof}

Because $\Sp_{2g}(\Z,\vec{v}) = \Sp_{2g}(\Z,-\vec{v})$, it is possible to replace each $\vec{a}_i$ with $\vec{b}_i$ and each $\vec{b}_i$ with $-\vec{a}_i$.  Therefore, we may assume without loss of generality that $k \geq \ell$.  Next,
let $V \cong \Z^{2(g-\ell)}$ be the orthogonal complement of the symplectic submodule
$\langle \vec{a}_1,\ldots,\vec{a}_\ell,\vec{b}_1,\ldots,\vec{b}_{\ell}\rangle$.  Defining 
$\Sp(V,\vec{a}_{\ell+1},\ldots,\vec{a}_k)$ in the obvious way, we then have
\[\Sp_{2g}(\Z,\vec{a}_1,\ldots,\vec{a}_k,\vec{b}_1,\ldots,\vec{b}_{\ell}) \cong \Sp(V,\vec{a}_{\ell+1},\ldots,\vec{a}_k).\]
A similar isomorphism holds for $\Sp_{2g}(\Z/2,\vec{\alpha}_1,\ldots,\vec{\alpha}_k,\vec{\beta}_1,\ldots,\vec{\beta}_{\ell})$.
Using this, we can reduce to the case $\ell = 0$.

We proceed by induction on $k$.  The base case
$k=0$ asserts that the map $\Sp_{2g}(\Z) \rightarrow \Sp_{2g}(\Z/2)$ is surjective, which is the aforementioned theorem of Newman--Smart. 
Assume now that $k \geq 1$.  Complete the partial symplectic basis
$(\vec{a}_1,\ldots,\vec{a}_k;)$ to a symplectic basis $(\vec{a}_1,\ldots,\vec{a}_g;\vec{b}_1,\ldots,\vec{b}_g)$
for $\Z^{2g}$ and let $\vec{\alpha}_i = r(\vec{a}_i)$ and $\vec{\beta}_j = r(\vec{b}_j)$ for
$k+1 \leq i \leq g$ and $1 \leq j \leq g$, so 
$(\vec{\alpha}_1,\ldots,\vec{\alpha}_g;\vec{\beta}_1,\ldots\vec{\beta}_g)$ is a symplectic basis for $(\Z/2)^{2g}$.  

We will regard $\Z^{2(g-1)}$ as the $\Z$-submodule $\langle \vec{a}_2,\vec{b}_2,\ldots,\vec{a}_g,\vec{b}_g \rangle$ of $\Z^{2g}$.  We can then identify 
$\Sp_{2(g-1)}(\Z,\vec{a}_2,\ldots,\vec{a}_k)$ with $\Sp_{2g}(\Z,\vec{a}_1,\vec{b}_1,\vec{a}_2,\ldots,\vec{a}_k)$,
and hence as a subgroup of $\Sp_{2g}(\Z,\vec{a}_1,\ldots,\vec{a}_k)$.
We define a surjective homomorphism
\[\rho : \Sp_{2g}(\Z,\vec{a}_1,\ldots,\vec{a}_k) \longrightarrow \Sp_{2(g-1)}(\Z,\vec{a}_2,\ldots,\vec{a}_k)\]
as follows.  Consider $M \in \Sp_{2g}(\Z,\vec{a}_1,\ldots,\vec{a}_k)$ and $\vec{v} \in \Z^{2(g-1)}$.  We can write
$M(\vec{v}) = c \vec{a}_1 + d \vec{b}_1 + \vec{w}$ for some $c,d \in \Z$ and $\vec{w} \in \Z^{2(g-1)}$.  Since
$\hat \imath(\vec{v},\vec{a}_1) = 0$, it follows that $d=0$.  We then define $\rho(M)(\vec{v}) = \vec{w}$.  Using the fact that $M(\vec{a}_1) = \vec{a}_1$, it
is easy to check that $\rho$ is a homomorphism.
Set $K_{\Z} = \ker(\rho)$, so
\[K_{\Z} = \left\{\text{$M \in \Sp_{2g}(\Z,\vec{a}_1,\ldots,\vec{a}_k)$ $|$ $M(\vec{v}) = \vec{v} + c \vec{a}_1$ \textrm{with} $c \in \Z$ for all $\vec{v} \in \Z^{2(g-1)}$}\right\}.\]
The surjection $\rho$ splits via the inclusion 
$\Sp_{2(g-1)}(\Z,\vec{a}_2,\ldots,\vec{a}_k) \hookrightarrow \Sp_{2g}(\Z,\vec{a}_1,\ldots,\vec{a}_k)$, so
\[ 
\Sp_{2g}(\Z,\vec{a}_1,\ldots,\vec{a}_k) \cong K_{\Z} \rtimes \Sp_{2(g-1)}(\Z,\vec{a}_2,\ldots,\vec{a}_k).
\]
Similarly regarding $(\Z/2)^{2(g-1)}$ as $\langle \vec{\alpha}_2,\vec{\beta}_2,\ldots,\vec{\alpha}_g,\vec{\beta}_g \rangle$,
we obtain a decomposition
\[
\Sp_{2g}(\Z/2,\vec{\alpha}_1,\ldots,\vec{\alpha}_k) \cong K_{2} \rtimes \Sp_{2(g-1)}(\Z/2,\vec{\alpha}_2,\ldots,\vec{\alpha}_k)
\]
with
\[K_2 = \left\{\text{$N \in \Sp_{2g}(\Z/2,\vec{\alpha}_1,\ldots,\vec{\alpha}_k)$ $|$ $N(\vec{v}) = \vec{v} + c \vec{\alpha}_1$ \textrm{with} $c \in \Z/2$ for all $\vec{v} \in (\Z/2)^{2(g-1)}$}\right\}.\]
The projection $\psi$ is compatible with the given decompositions of $\Sp_{2g}(\Z,\vec{a}_1,\ldots,\vec{a}_k)$ and of $\Sp_{2g}(\Z/2,\vec{\alpha}_1,\ldots,\vec{\alpha}_k)$ (the key point is that $\vec{\beta}_1 = r(\vec{b}_1)$).
Since $\Sp_{2(g-1)}(\Z,\vec{a}_2,\ldots,\vec{a}_k) \rightarrow \Sp_{2(g-1)}(\Z/2,\vec{\alpha}_2,\ldots,\vec{\alpha}_k)$
is surjective by induction, we are reduced to showing that $\psi|_{K_{\Z}} : K_{\Z} \rightarrow K_{2}$ is surjective.

Consider $N \in K_{2}$.  For $1 \leq i \leq g$, let $c_i \in \Z/2$ and $d_i \in \Z/2$ be the
$\vec{\alpha}_1$-components of $N(\vec{\alpha}_i)$ and $N(\vec{\beta}_i)$, respectively, so $c_1 = 1$ and $c_2 = \cdots = c_k = 0$.  Since $N$ fixes $\vec{\alpha}_1$, the $\vec{\beta}_1$-component of $N(\vec{\beta}_1)$ is $1$.  Similarly, using the fact that $N$ preserves $\hat \imath(\vec \beta_1,\vec \alpha_i)=0$ we conclude that the $\vec \beta_i$-component of $N(\vec{\beta}_1)$ is $-c_i$ for $2 \leq i \leq g$ and using the fact that $N$ preserves $\hat \imath(\vec \beta_1,\vec \beta_i)=0$ we conclude that the $\vec \alpha_i$-component of $N(\vec{\beta}_1)$ is $d_i$  for $2 \leq i \leq g$:
\[N(\vec{\beta}_1) = d_1 \vec{\alpha}_1 + \vec{\beta}_1 + d_2 \vec{\alpha}_2 - c_2 \vec{\beta}_2 + \cdots + d_g \vec{\alpha_g} - c_g \vec{\beta_g}.\]
For $1 \leq i \leq g$, let $\tilde{c}_i \in \Z$ and $\tilde{d}_i \in \Z$ be lifts
of $c_i \in \Z/2$ and $d_i \in \Z/2$, respectively.  Choose them such that $\tilde{c}_1 = 1$ and 
$\tilde{c}_2 = \cdots = \tilde{c}_k = 0$.
We can then define a $\Z$-linear map $M : \Z^{2g} \rightarrow \Z^{2g}$ via the formulas
\begin{align*}
M(\vec{a}_1) = \vec{a}_1, \quad M(\vec{b}_1) = \tilde{d}_1 \vec{a}_1 + \vec{b}_1 + \tilde{d}_2 \vec{a}_2 - \tilde{c}_2 \vec{b}_2 + \cdots + \tilde{d}_g \vec{a}_g - \tilde{c}_g \vec{b}_g, \\
M(\vec{a}_i) = \vec{a}_i + \tilde{c}_i \vec{a}_1, \quad \text{and} \quad M(\vec{b}_i) = \vec{b}_i + \tilde{d}_i \vec{a}_1 \quad \quad (2 \leq i \leq g).
\end{align*}
It is clear that $M \in K_{\Z}$ and $\psi(M) = N$, as desired.
\end{proof}

\begin{corollary}
\label{corollary:movesimplices}
Let $\widetilde{\sigma}_1$ and $\widetilde{\sigma}_2$ be simplices of $\PartialIBasesEx_g(\Z)$ with 
$\widehat{\zeta}(\widetilde{\sigma}_1) = \widehat{\zeta}(\widetilde{\sigma}_2)$.
Then there exists some $M \in \Sp_{2g}(\Z)[2]$ with $M(\widetilde{\sigma}_1) = \widetilde{\sigma}_2$.  The analogous statement holds for $\PartialIBases_g(\Z)$.
\end{corollary}

\begin{proof}

Lemma \ref{lemma:simplexproject} implies $\widetilde{\sigma}_1$ and $\widetilde{\sigma}_2$ are simplices
of the same type and dimension.  We deal with the three types in turn.  Observe that Case 1 below is sufficient to deal with $\PartialIBases_g(\Z)$.

\medskip \hspace*{3ex}\emph{Case 1.}
The $\widetilde{\sigma}_i$ are standard simplices.

\medskip

\noindent
Let $\psi : \Sp_{2g}(\Z) \rightarrow \Sp_{2g}(\Z/2)$ be the projection.  Write
$\widetilde{\sigma}_1 = \{\lax{\vec{a}_1},\ldots,\lax{\vec{a}_k}\}$ and $\widetilde{\sigma}_2 = \{\lax{\vec{a}_1'},\ldots,\lax{\vec{a}_k'}\}$,
where both $(\vec{a}_1,\ldots,\vec{a}_k;)$ and $(\vec{a}_1',\ldots,\vec{a}_k';)$ are partial symplectic
bases and $r(\vec{a}_i) = r(\vec{a}_i')$ for $1 \leq i \leq k$.  Extend these partial symplectic bases
to symplectic bases $(\vec{a}_1,\ldots,\vec{a}_g;\vec{b}_1,\ldots,\vec{b}_g)$ and
$(\vec{a}_1',\ldots,\vec{a}_g' ; \vec{b}_1',\ldots,\vec{b}_g')$.  There exists $M_1 \in \Sp_{2g}(\Z)$
such that $M_1(\vec{a}_i) = \vec{a}_i'$ and $M_1(\vec{b}_i) = \vec{b}_i'$ for $1 \leq i \leq g$, so
$M_1(\widetilde{\sigma}_1) = \widetilde{\sigma}_2$.  Set $\vec{v}_i = r(\vec{a}_i)$ for
$1 \leq i \leq k$, so
$\psi(M_1) \in \Sp_{2g}(\Z/2,\vec{v}_1,\ldots,\vec{v}_k)$.  By Lemma~\ref{lemma:stabilizersurject}, 
we can find some $M_2 \in \Sp_{2g}(\Z,\vec{a}_1,\ldots,\vec{a}_k)$ such that $\psi(M_2) = \psi(M_1)$.  Setting
$M = M_1 M_2^{-1}$, we have $M(\widetilde{\sigma}_1) = \widetilde{\sigma}_2$ and $M \in \ker(\psi) = \Sp_{2g}(\Z)[2]$. 

\medskip \hspace*{3ex}\emph{Case 2.}
The $\widetilde{\sigma}_i$ are simplices of intersection type.

\medskip

\noindent
We write
$\widetilde{\sigma}_1 = \{\lax{\vec{a}_1},\ldots,\lax{\vec{a}_k},\lax{\vec{b}_1}\}$ and $\widetilde{\sigma}_2 = \{\lax{\vec{a}_1'},\ldots,\lax{\vec{a}_k'},\lax{\vec{b}_1'}\}$,
where both $(\vec{a}_1,\ldots,\vec{a}_k;\vec{b}_1)$ and $(\vec{a}_1',\ldots,\vec{a}_k';\vec{b}_1')$ 
are partial symplectic bases.  The sets $\{r(\vec{a}_1),r(\vec{b}_1)\}$ and $\{r(\vec{a}_1'),r(\vec{b}_1')\}$ are equal since these are the unique pairs of elements with nontrivial pairing under the symplectic form.  If necessary, we replace $(\vec{a}_1,\vec{b}_1)$ with $(\vec{b}_1,-\vec{a}_1$) in order to ensure that $r(\vec{a}_1) = r(\vec{a}_1')$ and $r(\vec{b}_1) = r(\vec{b}_1')$; this does not change the fact that $(\vec{a}_1,\ldots,\vec{a}_k;\vec{b}_1)$ is a partial
symplectic basis.  We can further reorder the vertices of $\widetilde{\sigma}_1$ so that $r(\vec{a}_i) = r(\vec{a}_i')$ for $2 \leq i \leq k$.    The desired $M \in \Sp_{2g}(\Z)[2]$ can now be found 
exactly as in Case 1.
 
\medskip \hspace*{3ex}\emph{Case 3.}
The $\widetilde{\sigma}_i$ are simplices of additive type.

\medskip

\noindent
We can write
$\widetilde{\sigma}_1 = \{\lax{\vec{v}},\lax{\vec{a}_1},\ldots,\lax{\vec{a}_k}\}$ and $\widetilde{\sigma}_2 = \{\lax{\vec{v}'},\lax{\vec{a}_1'},\ldots,\lax{\vec{a}_k'}\}$,
where both $(\vec{a}_1,\ldots,\vec{a}_k;)$ and $(\vec{a}_1',\ldots,\vec{a}_k';)$
are partial symplectic bases and where $\vec{v} = \sum_{i=1}^{h} \vec{a}_i$ and $\vec{v}' = \sum_{j=1}^{\ell} \vec{a}_j'$
for some $h,\ell \in \{2,3\}$.
Among nonempty subsets of $\{\vec{v},\vec{a}_1,\ldots,\vec{a}_k\}$ and $\{\vec{v}',\vec{a}_1',\ldots,\vec{a}_k'\}$, the
sets $\{\vec{v},\vec{a}_1,\ldots,\vec{a}_h\}$ and $\{\vec{v}',\vec{a}_1',\ldots,\vec{a}_{\ell}'\}$ 
are the unique minimal linearly dependent sets, respectively.  They both must map to the unique
minimal linearly dependent set among nonempty subsets of 
$\{r(\vec{v}),r(\vec{a}_1),\ldots,r(\vec{a}_k)\} = \{r(\vec{v}'),r(\vec{a}'_1),\ldots,r(\vec{a}'_k)\}$.  We conclude
that the unordered sets
$\{r(\vec{v}),r(\vec{a}_1),\ldots,r(\vec{a}_h)\}$ and 
$\{r(\vec{v}'),r(\vec{a}_1'),\ldots,r(\vec{a}_{\ell}')\}$ are equal; in particular, $h = \ell$.  Reordering
the elements of $\{\vec{v}',\vec{a}_1',\ldots,\vec{a}_k'\}$ we can assume that $r(\vec{v}) = r(\vec{v}')$ and that
$r(\vec{a}_i) = r(\vec{a}_i')$ for $1 \leq i \leq k$; however, after doing this we can only
assume that $\vec{v}' = \sum_{i=1}^{h} e_i \vec{a}_i'$ for some choices of $e_i = \pm 1$ (here we are using the
fact that there is a linear dependence among $\{\vec{v}',\vec{a}_1',\ldots,\vec{a}_{h}'\}$ all of whose
coefficients are $\pm 1$).  Now replace 
$\vec{a}_i'$ with $e_i\vec{a}_i'$ for $1 \leq i \leq h$; this does not change the $\lax{\vec{a}_i'}$ or
$r(\vec{a}_i')$, but we now again have $\vec{v}' = \sum_{i=1}^{h} \vec{a}_i'$.  By
Case 1, there exists some $M \in \Sp_{2g}(\Z)[2]$ such that $M(\vec{a}_i) = \vec{a}_i'$ for
$1 \leq i \leq k$.  It follows that
\[M(\vec{v}) = \sum_{i=1}^h M(\vec{a}_i) = \sum_{i=1}^{h} \vec{a}_i' = \vec{v}', \]
and so $M(\widetilde{\sigma}_1) = \widetilde{\sigma}_2$, as desired.
\end{proof}

\begin{proof}[{Proof of Proposition \ref{proposition:quotient}}]

We will deal with $\PartialIBasesEx_g(\Z)$; the other case is similar. 
Corollary~\ref{corollary:norotations} says that $\Sp_{2g}(\Z)[2]$ acts without rotations on 
$\PartialIBasesEx_g(\Z)$.  We must identify 
the quotient.  Lemma \ref{lemma:simplexproject} gives an $\Sp_{2g}(\Z)[2]$-invariant projection map
$\widehat{\zeta} : \PartialIBasesEx_g(\Z) \rightarrow \PartialIBasesEx_g(\Z/2)$.
This induces a map $\widehat{\eta} : \PartialIBasesEx_g(\Z)/\Sp_{2g}(\Z)[2] \rightarrow \PartialIBasesEx_g(\Z/2)$.
Giving $\PartialIBasesEx_g(\Z) / \Sp_{2g}(\Z)[2]$ its natural CW complex structure (see the remark
after the statement of Proposition \ref{proposition:quotient}), Lemma~\ref{lemma:simplexproject} implies that
$\PartialIBasesEx_g(\Z) / \Sp_{2g}(\Z)[2]$ is a regular CW complex (i.e. attaching maps are injective) and that for each cell $\sigma$ of $\PartialIBasesEx_g(\Z)/\Sp_{2g}(\Z)[2]$ the map $\widehat{\eta}$ restricts
to a homeomorphism from $\sigma$ onto a simplex of $\PartialIBasesEx_g(\Z/2)$.
Corollary \ref{corollary:liftsimplex} implies that the image of $\widehat{\eta}$ contains every simplex of
$\PartialIBasesEx_g(\Z/2)$, and Corollary \ref{corollary:movesimplices} implies that distinct cells of
$\PartialIBasesEx_g(\Z)/\Sp_{2g}(\Z)[2]$ are mapped to distinct simplices of $\PartialIBasesEx_g(\Z/2)$.
We conclude that $\widehat{\eta}$ is an isomorphism of CW complexes, as desired.
\end{proof}

\subsection{Connectivity of the complex of lax isotropic bases}
\label{section:simplicialcomplexes}

In this section, we prove Proposition~\ref{proposition:partialbasesconnected1}, which states that for $R$ either a field or $\Z$, the complex $\PartialIBases_g(R)$ is homotopy equivalent to a wedge of $(g-1)$-spheres.  The proof is similar to a proof of a
related result due to Charney; see \cite[Theorem 2.9]{CharneyVogtmann}.  Before we begin to prove Proposition~\ref{proposition:partialbasesconnected1}, we recall some basic generalities about posets.

\p{Posets}
Let $P$ be a poset.  Consider $p \in P$.  The {\em height} of $p$, denoted $\height(p)$, is the largest number
$k$ such that there exists a strictly increasing chain
$$p_0 < p_1 < \cdots < p_k = p.$$
We will denote by $P_{>p}$ the subposet of $P$ consisting of elements strictly greater than $p$.
Also, if $f : Q \rightarrow P$ is a poset map, then
$$f/p = \{\text{$q \in Q$ $|$ $f(q) \leq p$}\}.$$
Finally, the {\em geometric realization} of $P$, denoted $|P|$, is the simplicial 
complex whose vertices are elements
of $P$ and whose $k$-simplices are sets $\{p_0,\ldots,p_k\}$ of elements of $P$ satisfying
$$p_0 < p_1 < \cdots < p_k.$$
A key example is as follows.  Let $X$ be a simplicial complex.  Then the set $\Poset(X)$ of simplices of
$X$ forms a poset under inclusion and $|\Poset(X)|$ is the barycentric subdivision of $X$.  

We shall require the following version of Quillen's Theorem A \cite[Theorem 9.1]{QuillenPoset}.  In what follows, when we say that a poset has some topological property, we mean that its geometric realization has that property.

\begin{theorem}
\label{theorem:quillen1}
Let $f : Q \rightarrow P$ be a poset map.  For some $m$, assume that $P$ is homotopy 
equivalent to a wedge of $m$-spheres.  Also, for all $p \in P$ assume
that $P_{>p}$ is homotopy equivalent to a wedge of $(m-\height(p)-1)$-spheres
and that $f/p$ is homotopy equivalent to a wedge of $\height(p)$-spheres.
Then $Q$ is homotopy equivalent to a wedge of $m$-spheres.
\end{theorem}

In our application of Theorem~\ref{theorem:quillen1}, we will take $Q$ to be $\PartialIBases_g(R)$.  The roles of $P$ and $f/p$ will be played by the Tits building $\Building_g(R)$ and the complex of lax partial bases $\PartialBases_n(R)$, both to be defined momentarily.  Theorems~\ref{theorem:solomontits} and~\ref{theorem:maazen}
below give that $\PartialIBases_g(R)$ and $\Building_g(R)$ (and the natural map between them) satisfy the hypotheses of Theorem~\ref{theorem:quillen1} with $m=g-1$, and so we will conclude that $\PartialIBases_g(R)$ is a wedge of $(g-1)$-spheres, as desired.

\p{Buildings}
Let $\Field$ be a field and $\hat \imath$ the standard symplectic form on $\Field^{2g}$.  An isotropic subspace of $\Field^{2g}$ is a subspace on which $\hat \imath$ vanishes.  The Tits building $\Building_g(\Field)$ is the poset of nontrivial isotropic subspaces of $\Field^{2g}$.   
The key theorem about the topology of $\Building_g(\Field)$ is the Solomon--Tits theorem \cite[Theorem 4.73]{BrownAbramenko}.

\begin{theorem}[Solomon--Tits]
\label{theorem:solomontits}
If $\Field$ is a field, then $\Building_g(\Field)$ is homotopy equivalent to a wedge of $(g-1)$-spheres.
Also, for $V \in \Building_g(\Field)$ the poset $(\Building_g(\Field))_{>V}$ is homotopy equivalent to
a wedge of $(g-2-\height(V))$-spheres.
\end{theorem}

\p{Complexes of lax partial bases}
For $R$ equal to either $\Z$ or a field, 
let $\PartialBases_n(R)$ be the simplicial complex whose $k$-simplices are sets
$\{\lax{\vec{v}_0},\ldots,\lax{\vec{v}_k}\}$, where $\{\vec{v}_0,\ldots,\vec{v}_k\}$
is a set of elements of $R^n$ that forms a basis for a free
summand of $R^n$.  We then have the following theorem.

\begin{theorem}
\label{theorem:maazen}
If $R$ is either $\Z$ or a field, then $\PartialBases_n(R)$ is homotopy equivalent
to a wedge of $(n-1)$-spheres.
\end{theorem}

\begin{proof}

Let $\PartialBases_n'(R)$ be the simplicial complex whose $k$-simplices are
sets $\{\vec{v}_0,\ldots,\vec{v}_k\}$ of elements of $R^n$ that form a basis
for a free summand of $R^n$.  In his unpublished thesis \cite{MaazenThesis}, Maazen
proved that under our assumption that $R$ is either $\Z$ or a field, 
$\PartialBases_n'(R)$ is $(n-2)$-connected, and thus is homotopy equivalent to a wedge
of $(n-1)$-spheres.
For $R = \Z$, there is a published account of Maazen's theorem in
\cite[Proof of Theorem B, Step 2]{DayPutmanComplex}.  This proof can be easily
adapted to work for any Euclidean domain $R$ by replacing 
all invocations of the absolute value function $|\cdot|$ on $\Z$ with the
Euclidean function on $R$; in particular, the proof works for a field.  There
is a natural map $\rho : \PartialBases_n'(R) \rightarrow \PartialBases_n(R)$.  Let
$\eta^{(0)} : (\PartialBases_n(R))^{(0)} \rightarrow (\PartialBases_n'(R))^{(0)}$ be
an arbitrary right inverse for $\rho|_{(\PartialBases_n'(R))^{(0)}}$.  Clearly
$\eta^{(0)}$ extends to a simplicial map 
$\eta : \PartialBases_n(R) \rightarrow \PartialBases_n'(R)$ satisfying
$\rho \circ \eta = \text{id}$.  It follows that $\rho$ induces a surjection on
all homotopy groups, so $\PartialBases_n(R)$ is also $(n-2)$-connected and thus
homotopy equivalent to a wedge of $(n-1)$-spheres.
\end{proof}

\p{Connectivity of \boldmath$\PartialIBases_g(R)$}
We are almost ready to prove Proposition~\ref{proposition:partialbasesconnected1}, which asserts that $\PartialIBases_g(R)$ is homotopy equivalent to a wedge of $(g-1)$-spheres for $R$ equal
to $\Z$ or a field.  We first need the following classical lemma.

\begin{lemma}
\label{lemma:subspaces}
Let $V \subset \Q^n$ be a subspace.  Then $V \cap \Z^n$ is a direct summand of $\Z^n$.
\end{lemma}

\begin{proof}

Write $V = \ker(T)$ for some linear map $T : \Q^n \rightarrow \Q^n$.  Then
$V \cap \Z^n = \ker(T|_{\Z^n})$.  Moreover, since $T(\Z^n) \subset \Q^n$ is a torsion-free $\Z$-module,
it must be a projective $\Z$-module.  This allows us to split the short exact sequence of $\Z$-modules
\[0 \longrightarrow V \cap \Z^n \longrightarrow \Z^n \stackrel{T|_{\Z^n}}{\longrightarrow} T(\Z^n) \longrightarrow 0. \qedhere\]
\end{proof}

\begin{proof}[{Proof of Proposition~\ref{proposition:partialbasesconnected1}}]

Let $\Field = R$ if $R$ is a field and $\Field = \Q$ if $R = \Z$.  Define a poset map
$$\spanmap : \Poset(\PartialIBases_g(R)) \rightarrow \Building_g(\Field)$$
by $\spanmap(\{\lax{\vec{v}_0},\ldots,\lax{\vec{v}_k}\}) = \text{span}(\vec{v}_0,\ldots,\vec{v}_k)$.  Consider
$V \in \Building_g(\Field)$, and set $d = \dim(V)$, so $\height(V) = d-1$.  The poset
$\spanmap/V$ is isomorphic to $\PartialBases_d(R)$ (this uses Lemma \ref{lemma:subspaces} if $R = \Z$), 
so Theorem~\ref{theorem:maazen} says that
$\spanmap/V$ is homotopy equivalent to a wedge of $\height(V)$-spheres.  
Theorem~\ref{theorem:solomontits} says that the remaining 
assumptions of Theorem~\ref{theorem:quillen1} are satisfied for $\spanmap$ with $m = g-1$.  The conclusion of this theorem is that $\Poset(\PartialIBases_g(R))$, hence $\PartialIBases_g(R)$, is a wedge of $(g-1)$-spheres, as desired.
\end{proof}

\subsection{Connectivity of \boldmath$\PartialIBasesEx_g(R)$}
\label{section:partialbasesexconnectedproof}

We now prove Proposition~\ref{proposition:partialbasesconnected2}, which asserts
that the complexes $\PartialIBasesEx_3(\Z)$ and $\PartialIBasesEx_3(\Z/2)$ are $1$-connected and $2$-connected, respectively.  The proof is more complicated than the one for Proposition~\ref{proposition:partialbasesconnected1}, and so we begin with an outline.

For $R$ either $\Z$ or a field, denote by $\PartialIBases_g^{\alpha}(R)$ the subcomplex of $\PartialIBasesEx_g(R)$ consisting of standard simplices and simplices of additive type.  The proof of Proposition~\ref{proposition:partialbasesconnected2} consists of three steps:
\begin{enumerate}\setlength{\itemsep}{0ex}
\item We show that the map $\pi_k(\PartialIBases_3^{\alpha}(R)) \to \pi_k(\PartialIBasesEx_3(R))$ is surjective for $k=1,2$.
\item We find an explicit generating set for $\pi_2(\PartialIBases_3^{\alpha}(\Z/2))$.  
\item We show that each generator of $\pi_2(\PartialIBases_3^{\alpha}(\Z/2))$ maps to zero in $\pi_2(\PartialIBasesEx_3(\Z/2))$.
\end{enumerate}
%
The 1-connectivity of $\PartialIBasesEx_3(\Z)$ and $\PartialIBasesEx_3(\Z/2)$ follow from the $k=1$ version of the first step, as $\PartialIBases_3(R)$ is $1$-connected (Proposition~\ref{proposition:partialbasesconnected1}) and $\PartialIBases_3(R)$ contains the entire 1-skeleton of $\PartialIBases_3^{\alpha}(R)$.  Together with the $k=2$ version of the first step, the latter two steps together imply that $\pi_2(\PartialIBasesEx_3(\Z/2))$ is trivial.  

\p{Pushing into \boldmath$\PartialIBases_3^{\alpha}(R)$}
As above, the first step of the proof is to show that $\PartialIBases_g^{\alpha}(R)$ carries all of $\pi_1(\PartialIBasesEx_g(R))$ and $\pi_2(\PartialIBasesEx_g(R))$.  We in fact have a more general statement.

\begin{lemma}
\label{lemma:push}
Let $R$ equal $\Z$ or a field.  For $1 \leq k \leq g-1$, the inclusion map 
$\PartialIBases_g^{\alpha}(R) \rightarrow \PartialIBasesEx_g(R)$ induces a surjection on $\pi_k$.
\end{lemma}

\begin{proof}

Let $S$ be a simplicial complex that is a combinatorial triangulation of a $k$-sphere (recall that a combinatorial triangulation of a manifold is a triangulation where the link of each $d$-simplex is a triangulation of a $(k-d-1)$-sphere) and let 
$f : S \rightarrow \PartialIBasesEx_g(R)$ be a simplicial map.  It is enough to homotope
$f$ such that its image lies in $\PartialIBases_g^{\alpha}(R)$.  If the image
of $f$ is not contained in $\PartialIBases_g^{\alpha}(R)$, then there exists some simplex
$\sigma$ of $S$ such that $f(\sigma)$ is a $1$-simplex $\{\lax{\vec{a}},\lax{\vec{b}}\}$ of
intersection type.  Choose $\sigma$ such that $d = \dim(\sigma)$ is maximal; since $f$ need
not be injective, we might have $d > 1$.  The link $\Link_S(\sigma)$ is homeomorphic
to a $(k-d-1)$-sphere.  Moreover, 
$$f(\Link_S(\sigma)) \subseteq \Link_{\PartialIBasesEx_g(R)} \{\lax{\vec{a}},\lax{\vec{b}}\}.$$
The key observation is that $\Link_{\PartialIBasesEx_g(R)} \{\lax{\vec{a}},\lax{\vec{b}}\}$
can only contain standard simplices, and moreover all of its vertices are lax vectors
$\lax{\vec{v}}$ such that $\hat \imath(\vec{a},\vec{v}) = \hat \imath(\vec{b},\vec{v}) = 0$.  
Indeed, if a simplex $\{\lax{\vec{w}_0},\ldots,\lax{\vec{w}_k}\}$ spans a simplex with the edge 
$\{\lax{\vec{a}},\lax{\vec{b}}\}$, then necessarily $\{\lax{\vec{a}},\lax{\vec{b}},\lax{\vec{w}_0},\ldots,\lax{\vec{w}_k}\}$ 
is a simplex of intersection type, which implies the observation.

The $\hat \imath$-orthogonal complement of $\text{span}(\vec{a},\vec{b})$ is isomorphic to a
$2(g-1)$-dimensional free symplectic module over $R$, so we deduce that
$$\Link_{\PartialIBasesEx_g(R)} \{\lax{\vec{a}},\lax{\vec{b}}\} \cong \PartialIBases_{g-1}(R).$$
Proposition~\ref{proposition:partialbasesconnected1} says that $\PartialIBases_{g-1}(R)$ is $(g-3)$-connected.
Since $k-d-1 \leq g-3$, we conclude that there exists a combinatorially triangulated $(k-d)$-ball $B$ with $\partial B = \Link_S(\sigma)$ and a simplicial map 
$$F : B \rightarrow \Link_{\PartialIBasesEx_g(R)} \{\lax{\vec{a}}, \lax{\vec{b}}\}$$
such that $F|_{\partial B} = f|_{\Link_S(\sigma)}$.  We can therefore homotope $f$ so as
to replace $f|_{\Star_S(\sigma)}$ with $F|_{B}$.  This eliminates $\sigma$ without
introducing any new $d$-dimensional simplices mapping to $1$-simplices of intersection type.  Doing
this repeatedly homotopes $f$ so that its image contains no simplices of intersection type.
\end{proof}

\p{Generators for \boldmath$\pi_2(\PartialIBases_3^{\alpha}(\Z/2))$}  Our next goal is to give generators for $\pi_2(\PartialIBases_3^{\alpha}(\Z/2))$.  This has two parts.  We first recall a known, explicit generating set for $\pi_{2}(\Building_g(\Z/2))$ (Theorem~\ref{theorem:solomontits2}), and then we show in Lemma~\ref{lemma:pi2 iso} that the map
\[\spanmap : \Poset(\PartialIBases_3^{\alpha}(\Z/2)) \rightarrow \Building_3(\Z/2)\]
given by $\spanmap(\{\lax{\vec{v}_0},\ldots,\lax{\vec{v}_k}\}) = \text{span}(\vec{v}_0,\ldots,\vec{v}_k)$ induces an isomorphism on the level of $\pi_2$; thus the generators for $\pi_{2}(\Building_g(\Z/2))$ give generators for $\pi_2(\PartialIBases_3^{\alpha}(\Z/2))$.  

Let $\Field$ be a field.  Recall that the Solomon--Tits theorem (Theorem~\ref{theorem:solomontits}) says that the Tits building 
$\Building_g(\Field)$ is homotopy equivalent to a wedge of $(g-1)$-spheres.  The next theorem  gives explicit generators for $\pi_{g-1}(\Building_g(\Field))$.   First, we require some setup.

Let $Y_g$ be the join of $g$ copies of $S^0$, so $Y_g \cong S^{g-1}$.  If $x_{i}$ and $y_{i}$ are the vertices of the $i$th copy of $S^0$ in $Y_g$, then the simplices of $Y_g$ are the nonempty
subsets $\sigma \subset \{x_1,y_1,\ldots,x_{g},y_g\}$ such that $\sigma$
contains at most one of $x_{i}$ and $y_{i}$ for each $1 \leq i \leq g$.  Given a
symplectic basis $B = (\vec{a}_1,\ldots,\vec{a}_g;\vec{b}_1,\ldots,\vec{b}_g)$ for $\Field^{2g}$, we obtain a poset map
$\alpha_B : \Poset(Y_g) \rightarrow \Building_g(\Field)$ as follows.  Consider 
$$\sigma = \{x_{i_1},\ldots,x_{i_k},y_{j_1},\ldots,y_{j_{\ell}}\} \in \Poset(Y_g).$$
We then define 
$$\alpha_B(\sigma) = \text{span}(\vec{a}_{i_1},\ldots,\vec{a}_{i_k},\vec{b}_{j_1},\ldots,\vec{b}_{j_{\ell}}) \in \Building_g(\Field).$$
The resulting map $\alpha_B : Y_g \to \Building_g(\Field)$ is a $(g-1)$-sphere in $\Building_g(\Field)$.  
We have the following theorem \cite[Theorem 4.73]{BrownAbramenko}.

\begin{theorem}
\label{theorem:solomontits2}
The group $\pi_{g-1}(\Building_g(\Field))$ is generated by the set 
\[ \{\text{$[\alpha_B] \in \pi_{g-1}(\Building_g(\Field))$ $|$ $B$ a symplectic basis for $\Field^{2g}$}\}. \]
\end{theorem}

\Figure{figure:crosscomplex}{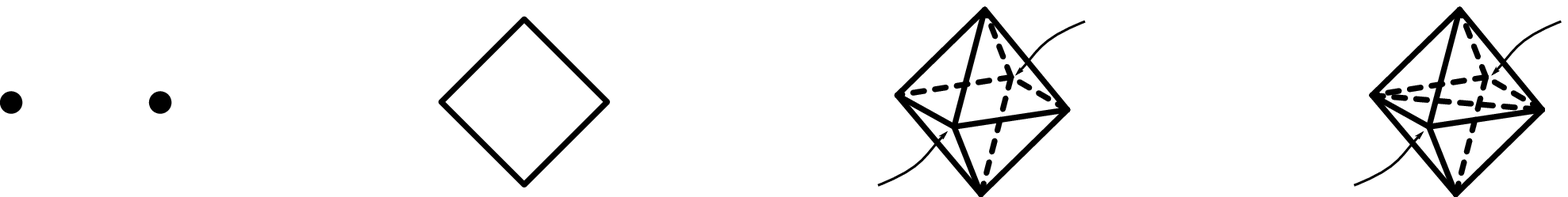}{The first three pictures depict the images of $Y_1$, $Y_2$, and $Y_3$ in $\PartialIBases_g(\Z/2)$.  The fourth picture indicates a nullhomotopy of $Y_3$ in $\PartialIBasesEx_g(\Z/2)$ using simplices of intersection type.}

Now, in the same way as we defined the $\alpha_B(\sigma)$, we can also define
\[ \widetilde{\alpha}_B : Y_g \rightarrow \PartialIBases_g(\Z/2) \]
via
\[ \widetilde\alpha_B(\sigma) = \{\vec{a}_{i_1},\ldots,\vec{a}_{i_k},\vec{b}_{j_1},\ldots,\vec{b}_{j_{\ell}}\} \in \Poset(\PartialIBases_g(\Z/2)); \]
see Figure~\ref{figure:crosscomplex} (recall that over $\Z/2$ lax vectors are the same as vectors).  We have
\[ \alpha_B = \spanmap \circ \widetilde \alpha_B. \]
We will show $\spanmap_{\ast} : \pi_2(\PartialIBases_g^{\alpha}(\Z/2)) \rightarrow \pi_2(\Building_g(\Z/2))$ is an isomorphism (Lemma~\ref{lemma:pi2 iso}), and hence the 
$\widetilde\alpha_B(\sigma)$ generate $\pi_2(\PartialIBases_g^{\alpha}(\Z/2))$ (Lemma~\ref{lemma:pi2 gens}).  The starting point here is another version of Quillen's Theorem A \cite[Theorem 2]{BjornerNerves}.

\begin{theorem}
\label{theorem:quillen2}
Let $Q$ and $P$ be connected posets and $f : Q \rightarrow P$ a poset map.  Assume
that $f/p$ is $m$-connected for all $p \in P$.  Then the induced map $f_{\ast} : \pi_k(Q) \rightarrow \pi_k(P)$
is an isomorphism for $1 \leq k \leq m$.
\end{theorem}

We will also need the following easy lemma.

\begin{lemma}
\label{lemma:setsofvec}
Any subset of $(\Z/2)^n\setminus \{0\}$ with cardinality at most 4 has one of the forms:
\begin{align*}
\{\}, \{\vec{v}_1\}, \{\vec{v}_1,\vec{v}_2\}, &\{\vec{v}_1,\vec{v}_2,\vec{v}_3\}, \{\vec{v}_1,\vec{v}_2,\vec{v}_1+\vec{v}_2\},  \\
 \{\vec{v}_1,\vec{v}_2,\vec{v}_3,\vec{v}_4\}, \{\vec{v}_1,\vec{v}_2,&\vec{v}_3,\vec{v}_1+\vec{v}_2\}, \{\vec{v}_1,\vec{v}_2,\vec{v}_3,\vec{v}_1+\vec{v}_2+\vec{v}_3\},
\end{align*}
where in each set the $\vec{v}_i$ are linearly independent vectors in
$(\Z/2)^n$.
\end{lemma}

\begin{lemma}
\label{lemma:pi2 iso}
The map $\spanmap_{\ast} : \pi_2(\PartialIBases_g^{\alpha}(\Z/2)) \rightarrow \pi_2(\Building_g(\Z/2))$ is an isomorphism
for all $g \geq 2$.
\end{lemma}

\begin{proof}
Consider $V \in \Building_g(\Z/2)$, and let $d = \dim(V)$.  The poset $\spanmap/V$ is isomorphic to the result $\PartialBases_d^\alpha(\Z/2)$ of attaching the analogues
of cells of additive type to $\PartialBases_d(\Z/2)$.  As vertices of $\PartialBases_d^\alpha(\Z/2)$ correspond to nonzero lax
vectors, Lemma~\ref{lemma:setsofvec}
implies that the $3$-skeleton of $\PartialBases_d^\alpha(\Z/2)$ contains all subsets of vertices of $\PartialBases_d^\alpha(\Z/2)$ of size at most $4$.  In
particular, as $\PartialBases_d^\alpha(\Z/2)$ is obviously nonempty, it is $2$-connected.  The lemma now follows from Theorem~\ref{theorem:quillen2}.
\end{proof}

We have the following immediate consequence of Theorem~\ref{theorem:solomontits2} and Lemma~\ref{lemma:pi2 iso}.

\begin{lemma}
\label{lemma:pi2 gens}
As $B$ ranges over all symplectic bases, the homotopy classes of the maps
$$Y_3 \stackrel{\widetilde{\alpha}_B}{\longrightarrow} \PartialIBases_3(\Z/2) \hookrightarrow \PartialIBases_3^{\alpha}(\Z/2)$$
generate $\pi_2(\PartialIBases_3^{\alpha}(\Z/2))$.
\end{lemma}

\begin{proof}[{Proof of Proposition~\ref{proposition:partialbasesconnected2}}]

We already explained how the 1-connectivity of $\PartialIBasesEx_3(\Z)$ and of $\PartialIBasesEx_3(\Z/2)$ follow from Lemma~\ref{lemma:push}.  It remains to show that $\PartialIBasesEx_3(\Z/2)$ is $2$-connected.  Lemma~\ref{lemma:push} says
that the inclusion map $\PartialIBases_3^{\alpha}(\Z/2) \hookrightarrow \PartialIBasesEx_3(\Z/2)$ induces a surjection on $\pi_2$.  We will prove that it induces the zero map as well.  

By Lemma~\ref{lemma:pi2 gens}, it suffices to show that for each symplectic basis 
$B$ of $(\Z/2)^{6}$, the map
$$Y_3 \stackrel{\widetilde{\alpha}_B}{\longrightarrow} \PartialIBases_3(\Z/2) \hookrightarrow \PartialIBases_3^{\alpha}(\Z/2) \hookrightarrow \PartialIBasesEx_3(\Z/2)$$
is nullhomotopic.  In the fourth picture in Figure~\ref{figure:crosscomplex}, we indicate an explicit nullhomotopy of $\widetilde{\alpha}_B(Y_3)$ in $\PartialIBasesEx_3(\Z/2)$.  Specifically, we realize $\widetilde{\alpha}_B(Y_3)$ as the boundary of a 3-ball formed by four simplices of intersection type: 
$\{\vec{a}_1,\vec{b}_1,\vec{a}_2,\vec{a}_3\}$, $\{\vec{a}_1,\vec{b}_1,\vec{b}_2,\vec{a}_3\}$, $\{\vec{a}_1,\vec{b}_1,\vec{a}_2,\vec{b}_3\}$, and $\{\vec{a}_1,\vec{b}_1,\vec{b}_2,\vec{b}_3\}$.   We conclude
that the inclusion $\PartialIBases_3^{\alpha}(\Z/2) \hookrightarrow \PartialIBasesEx_3(\Z/2)$ induces the zero map on $\pi_2$, as desired.
\end{proof}


\section{The symplectic group action}
\label{section:action}

We now discuss the second ingredient for the proof of Proposition~\ref{proposition:main}, namely, the group action of  $\Sp_{2g}(\Z)$ on $\Quotient_g$.  The group $\Sp_{2g}(\Z)$ acts on $\Sp_{2g}(\Z)[2]$ by conjugation.  We wish to lift this to an action of $\Sp_{2g}(\Z)$ on $\Quotient_g$ in a natural way.  

\subsection{Setup}
\label{section:action setup}

Our first task is to give a precise description of the action we would like to obtain (Proposition~\ref{proposition:spaction}).
Recall that $\Quotient_g = \PureBraid_{2g+1} / \Theta_{2g+1}$.  Define $\QuotientEx_g = \Braid_{2g+1} / \Theta_{2g+1}$,
and let
\begin{center}
\begin{tabular}{lcl}
$\rho : \PureBraid_{2g+1} \rightarrow \Quotient_g$  & and & $\widehat{\rho} : \Braid_{2g+1} \rightarrow \QuotientEx_{g}$\\
$\pi  : \Quotient_g \rightarrow \Sp_{2g}(\Z)[2]$    &     & $\widehat{\pi}  : \QuotientEx_{g} \rightarrow \Sp_{2g}(\Z)$\\
\end{tabular}
\end{center}
be the quotient maps.    The first two parts of Proposition~\ref{proposition:spaction} below
posit the existence of an action of $\Sp_{2g}(\Z)$ on the
$\Quotient_g$ that is natural with respect to the actions of $\Sp_{2g}(\Z)$ on $\Sp_{2g}(\Z)[2]$ and $\QuotientEx_{g}$ on $\Quotient_{g}$.

We will require our action to have one extra property, which requires some setup.  Let $c_{23}$ be the curve in $\D_{2g+1} \cong \Sigma_{g}^1 / \Hyper$ shown in Figure~\ref{figure:mainproof} and let $(\PureBraid_{2g+1})_{c_{23}}$ be its stabilizer.  Next, define
\[ \Omega_{23} = \rho\left((\PureBraid_{2g+1})_{c_{23}}\right) \subseteq \Quotient_g.\]
Finally, let $\lax{\vec{v}_{23}}$ be the lax vector of $H_1(\Sigma_g^1;\Z)$ represented by one component of the preimage of $c_{23}$ in $\Sigma_g^1$ and let $(\Sp_{2g}(\Z))_{\lax{\vec{v}_{23}}}$ denote the stabilizer. 

\Figure{figure:mainproof}{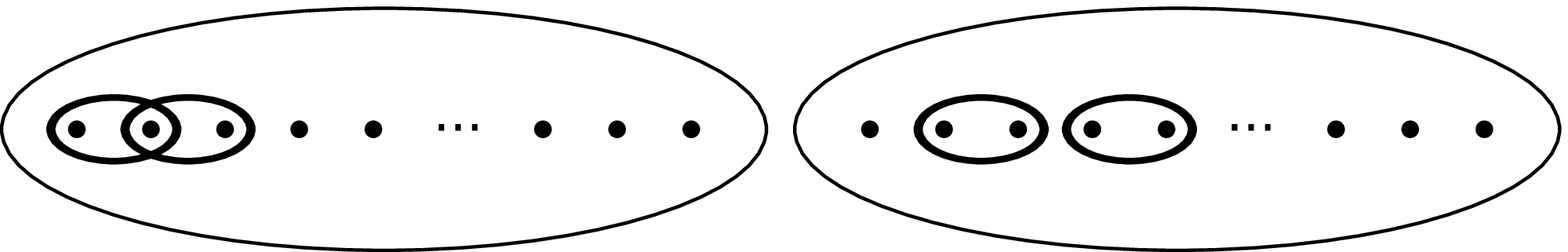}{The curves $c_{12}$, $c_{23}$, and $c_{45}$ in $\D_{2g+1}$ used in setting up Proposition~\ref{proposition:spaction} and Lemma~\ref{lemma:sp pb stab}.}

\begin{proposition}
\label{proposition:spaction}
Let $g \geq 3$.  Assume that $\BTorelli_{2h+1}=\Theta_{2h+1}$ for all $h < g$.  There then exists an action of $\Sp_{2g}(\Z)$ on the group $\Quotient_{g}$ with the following three properties:
\begin{enumerate}\setlength{\itemsep}{0ex}
\item for $Z \in \Sp_{2g}(\Z)$ and $\eta \in \Quotient_{g}$ we have 
$\pi(Z \cdot \eta) = Z \pi(\eta) Z^{-1}$,
\item for $\nu \in \QuotientEx_g$ and $\eta \in \Quotient_g$ we have
$\widehat{\pi}(\nu) \cdot \eta = \nu \eta \nu^{-1}$, and 
\item the action of $(\Sp_{2g}(\Z))_{\lax{\vec{v}_{23}}}$ on $\Quotient_{g}$ preserves $\Omega_{23}$.
\end{enumerate}
\end{proposition}

The second statement of Proposition~\ref{proposition:spaction} already completely specifies the desired action on a finite-index subgroup of $\Sp_{2g}(\Z)$, the image in $\Sp_{2g}(\Z)$ of $\Braid_{2g+1}$. 
The second statement also immediately implies that $\BTorelli_{2g+1}/\Theta_{2g+1}$ is central in $\Quotient_g$ (of course, our goal is to show that this quotient is trivial).

We will prove Proposition~\ref{proposition:spaction} in five steps.  First, in Section~\ref{section:presentations} we give explicit finite presentations for $\Quotient_g$ and $\Sp_{2g}(\Z)$.  Then in Section~\ref{section:constructaction} we propose an action of $\Sp_{2g}(\Z)$ on $\Quotient_g$ by declaring where each generator of $\Sp_{2g}(\Z)$ sends each generator of $\Quotient_g$.  Next, in Section~\ref{section:calc1} we check that the proposed action respects the relations of $\Quotient_g$, and in Section~\ref{section:calc2} we check that the proposed action respects the relations of $\Sp_{2g}(\Z)$.  Finally, in Section~\ref{section:cleanup}, we verify that the resulting action of $\Sp_{2g}(\Z)$ on $\Quotient_g$ satisfies the three properties listed in Proposition~\ref{proposition:spaction}.


\subsection{Presentations for \boldmath$\Quotient_g$ and \boldmath$\Sp_{2g}(\Z)$}
\label{section:presentations}

In this section, we give explicit finite presentations $\Sp_{2g}(\Z) \cong \Presentation{S_{\Sp}}{R_{\Sp}}$ and $\Quotient_g \cong \Presentation{S_{\Quotient}}{R_{\Quotient}}$.
The trick here is to find just the right balance: more generators for $\Sp_{2g}(\Z)$ will mean that checking the well-definedness of our action with respect to the $\Sp_{2g}(\Z)$ relations is easier (relations are smaller), but checking the well-definedness with respect to the $\Quotient_g$ relations is harder (more cases to check), and vice versa.  

\p{Generators for \boldmath$\Quotient_g$} Since $\Quotient_{g}$ is a quotient of $\PureBraid_{2g+1}$, any set of generators for $\PureBraid_{2g+1}$ descends to a set of generators for $\Quotient_{g}$.   We identify $\PureBraid_{2g+1}$ with the pure mapping class group of a disk $\D_{2g+1}$ with $2g+1$ marked points $p_1, \dots, p_{2g+1}$, that is, the group of homotopy classes of homeomorphisms of $\D_{2g+1}$ that fix each $p_i$ and each point of the boundary; see \cite[Section 9.3]{FarbMargalitPrimer}.  For concreteness, we take $\D_{2g+1}$ to be a convex Euclidean disk and the $p_i$ to lie on the vertices of a regular $(2g+1)$-gon, appearing clockwise in cyclic order; see Figure~\ref{figure:disc}.  Choose this identification so that if $c_{ij}$ is one of the curves $c_{12}$, $c_{23}$, or $c_{45}$ in 
Figure~\ref{figure:mainproof}, then $c_{ij}$ is the boundary of a convex region containing $p_i$ and $p_j$ 
and no other $p_k$.

\Figure{figure:disc}{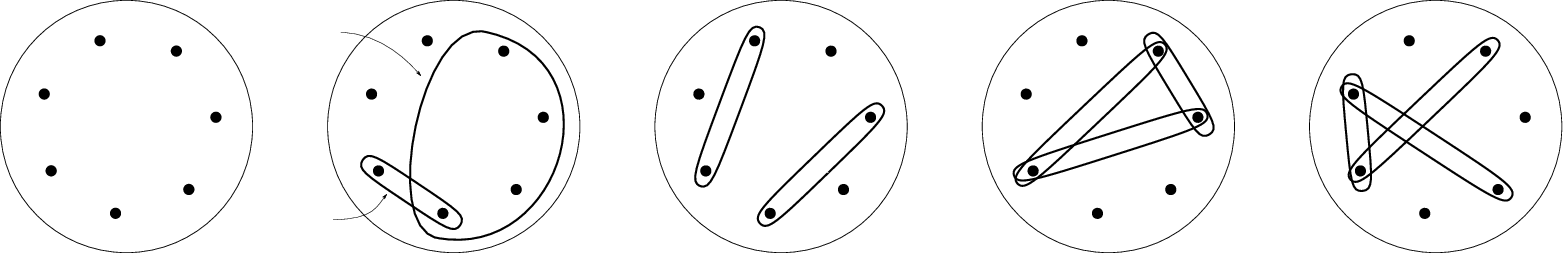}{The disk $\D_{2g+1}$ with its marked points arranged clockwise on the vertices of a convex $(2g+1)$-gon, then two convex simple closed curves, then the configurations of curves used in the disjointness relations, the triangle relations, and the crossing relations for the pure braid group.}

More generally, for any subset $A$ of $\{1, \dots,2g+1\}$ we denote by $c_A$ the simple closed curve in $\D_{2g+1}$ that bounds a convex region of $\D_{2g+1}$ containing precisely $\{p_i \mid i \in A\}$ in its interior; this curve is unique up to homotopy in $\D_{2g+1}$.  We will write $c_{ij}$ or $c_{i,j}$ for $c_{\{i,j\}}$, etc., when convenient.  The curves $c_{1234}$ and $c_{45}$ are shown in Figure~\ref{figure:disc}.

Artin proved that $\PureBraid_{2g+1}$ is generated by the Dehn twists about the curves in the set
\[  \C(S_\Quotient) = \{ c_{ij}  \mid 1 \leq i < j \leq 2g+1 \}. \]
The resulting generating set for $\Quotient_{g}$ is
\[  S_{\Quotient} = \{s_{c} \mid c \in \C(S_{\Quotient}) \}, \]
where by definition the element of $\Quotient_g$ associated to $s_c$ is $\rho(T_c)$.

\p{Relations for \boldmath$\Quotient_g$} Our set of relations $R_{\Quotient}$ for $\Quotient_{g}$ will consist of the four families of relations below.
Recall that $\Quotient_{g}$ is defined as $\Quotient_{g} = \PureBraid_{2g+1}/\Theta_{2g+1}$.
We first give a finite presentation for $\PureBraid_{2g+1}$, and then add relations for normal generators of $\Theta_{2g+1}$ inside $\PureBraid_{2g+1}$.
There are many presentations for the pure braid group, most notably the original one due to Artin \cite{ArtinBraids}.  We will use here a modified version of Artin's presentation due to the second author and McCammond \cite[Theorem 2.3]{MargalitMcCammond}.   There are three types of defining relations for $\PureBraid_{2g+1}$, as follows; refer to Figure~\ref{figure:disc}.  We will write $i_1 < \cdots < i_n$ to refer to the cyclic clockwise ordering of labels.
%
\begin{enumerate}
\item Disjointness relations: \quad $ [s_{c_{ij}}, s_{c_{rs}}] =1\  \text { if }\  i<j<r<s$.

\item Triangle relations: \quad $s_{c_{ij}}  s_{c_{jk}}   s_{c_{ki}} =  s_{c_{jk}}   s_{c_{ki}} s_{c_{ij}}    =  s_{c_{ki}} s_{c_{ij}}  s_{c_{jk}}     \ \text{ if }\ i < j < k$.

\item Crossing relations: \quad  $[s_{c_{ij}}, s_{c_{js}} s_{c_{rs}} s_{c_{js}}^{-1}] = 1
 \ \text{ if }\  i < r < j < s$.
 
\setcounter{enumi_saved}{\value{enumi}}
\end{enumerate}
%
We now add relations coming from $\Theta_{2g+1}$.  This group is normally generated in $\PureBraid_{2g+1}$ by the squares of Dehn twists about the convex curves in $\D_{2g+1}$ surrounding odd numbers of marked points; indeed any two Dehn twists about curves surrounding the same marked points are conjugate in $\PureBraid_{2g+1}$; cf. \cite[Section 1.3]{FarbMargalitPrimer}.  We need to add one relation for each of these elements. 
%
\begin{enumerate}
\setcounter{enumi}{\value{enumi_saved}}
\item Odd twist relations: \quad $\left((s_{c_{i_1i_2}} \cdots s_{c_{i_1i_n}}) \cdots  (s_{c_{i_{n-2}i_{n-1}}} s_{c_{i_{n-2}i_n}})s_{c_{i_{n-1}i_n}}\right)^2 = 1$ \vspace*{1ex}\\
\hspace*{2in}for any $i_1 < \cdots < i_n$, where $3 \leq n \leq 2g+1$.
\setcounter{enumi_saved}{\value{enumi}}
\end{enumerate}
The last relation comes from the following relation in the pure braid group: \label{pb relation}
\[ (T_{c_{i_1i_2}} \cdots T_{c_{i_1i_n}}) \cdots  (T_{c_{i_{n-2}i_{n-1}}} T_{c_{i_{n-2}i_n}})T_{c_{i_{n-1}i_n}} = T_{c_{i_1 i_2 \cdots i_n}};\]
see \cite[Section 9.3]{FarbMargalitPrimer}.

\bigskip

\p{Transvections in \boldmath$\Sp_{2g}(\Z)$}  We now turn to the symplectic group.  The \emph{transvection} on $\vec{v} \in \Z^{2g}$ is the element $\tau_{\vec{v}} \in \Sp_{2g}(\Z)$ given by 
\[ \tau_{\vec{v}}(\vec{w}) = \vec{w} +\hat \imath (\vec{w},\vec{v})\,\vec{v} \quad \quad (\vec{w} \in \Z^{2g}), \]
where $\hat \imath$ is the symplectic form.  Note that $\tau_{\vec{v}} = \tau_{-\vec{v}}$.  The group $\Sp_{2g}(\Z)$ is generated by transvections on primitive elements 
of $\Z^{2g}$.  Also, if $c$ is a simple closed curve in $\Sigma_g^1$, then the image of the Dehn twist $T_c \in \Mod_g^1$ in
$\Sp_{2g}(\Z)$ is $\tau_{[c]}$ for any choice of orientation of $c$.

\p{Transvections and simple closed curves} Consider a simple closed curve $a$ in $\D_{2g+1}$ surrounding an even number of marked points.  We construct a transvection {\em associated} to $a$
as follows.  The preimage of $a$ in $\Sigma_g^1$ is a pair of disjoint nonseparating simple closed curves $\tilde{a}_1$ and $\tilde{a}_2$ that
are homologous (with respect to some choice of orientation).  The transvection associated to $a$ is then $\tau_{[\tilde a_1]} = \tau_{[\tilde a_2]}$.
We pause now to record the following lemma.
\begin{lemma}
\label{lemma:s vs pi}
For a simple closed curve $a$ in $\D_{2g+1}$ surrounding an even number
of marked points, we have
\[ \pi(\rho(T_a)) = \tau_{\vec{v}}^2, \]
where $\tau_{\vec{v}}$ is the transvection associated to $a$.
\end{lemma}
\begin{proof}
We must determine the image of $T_a$ under $\Braid_{2g+1} \stackrel{L}{\to} \Mod_g^1 \to \Sp_{2g}(\Z)$, 
where $L$ is the lifting map from Section~\ref{section:intro}.
The preimage in $\Sigma_g^1$ of $a$ is a pair of disjoint simple closed homologous curves $\tilde{a}_1$ and $\tilde{a}_2$,
and $L(T_a) = T_{\tilde{a}_1} T_{\tilde{a}_2}$.  By the definition of $\tau_{\vec{v}}$, both
$T_{\tilde{a}_1}$ and $T_{\tilde{a}_2}$ map to $\tau_{\vec{v}} \in \Sp_{2g}(\Z)$, and the lemma follows.
\end{proof}

\p{Generators for \boldmath$\Sp_{2g}(\Z)$} Denote by $a_0$ the convex simple closed curve $c_{1234}$.  Also, for $1 \leq i \leq 2g$, set $a_i = c_{i,i+1}$.
Humphries (see \cite[Section 4]{FarbMargalitPrimer}) proved that one can choose connected components $\tilde{a}_i$ of $a_i$ in $\Sigma_{g}^1$ 
such that $\Mod_g^1$ is generated by the Dehn twists about the curves $\tilde{a}_0,\dots,\tilde{a}_{2g}$ (in fact, any set of choices will do).
Since $\Mod_g^1$ surjects onto $\Sp_{2g}(\Z)$, it follows that the transvections associated to $a_0,a_1,\dots,a_{2g}$ generate $\Sp_{2g}(\Z)$.  

In order to simplify our presentation for $\Sp_{2g}(\Z)$ we need to add some auxiliary generators to $\Sp_{2g}(\Z)$.  Consider the following curves:
\begin{align*}
a_0' &= c_{1245} & b_1' &= c_{2356} & b_3 &= c_{123456} & u' &= c_{2345}& v' &= c_{123567} \\
b_1 &= c_{1256} & b_2 &= c_{3456}& u &= c_{1267} & v &= c_{134567} & v'' &= c_{123467}.
\end{align*}
Let 
\[
\C(S_{\Sp}) = \{ a_0, \dots, a_{2g} \} \cup \{ a_0',b_1, b_1',b_2, b_3, u, u',v,v',v'' \}.
 \]
The resulting generating set for $\Sp_{2g}(\Z)$ is
\[ S_{\Sp} = \{ t_a \mid a \in \C(S_{\Sp}) \}, \]
where the element $\overline{t}_a \in \Sp_{2g}(\Z)$ associated to the generator $t_a$ is the transvection associated to $a$ as above.  It is remarkable that all of the curves in the generating set are convex.


\p{Relations for \boldmath$\Sp_{2g}(\Z)$} Our set of relations $R_{\Sp}$ for $\Sp_{2g}(\Z)$ will consist of the six families of relations below.
Since $\Sp_{2g}(\Z)\cong\Mod_g^1/\I_g^1$, we obtain a presentation for $\Sp_{2g}(\Z)$ by starting with a presentation for $\Mod_g^1$ and adding one relation for each normal generator of $\I_g^1$ in $\Mod_g^1$.
Wajnryb gave a finite presentation for $\Mod_g^1$ with Humphries' generating set $\{T_{\tilde{a}_0}, \dots, T_{\tilde{a}_{2g}}\}$; see Wajnryb's original paper \cite{WajnrybPresentation} and the erratum by Birman and Wajnryb \cite{WajnrybErrata}.  The image of $T_{\tilde{a}_i}$ in $\Sp_{2g}(\Z)$ is $\overline t_{a_i}$ and so we obtain the first part of our presentation for $\Sp_{2g}(\Z)$ by replacing each $T_{\tilde{a}_i}$ in Wajnryb's presentation with $t_{a_i}$.  We have the following list of relations, derived from the Wajnryb's standard presentation \cite[Theorem 5.3]{FarbMargalitPrimer}.  Here $i(\cdot,\cdot)$ denotes the geometric intersection number of two curves.
\begin{enumerate}
 \item Disjointness relations: \quad $t_{a_i} t_{a_j} = t_{a_j}t_{a_i} \text { if } i(a_i,a_j) = 0$
\item Braid relations: \quad $t_{a_i} t_{a_{j}} t_{a_i} = t_{a_{j}} t_{a_i} t_{a_{j}} \text{ if } i(a_i,a_j) = 2$
\item 3-chain relation: \quad $(t_{a_1}t_{a_2}t_{a_3})^4 = t_{a_0}t_{b_0},$
where
\[ t_{b_0} = (t_{a_4}t_{a_3}t_{a_2}t_{a_1}t_{a_1}t_{a_2}t_{a_3}t_{a_4}) t_{a_0} (t_{a_4}t_{a_3}t_{a_2}t_{a_1}t_{a_1}t_{a_2}t_{a_3}t_{a_4})^{-1} \]
\item Lantern relation: \quad $t_{a_0}t_{b_2}t_{b_1} = t_{a_1}t_{a_3}t_{a_5}t_{b_3}$
\setcounter{enumi_saved}{\value{enumi}}
\end{enumerate}
In the lantern relation, we have replaced some complicated expressions from Wajnryb's relations with some of our auxiliary generators.  Thus, similar to the reference \cite[Theorem 5.3]{FarbMargalitPrimer}, we need to add relations that express each of these generators in terms of the $t_{a_i}$.
\begin{enumerate}
\setcounter{enumi}{\value{enumi_saved}}
\item Auxiliary relations:
\begin{alignat*}{4}
(i) &\quad& t_{a_0'} &= (t_{a_4}t_{a_3})^{-1} t_{a_0} (t_{a_4}t_{a_3}) \quad\quad&(vi) &\quad& t_{u'} &= (t_{a_4}t_{a_3}t_{a_2}t_{a_1})^{-1}t_{a_0}(t_{a_4}t_{a_3}t_{a_2}t_{a_1}) \\
(ii) &\quad& t_{b_1} &= (t_{a_5}t_{a_4})^{-1} t_{a_0'} (t_{a_5}t_{a_4}) &(vii) &\quad& t_v &= t_u t_{u'} t_u^{-1} \\
(iii) &\quad& t_{b_1'} &= (t_{a_2}t_{a_1})^{-1} t_{b_1} (t_{a_2}t_{a_1}) &(viii) &\quad& t_{v'} &= (t_{a_3}t_{a_2})t_v(t_{a_3}t_{a_2})^{-1} \\
(iv) &\quad& t_{b_2} &= (t_{a_3}t_{a_2})^{-1} t_{b_1'} (t_{a_3}t_{a_2}) &(ix) &\quad& t_{v''} &= t_{a_4}t_{v'}t_{a_4}^{-1} \\
(v) &\quad& t_u &= (t_{a_6}t_{a_5})^{-1}t_{b_1}(t_{a_6}t_{a_5}) &(x) &\quad& t_{b_3} &= (t_{a_6}t_{a_5})t_{v''}(t_{a_6}t_{a_5})^{-1} 
\end{alignat*}
\setcounter{enumi_saved}{\value{enumi}}
\end{enumerate}
%
The auxiliary generators were introduced exactly so that we could break up the lantern relation into these shorter auxiliary relations.  This feature will be used in Section~\ref{section:calc2}.

By work of Johnson \cite{JohnsonEliminateSep}, the group $\I_g^1$ is normally generated by a single bounding pair map of genus 1 when $g \geq 3$.  Thus, to obtain our presentation for $\Sp_{2g}(\Z)$, we simply need one more relation.
\begin{enumerate}
\setcounter{enumi}{\value{enumi_saved}}
\item Bounding pair relation: \quad $t_{a_0}=t_{b_0}$, where $t_{b_0}$ is as in the 3-chain relation above.
\end{enumerate}

\p{Generators and intersection numbers} In choosing auxiliary generators for $\Sp_{2g}(\Z)$, we were careful not to introduce too many new generators; by inspection, we see that all generators satisfy the following useful property, used several times below.

\begin{lemma}
\label{lemma:the fact}
For any $a \in \C(S_{\Sp})$ and any convex simple closed curve $d$ in $\D_{2g+1}$, we have $i(a,d) \leq 4$.
\end{lemma}

An example of a curve $a$ that does not satisfy Lemma~\ref{lemma:the fact} is $a=c_{1246}$.


\subsection{Construction of the action}
\label{section:constructaction}

Let $t \in S_{\Sp}^{\pm 1}$ and $s \in S_{\Quotient}$.  The goal of this section is to construct an element $\genbygen{t}{s} \in \Quotient_g$ that satisfies the naturality property
\begin{equation}
\label{eqn:wordworks}
\pi\left( \genbygen{t}{s} \right) = \overline{t} \pi(\overline{s}) \overline{t}^{-1},
\end{equation}
where $\overline{w}$ denotes the image of an element of the free group on $S_{\Quotient}$ or $S_{\Sp}$ in the corresponding group; this is Proposition~\ref{prop:consistent} below.
We will show in Sections~\ref{section:calc1} and~\ref{section:calc2} that there is an action of $\Sp_{2g}(\Z)$ on $\Quotient_{g}$ defined by
\[ \overline{t} \cdot \overline{s} = \genbygen{t}{s} \]
and in Section~\ref{section:cleanup} we will show that this action satisfies Proposition~\ref{proposition:spaction}.

\p{Analogy with transvections} For a transvection $\tau_{\vec{w}} \in \Sp_{2g}(\Z)$ and a square of a transvection $\tau_{\vec{v}}^2 \in \Sp_{2g}(\Z)[2]$, we have
\begin{equation}\label{transvection formula} \tau_{\vec{w}}\tau_{\vec{v}}^2\tau_{\vec{w}}^{-1} = \tau_{\tau_{\vec{w}}(\vec{v})}^2. \end{equation}
Since transvections generate $\Sp_{2g}(\Z)$ and squares of transvections generate $\Sp_{2g}(\Z)[2]$, the action of $\Sp_{2g}(\Z)$ on $\Sp_{2g}(\Z)[2]$ 
is completely described by this formula.  If we write $\vec{w}_+(\vec{v})$ for $\tau_{\vec{w}}(\vec{v})$, then this formula becomes $\tau_{\vec{w}}\tau_{\vec{v}}^2\tau_{\vec{w}}^{-1} = \tau_{\vec{w}_+(\vec{v})}^2$.  In other words, the action of $\Sp_{2g}(\Z)$ on $\Sp_{2g}(\Z)[2]$ is given by an ``action'' of $\Z^{2g}$ on itself.  Our strategy is to give an analogous action of the set of curves in $\D_{2g+1}$ on itself, and use this to define each $\genbygen{t_a}{s_c}$. 

\Figure{figure:surgery2curve}{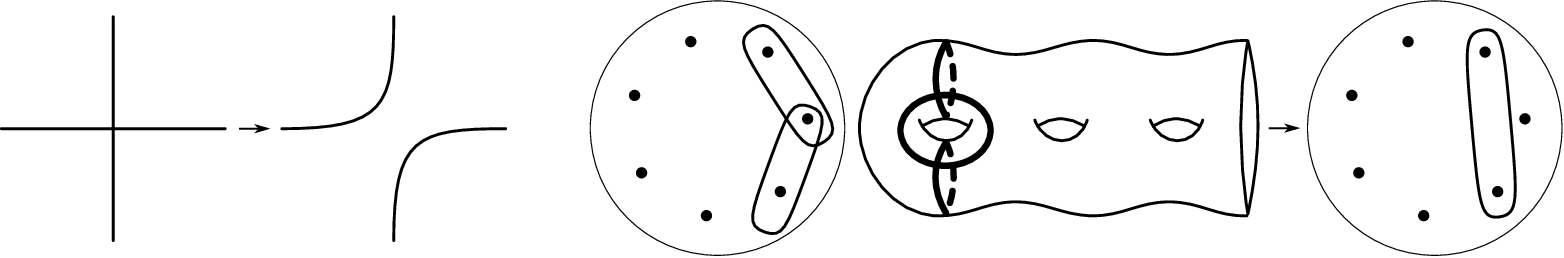}{Left: Surgery on curves in $\D_{2g+1}$.  Right: Surgery as a half-twist.
The preimage in $\Sigma_g^1$ of $a$ is $\tilde{a} \cup \Hyper(\tilde{a})$, and $\tilde{c}$ is one component of the preimage of $c$.  Thus $|\hat \imath|(a,c)=1$ and
$a_+(c) = \Surger(a,c)$ is the result of applying a half-twist about $a$ to $c$.}

\p{An action of curves on curves} If $d$ and $e$ are two collections of pairwise disjoint simple closed curves in $\D_{2g+1}$ in minimal position, we define $\Surger(d,e)$ to be the collection of simple closed curves obtained from $d \cup e$ by performing the surgery shown in the left-hand side of Figure~\ref{figure:surgery2curve} at each point of $d \cap e$ and then discarding any inessential components.  Note that this definition does not depend on any orientations of the elements of $d$ or $e$.

Next, for two simple closed curves 
$a$ and $c$ in $\D_{2g+1}$ we define $|\hat \imath|(a,c)$ 
to be the absolute value of the algebraic intersection number of any two connected components of the preimages of $a$ and $c$ in 
$\Sigma_g^1$.  These curves do not have a canonical orientation, so the algebraic intersection is not itself well defined.  Also, let $na$ denote $n$ parallel copies of the curve a.  Note that $\Surger(i(a,c)a,c)=T_a(c)$.

We now give our ``action'' of $\C(S_{\Sp})$ on $\C(S_\mathcal{Q})$.  For 
$a \in \C(S_{\Sp})$ and $c \in \C(S_\mathcal{Q})$, we define
\begin{align*}
a_+(c) &= \Surger(|\hat \imath|(a,c) a,c), \  \text{and} \\
a_-(c) &= \Surger(c,|\hat \imath|(a,c) a).
\end{align*}
It is important here that the curves $a$ and $c$ are in minimal position.

\p{The effect of the action on homology} In the case that $a$ surrounds two marked points and intersects $c$ in two points, then $a_+(c)$ is precisely the image of $c$ under the positive half-twist $H_a$ about $a$; see Figure~\ref{figure:surgery2curve}.  By definition, $\overline{t}_a$ is the image in $\Sp_{2g}(\Z)$ of $T_{\tilde a}$, where $\tilde a$ is one component of the preimage of $a$ in $\Sigma_g^1$.  But since $T_{\tilde a}$ is the image of $H_a$ in $\SMod_g^1$, it follows that $\overline{t}_a([\tilde c])$ is represented by one of the components of the preimage of $a_+(c)$ in $\Sigma_g^1$.  This naturality property is precisely the reason for our definition $a_+(c)$.

\Figure{figure:surgery4curve}{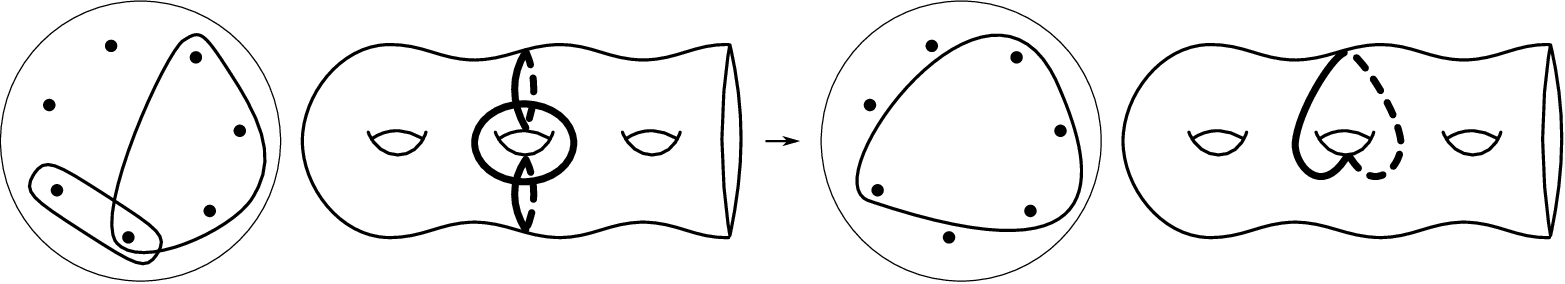}{A ``fake half-twist.''
The preimage in $\Sigma_g^1$ of $a$ is $\tilde{a} \cup \Hyper(\tilde{a})$, and $\tilde{c}$ is one component of the preimage of $c$.  Thus $|\hat \imath|(a,c)=1$ and
$a_+(c) = \Surger(a,c)$ is as shown.}

When $a$ surrounds four or more marked points, the situation is more subtle. Consider the curves $a=a_0$ and $c=c_{45}$; these curves and their preimages $\tilde c$ and $\tilde a \cup \iota(\tilde a)$ in $\Sigma_g^1$ are shown in Figure~\ref{figure:surgery4curve}.  The transvection $\overline{t}_a$ is the image in $\Sp_{2g}(\Z)$ of $T_{\tilde a}$.  This Dehn twist does not lie in $\SMod_g^1$, and so does not project to a homeomorphism of $\D_{2g+1}$.  However, the curve $T_{\tilde a}(\tilde c)$, which represents $\overline{t}_a([\tilde c])$,  still projects to a simple curve in $\D_{2g+1}$, namely $a_+(c)$.  So we again have the same naturality property as in the previous paragraph, that $\overline{t}_a([\tilde c])$ is represented by a lift of $a_+(c)$, even though $a_+(c)$ is not derived from $c$ via a homeomorphism of $\D_{2g+1}$.  We now show that this naturality property holds in general.

\begin{lemma}
\label{lemma:curve action}
For $a \in \C(S_{\Sp})$ and $c \in \C(S_\mathcal{Q})$, we have that $a_+(c)$ is a (connected) simple closed curve surrounding an even number of marked points.  If $\tilde a$, $\tilde c$, and $\widetilde{a_+(c)}$ are connected components of the preimages of $a$, $c$, and $a_+(c)$ in $\Sigma_g^1$, then, up to choosing compatible orientations on $\tilde c$ and $\widetilde{a_+(c)}$, we have 
\[ \overline{t}_a([\tilde c]) = \left[\widetilde{a_+(c)}\right]. \]
Similarly, $a_-(c)$ surrounds an even number of marked points and $\overline{t}_a^{-1}([\tilde c]) = \left[\widetilde{a_-(c)}\right]$.
\end{lemma}

\begin{proof}

We only treat the case of $a_+(c)$ with the other case being completely analogous.  We begin with the first statement, that $a_+(c)$ is connected and surrounds an even number of marked points.  The geometric intersection number $i(a,c)$ is equal to 0, 2, or 4; this is because $c_{ij}$ is the boundary of a regular neighborhood of the straight line segment connecting $p_i$ to $p_j$, and a straight line segment can intersect a convex curve in 0, 1, or 2 points (cf. Lemma~\ref{lemma:the fact}).  We treat each of the three cases in turn.  

If $i(a,c)=0$, then $|\hat \imath|(a,c)=0$.  Thus, $a_+(c)$ is equal to $c$, which is a simple closed curve.
If $i(a,c)=2$, then we claim that $|\hat \imath|(a,c)=1$.  Indeed, the arc of $a$ crossing through $c$ necessarily separates the two marked points inside $c$ from each other, creating two bigons, each containing one marked point.  The preimage of one bigon in $\Sigma_g^1$ is a square whose four corners are the four intersection points of the preimages of $a$ and $c$.  We know that the hyperelliptic involution $\Hyper$ interchanges the two lifts of each curve and that $\Hyper$ rotates the square by $\pi$.  Our claim follows.  It thus remains to check that $\Surger(a,c)$ is a simple closed curve surrounding an even number of marked points, which is immediate from Figure~\ref{figure:surgery3}.

\Figure{figure:surgery3}{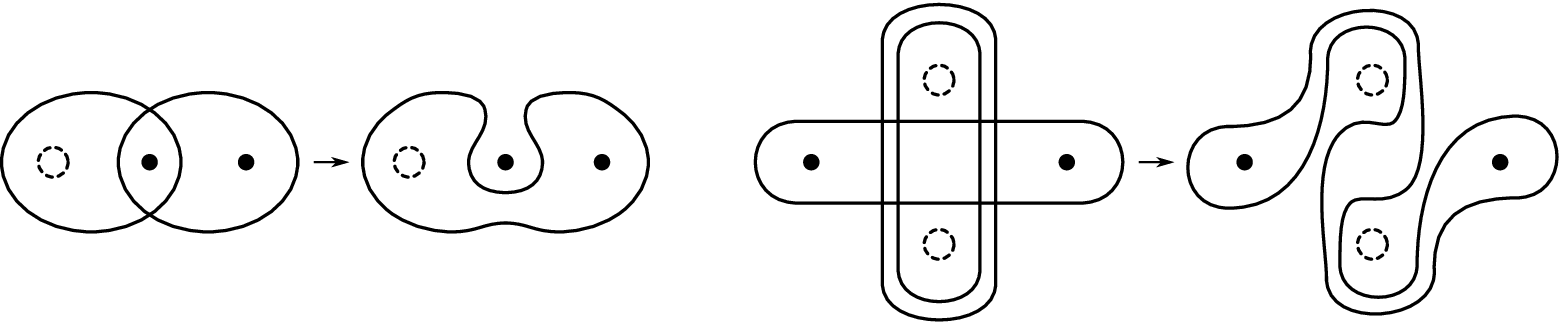}{Left: If $i(a,c)=2$, then $\Surger(a,c)$ is a simple closed curve.  Right: If $i(a,c)=4$, then $\Surger(2a,c)$ is a simple closed
curve.  In both figures, the small dashed circles contain an unspecified (but nonzero) number of marked points; both dashed circles on the left-hand picture contain an odd number of marked points.}

If $i(a,c)=4$, then we claim that $|\hat \imath|(a,c)$ is equal to either 0 or 2, depending on whether the arcs of $c$ divide the marked points inside $a$ into two sets of even cardinality or odd cardinality, respectively.  The curve $a$ divides the convex region bounded by $c$ into three connected components: one square and two bigons, each with one marked point.  Consider the union of the square and one bigon. The preimage in $\Sigma_g^1$ is a rectangle made up of three squares; there is one central square (the preimage of the bigon) and two other squares with edges glued to the left and right edges of the central square.  Since each intersection point in $\D_{2g+1}$ lifts to two intersection points in $\Sigma_g^1$, and since we already see 8 intersection points on the boundary of this rectangle, we conclude that this picture contains all of the intersection points of preimages of $a$ and $c$.  Also, by construction the horizontal sides of the rectangle belong to preimages of $c$.  The involution $\Hyper$ acts on this rectangle, rotating it by $\pi$ though the center.  We also know that $\Hyper$ interchanges the two preimages of each of $a$ and $c$.  Therefore, it suffices to count the intersections of the bottom of the rectangle with the vertical sides of the rectangle belonging to a single component of the preimage of $a$.  Again because $\Hyper$ exchanges the two components of the preimage of $a$, two of the vertical segments belong to one component, and two to the other. Thus, if we choose one component of the preimage of $a$, it intersects the bottom edge of the rectangle in precisely two points.  It immediately follows that $|\hat \imath|(a,c)$ is equal to either 0 or 2, as claimed.  By the claim, it suffices to check that $\Surger(2a,c)$ is a simple closed curve surrounding an even number of marked points, which is again immediate from
Figure~\ref{figure:surgery3}.

We now address the second statement of the lemma.  The preimage in $\Sigma_g^1$ of $|\hat \imath|(a,c) a \cup c$ is a symmetric configuration (that is, preserved by $\Hyper$).  It contains both preimages of $c$ and $|\hat \imath|(a,c)$ parallel copies of each  preimage of $a$.
We orient these so all preimages of $a$ represent the same element of $H_1(\Sigma_g^1;\Z)$.  We do the same for $c$; there are two choices, and we use the one that is consistent with the surgery in Figure~\ref{figure:surgery2curve}. When we perform surgery on this configuration, we therefore obtain a symmetric representative of the homology class
\[ 2[\tilde c] + 2|\hat \imath|(a,c)[\tilde a]. \]
This symmetric representative is the preimage of $a_+(c)$ and so the first statement of the lemma implies that this representative has exactly two connected components that are interchanged by the hyperelliptic involution.  It follows that each component, in particular $\widetilde{a_+(c)}$, represents
\[ \tau_{[\tilde a]}([\tilde c]) = [\tilde c] + |\hat \imath|(a,c)[\tilde a]. \]
But (up to sign) this is equal to $\overline{t}_a([\tilde c])$, and the lemma is proven.
\end{proof}

\p{Definition of \boldmath$\genbygen{t}{s}$} We can now define the elements $\genbygen{t_a^{\pm 1}}{s_c}  \in \Quotient_{g}$ for $t_a \in S_{\Sp}$ and $s_c \in S_{\Quotient}$:
\[\genbygen{t_a}{s_c} = \rho\left(T_{a_+(c)}\right) \quad \quad \text{and} \quad \quad
\genbygen{t_a^{-1}}{s_c} = \rho\left(T_{a_-(c)}\right).\]
These are both well-defined elements of $\Quotient_{g}$ since $T_{a_\pm(c)}$ only depends on the homotopy class of 
$a_\pm(c)$, and we already said that the latter is a well-defined simple closed curve.  

\p{Naturality}
We now verify the naturality property \eqref{eqn:wordworks} from the start of this section.

\begin{proposition}
\label{prop:consistent}
For any $t_a \in S_{\Sp}$ and $s_c \in S_{\Quotient}$ and $\epsilon \in \{-1,1\}$, we have
\[ \pi\left(\genbygen{t_a^\epsilon}{s_c} \right) = \overline{t}_a^{\,\epsilon} \pi(\overline{s}_c) \overline{t}_a^{\,-\epsilon}. \]
\end{proposition}

\begin{proof}

To simplify notation, we will treat the case $\epsilon=1$; the other case is essentially the same.  Let $\widetilde{a_+(c)}$ be one component in $\Sigma_g^1$ of the preimage of $a_+(c)$.  We have that 
\[ \pi(\genbygen{t_a}{s_c}) = \pi(\rho(T_{a_+(c)})) = 
\tau_{\left[\widetilde{a_+(c)}\right]}^2 
 = \tau_{\overline{t}_a([\tilde c])}^2 = \overline{t}_a \tau_{[\tilde{c}]}^2 \overline{t}_a^{\,-1} = \overline{t}_a \pi(\rho(T_c)) \overline{t}_a^{\,-1} = \overline{t}_a \pi(\overline{s}_c) \overline{t}_a^{\,-1}, \]
as desired.   The six equalities are given by 
the definition of $\genbygen{t_a}{s_c}$, 
Lemma~\ref{lemma:s vs pi}, 
Lemma~\ref{lemma:curve action}, 
Equation~\eqref{transvection formula} from the start of this section, Lemma~\ref{lemma:s vs pi}, and the definition of $s_c$.
\end{proof}


\subsection{Well-definedness with respect to \boldmath$\Quotient_{g}$ relations}
\label{section:calc1}

For $t \in S_{\Sp}^{\pm 1}$ and $s \in S_{\Quotient}$, we have now defined an element $\genbygen{t}{s}$ in $\Quotient_g$. Recall our goal is to show that the formula $\overline{t}_a \cdot \overline{s}_c = \genbygen{t_a}{s_c}$ defines an action of $\Sp_{2g}(\Z)$ on $\Quotient_{g}$.  However, at this point if we use this formula we do not even know that $(\overline{t}_a)^{-1} \cdot (\overline{t}_a \cdot \overline{s}_c)$ is equal to $\overline{s}_c$.

Let $F(S_{\Quotient})$ denote the free group on $S_{\Quotient}$.  For each $t \in S_{\Sp}^{\pm 1}$ what we do have now is a homomorphism $F(S_{\Quotient}) \rightarrow \Quotient_g$ given by $s \mapsto \genbygen{t}{s}$ for $s \in S_{\Quotient}$ (so it makes sense to write $\genbygen{t}{w}$ for $w \in F(S_{\Quotient})$).  The next proposition says that each of these homomorphisms respects the relations of $\Quotient_g$, which is to say that each of these homomorphisms induces an endomorphism of $\Quotient_g$.   To put it another way, the free monoid $\widehat{F}(S_{\Sp}^{\pm 1})$ acts on the group $\Quotient_g$.  In the next section we will show that this monoid action descends to a group action of $\Sp_{2g}(\Z)$ on $\Quotient_{g}$. 

\begin{proposition}
\label{prop:q well def}
Let $g \geq 3$ and assume $\BTorelli_{2h+1}=\Theta_{2h+1}$ for $h < g$.  For all $t \in S_{\Sp}^{\pm 1}$ and $r \in R_{\Quotient}$, we have $\premonoid{t}{r}=1$.
\end{proposition}

We introduce the following terminology, which will also be used in Section~\ref{section:calc2}.  Let $d$ be an essential simple closed curve in $\D_{2g+1}$.  An element of $\Quotient_g$ is said to be \emph{reducible along $d$} if it is the image of an element of $\PureBraid_{2g+1}$ that preserves the isotopy class of $d$.  The next lemma is an immediate consequence of Theorem~\ref{theorem:reducibletotwists}.

\begin{lemma}
\label{lemma:quasipreserve}
Assume that $\BTorelli_{2h+1} = \Theta_{2h+1}$ for $h < g$.  If $\eta \in \Quotient_g$ is reducible and $\pi(\eta) = 1$, then $\eta = 1$.
\end{lemma}

We need one more basic lemma about reducible elements, Lemma~\ref{lemma:Q red} below, but first we require a subordinate lemma, which follows immediately from Proposition~\ref{prop:consistent}.

\begin{lemma}
\label{lemma:pi of r}
Let $r \in F(S_\Quotient)$ be a relator for $\Quotient_{g}$ and let $t \in S_{\Sp}^{\pm 1}$.  Then $\pi(\premonoid{t}{r}) = 1$.
\end{lemma}

\begin{lemma}
\label{lemma:Q red}
Assume $\BTorelli_{2h+1} = \Theta_{2h+1}$ for $h < g$.  Let $t_a \in S_{\Sp}$, let $\epsilon = \pm 1$, and let $r \in F(S_\Quotient)$ be a relator for $\Quotient_{g}$.  
Suppose there is an essential simple closed curve $d$ in $\D_{2g+1}$ disjoint from $a$ and from each curve $c$ of 
$\C(S_\mathcal{Q})$ such that $s_c^{\pm 1}$ appears in $r$.  Then $\premonoid{t_a^{\epsilon}}{r} = 1$.
\end{lemma}

\begin{proof}
Write $r = s_{c_{i_1 j_1}}^{\epsilon_1} \cdots s_{c_{i_n j_n}}^{\epsilon_n}$ 
with $\epsilon_i=\pm 1$.  By hypothesis  each $c_{i_kj_k}$ is disjoint from $d$.  By definition of $\premonoid{t_a^{\epsilon}}{r}$, we have:
\[ \premonoid{t_a^{\epsilon}}{r} = (\genbygen{t_a^\epsilon}{s_{c_{i_1 j_1}}})^{\epsilon_1} \cdots (\genbygen{t_a^\epsilon}{s_{c_{i_n j_n}}})^{\epsilon_n}. \]
As $a$ is disjoint from $d$ and each $c_{i_kj_k}$ is disjoint from $d$, it follows from the definition 
of the action that each  $\genbygen{t_a^\epsilon}{s_{c_{i_k j_k}}}$ is reducible along $d$ (that is, 
if we surger two curves that are disjoint from $d$, the result is disjoint from $d$).  Since the set 
of elements of $\Quotient_g$ that are reducible along $d$ forms a subgroup of $\Quotient_g$, it follows that 
$\premonoid{t_a^{\epsilon}}{r}$ is reducible along $d$. 
By Lemma~\ref{lemma:pi of r}, $\pi(\premonoid{t_a^{\epsilon}}{r}) = 1$.  Lemma~\ref{lemma:quasipreserve} thus implies that $\premonoid{t_a^{\epsilon}}{r} = 1$.
\end{proof}

The next two lemmas give generating sets for two kinds of subgroups of $\PureBraid_{2g+1}$.  The first follows from the fact that any inclusion $\D_n \to \D_{2g+1}$ respecting marked points induces an inclusion on the level of mapping class groups \cite[Theorem 3.18]{FarbMargalitPrimer}.

\begin{lemma}
\label{lemma:diskstab}
Let $\Delta$ be a convex subdisk of $\D_{2g+1}$ containing $n$ marked points in its interior.  Then the subgroup of 
$\PureBraid_{2g+1}$ consisting of elements with representatives 
supported in $\Delta$ is isomorphic to $\PureBraid_{n}$ and is generated by the Dehn
twists $T_{c_{ij}}$ with $p_i,p_j \in \Delta$.
\end{lemma}

\begin{lemma}
\label{lemma:arcstab}
Let $1 \leq i,j \leq n$ be consecutive integers modulo $n$.  Then the stabilizer in $\PureBraid_{n}$ of the curve $c_{ij}$ is generated by the Dehn 
twists about curves in the set
\[ \C_{ij} = \{c_{ij}\} \cup \{ c_{k\ell} \mid k,\ell \notin \{i,j\} \} \cup \{ c_{ijk} \mid k \notin \{i,j\} \}.  \]
\end{lemma}

\begin{proof}

Let $(\PureBraid_{n})_{c_{ij}}$ denote the stabilizer in $\PureBraid_{n}$ of $c_{ij}$, and let $\gamma_{ij}$ denote 
the straight line segment connecting $p_i$ to $p_j$.   Any element of the group $(\PureBraid_{n})_{c_{ij}}$ has a representative that preserves $\gamma_{ij}$.  Any such  homeomorphism descends to a homeomorphism
of the disk with $n-1$ marked points obtained from $\D_{n}$ by collapsing $\gamma_{ij}$ to a single marked point.
This procedure gives rise to a short exact sequence:
\[ 1 \to \langle T_{c_{ij}} \rangle  \to (\PureBraid_{n})_{c_{ij}} \to \PureBraid_{n-1} \to 1; \]
cf. \cite[Proposition 3.20]{FarbMargalitPrimer}.  The curve $c_{ij}$ lies in $\C_{ij}$, and the Dehn twists about the other curves in $\C_{ij}$ map to the Artin generators for $\PureBraid_{n-1}$.  The lemma follows.
\end{proof}

Finally, in light of Lemma~\ref{lemma:arcstab}, we need to understand $\premonoid{t}{w}$, where $t \in S_{\Sp}^{\pm 1}$ and $w$ is an element of $F(S_{\Quotient})$ mapping to $\rho(T_{c_{ijk}}) \in \Quotient_g$.  We can obtain an explicit such $w$ using the relation in $\PureBraid_{2g+1}$ mentioned immediately after the list of relators for $\Quotient_g$ were introduced.  Indeed, if $s_{c_{ijk}} \in F(S_{\Quotient})$ is the element  $s_{c_{ij}}s_{c_{jk}}s_{c_{ik}}$, this relation tells us that $\overline s_{c_{ijk}} = \rho(T_{c_{ijk}})$.  We observe that, as an element of $F(S_{\Quotient})$, this $s_{c_{ijk}}$ depends on the order of $\{i,j,k\}$ (not just their cyclic order), though its image in $\Quotient_g$ only depends on the cyclic order.

\begin{lemma}
\label{lemma:3-curves}
Consider $1 \leq i,j,k \leq 2g+1$ with $i < j < k$ (up to cyclic permutation).  If $a=c_A \in \C(S_{\Sp})$ satisfies $i(a,c_{ij})=0$, then
$\premonoid{t_{a}^{\pm 1}}{s_{c_{ijk}}}$ is reducible along $c_{ij}$.
\end{lemma}

\Figure{figure:3curve}{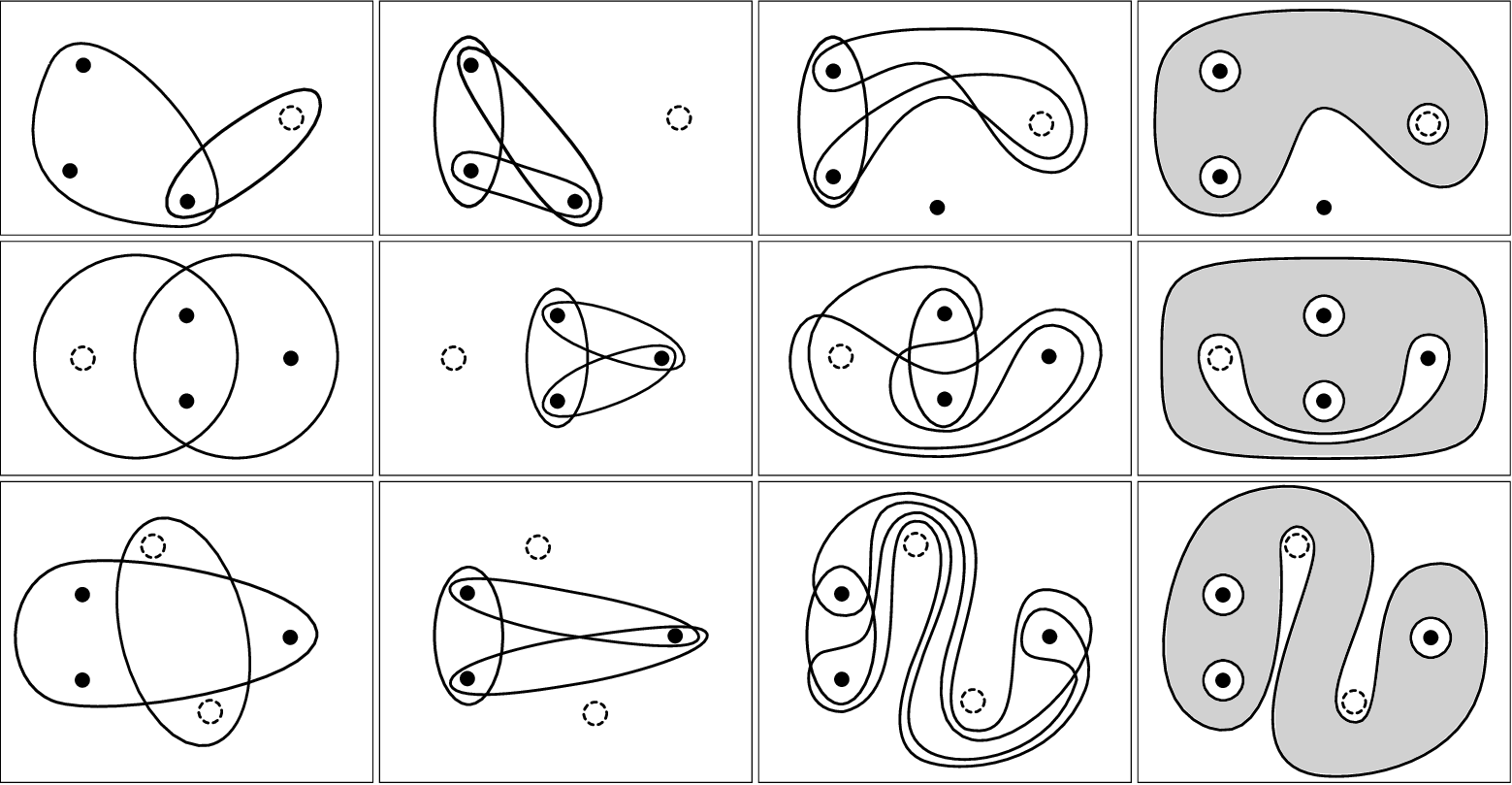}{The curves and four-holed spheres used in the proof of Lemma \ref{lemma:3-curves}.  From top to bottom, the rows correspond to the cases where $A \cap \{i,j,k\}$ is $\{k\}$, $\{i,j\}$, and $\emptyset$, respectively.  The dotted circles contain unspecified numbers of marked points.}

\begin{proof}

We will deal with $\premonoid{t_a}{s_{c_{ijk}}}$; the proof for $\premonoid{t_a^{-1}}{s_{c_{ijk}}}$ is similar.

If $|\hat\imath|(a,c_{jk}) = |\hat\imath|(a,c_{ik}) = 0$ (which holds in particular when
$i(a,c_{jk}) = i(a,c_{ik}) =0$), then we have
$a_+(c_{jk}) = c_{jk}$ and $a_+(c_{ik}) = c_{ik}$, so
\[\premonoid{t_a}{s_{c_{ijk}}} = (\premonoid{t_a}{s_{c_{ij}}}) (\premonoid{t_a}{s_{c_{jk}}}) (\premonoid{t_a}{s_{c_{ik}}}) = \rho(T_{c_{ij}}) \rho(T_{c_{jk}}) \rho(T_{c_{ik}}) = \rho(T_{c_{ijk}}).\]
Since $T_{c_{ijk}}$ fixes $c_{ij}$, the lemma follows.

We can therefore assume that at least one of $|\hat\imath|(a,c_{jk})$ and $|\hat\imath|(a,c_{ik})$ is nonzero.
The proof now divides into three cases depending on $A \cap \{i,j,k\}$.  Observe that $A$ is either disjoint
from or contains $\{i,j\}$.  Also, since either $|\hat\imath|(a,c_{jk})$ or $|\hat\imath|(a,c_{ik})$ is nonzero,
we cannot have $\{i,j,k\} \subset A$.

The first case is $A \cap \{i,j,k\} = \{k\}$; see the top row of Figure \ref{figure:3curve}.  In this case, $|\hat\imath|(a,c_{jk}) = |\hat\imath|(a,c_{ik}) = 1$, so
$a_+(c_{jk})$ and $a_+(c_{ik})$ are as in the top row of Figure \ref{figure:3curve}.  The
key to this step of the proof (as well as the subsequent ones) is the lantern relation in the mapping class
group (see \cite[Proposition 5.1]{FarbMargalitPrimer}).  This is a relation between seven Dehn twists
that lie in a sphere with four boundary components in any surface; the four-holed sphere
in this case is shaded in the top row of Figure \ref{figure:3curve}.   The associated lantern relation is
\[ T_{c_{ij}}T_{a_+\left(c_{jk}\right)}T_{a_+\left(c_{ik}\right)} 
=T_{d} T_{e} T_{c_i} T_{c_j}. \]
We can therefore compute that
\begin{align*}
\premonoid{t_a}{s_{c_{ijk}}} &= \premonoid{t_a}{(s_{c_{ij}}s_{c_{jk}}s_{c_{ik}})}=(\premonoid{t_{a}}{s_{c_{ij}}}) (\premonoid{t_{a}}{s_{c_{jk}}}) (\premonoid{t_{a}}{s_{c_{ik}}}) \\ 
&=\rho\left(T_{c_{ij}}\right) \rho\left(T_{a_+\left(c_{jk}\right)}\right) \rho\left(T_{a_+\left(c_{ik}\right)}\right)\\
&=\rho\left(T_{d} T_{e} T_{c_i} T_{c_j}\right) = \rho\left(T_{d} T_{e}\right);
\end{align*}
in order, the equalities use the definition of $s_{c_{ijk}}$, the definition of $\premonoid{t_a}{w}$, the definition of $\premonoid{t_a}{s_{c_{\ell m}}}$,
the above lantern relation, and the fact that $T_{c_i}$ and $T_{c_j}$ are trivial.  Since the curves $d$
and $e$ are disjoint from $c_{ij}$, it follows that $\premonoid{t_a}{s_{c_{ijk}}}$ is reducible along $c_{ij}$, as desired. 

The second case is $A \cap \{i,j,k\} = \{i,j\}$; refer now to the middle row of Figure \ref{figure:3curve}.  In
this case, we again have $|\hat\imath|(a,c_{jk}) = |\hat\imath|(a,c_{ik}) = 1$, so
$a_+(c_{jk})$ and $a_+(c_{ik})$ are as shown.  Just
like in the previous case, we can prove that $\premonoid{t_a}{s_{c_{ijk}}}$ is reducible along $c_{ij}$ using the indicated lantern relation.

The final case is $A \cap \{i,j,k\} = \emptyset$; refer to the bottom row of Figure~\ref{figure:3curve}.  Since at least one of $|\hat\imath|(a,c_{jk})$ 
and $|\hat\imath|(a,c_{ik})$ is nonzero, we cannot have $i(a,c_{ijk}) = 0$.  Using Lemma \ref{lemma:the fact},
we deduce that $a$ and $c_{ijk}$ are as shown.  We know
that $a$ must surround an even number of marked points, so the parities of the numbers of marked points
inside the dotted circles on the bottom row must be the same.  If this parity is even, then
$|\hat\imath|(a,c_{ik})=|\hat\imath|(a,c_{jk})=0$ (cf. the proof of Lemma~\ref{lemma:curve action}).  This is excluded by our assumptions (it was dealt with in the first paragraph of this proof), so this parity must be odd.  It then follows that $|\hat\imath|(a,c_{jk})=|\hat\imath|(a,c_{ik})=2$.  Therefore $a_+(c_{ij})$ and $a_+(c_{jk})$ and $a_+(c_{ik})$ are as in the bottom row of Figure \ref{figure:3curve}.  Just like in the case $A \cap \{i,j,k\} = \{k\}$, we can prove that $\premonoid{t_a}{s_{c_{ijk}}}$ is reducible along $c_{ij}$ using the indicated lantern relation.
\end{proof}

\begin{proof}[Proof of Proposition~\ref{prop:q well def}]

The proof will be broken into two steps.  For the first, let $R_{\PB} \subset R_{\Quotient}$
be the subset consisting of the disjointness, triangle, and crossing relations.  As was observed
in Section~\ref{section:presentations}, we have $\PureBraid_{2g+1} \cong \Presentation{S_{\Quotient}}{R_{\PB}}$.

\medskip \hspace*{3ex}\emph{Step 1.}
For $t \in S_{\Sp}^{\pm 1}$ and $r \in R_{\PB}$, we have $\premonoid{t}{r}=1$.  

\medskip

Write $t = t_a^{\epsilon}$ with $\epsilon = \pm 1$.
By Lemma~\ref{lemma:Q red}, it suffices to find an essential simple closed curve $d$ in $\D_{2g+1}$ 
disjoint from $a$ and from each curve of $\C(S_{\Quotient})$ that appears in $r$.

Denote by $\Delta_r$ the convex hull of curves of 
$\C(S_\mathcal{Q})$ that appear in $r$.  Examining the relations in $R_{\PB}$, we see that $\Delta_r$ contains at most $4$ marked points, and hence
there are at least $3$ marked points outside of $\Delta_r$.

It follows from Lemma~\ref{lemma:the fact} that the intersection of $a$ with the closure of the exterior of $\Delta_r$ is a union of at most two arcs.  These two arcs partition the marked points outside of $\Delta_r$ into at most three sets.  We deduce that one of the following holds:
\begin{enumerate}
\item some pair of marked points can be connected by an arc $\alpha$ disjoint from $a \cup \Delta_r$, or
\item the convex hull of of $a \cup \Delta_r$ contains at least one marked point in its exterior.  
\end{enumerate}
In the first case, we can take $d$ to be the boundary of a regular neighborhood of $\alpha$.  In the second case, 
we can take $d$ to be the boundary of the convex hull of $a \cup \Delta_r$.

\medskip

\hspace*{3ex}\emph{Step 2.} For $t \in S_{\Sp}^{\pm 1}$ and $r \in R_{\Quotient}$ an odd twist relator, $\premonoid{t}{r} = 1$.

\medskip

Again, write $t = t_a^{\epsilon}$ with $\epsilon = \pm 1$.  Consider $B \subset \{1,\ldots,2g+1\}$ with $3 \leq |B| \leq 2g+1$ and $|B|$ odd.  There is an odd twist relator $r_B$ corresponding to $B$ and its image under the map $F(S_\Quotient) \to \PureBraid_{2g+1}$ is $T_{c_B}^2$.  We need to show $\premonoid{t}{r_B}=1$.

It follows from Step 1 that if two elements $w$ and $w'$ of $F(S_\Quotient)$ have the same image in $\PureBraid_{2g+1}$, then $\premonoid{t}{w}=\premonoid{t}{w'}$.  Thus, we may replace the odd twist relator $r_B$ with any element of $F(S_\Quotient)$ whose image in $\PureBraid_{2g+1}$ is $T_{c_B}^2$.

First we treat the case where there is a marked point $p_k$ exterior to both $a$ and $c_B$.  Let $A_k = \{1,\dots,\hat k,\dots,2g+1\}$.  By Lemma~\ref{lemma:diskstab}, we can write $T_{c_B}^2$ as a product of Dehn twists (and inverse Dehn twists) about the $c_{ij}$ where $i,j \neq k$.  This implies
that there is a product $r_B'$ of $s_{c_{ij}}^{\pm 1} \in F(S_\Quotient)$ with $i,j \neq k$ whose image in $\PureBraid_{2g+1}$ is $T_{c_B}^2$.  Since $T_{c_B}^2$ lies in $\Theta_{2g+1}$, we have that $r_B'$ is a relator for $\Quotient_g$.  Since $a$ and each $c_{ij}$ with $i,j \neq k$ is disjoint from $c_{A_k}$, Lemma~\ref{lemma:Q red} gives that $\premonoid{t}{r_B'}$, hence $\premonoid{t}{r_B}$, is equal to 1, as desired.

Next suppose all marked points lie interior to either $a$ or $c_B$.  The proof of this case is similar, but we will have to contend with curves that surround three marked points, not just two, so Lemma~\ref{lemma:Q red} does not apply directly.  To begin, we claim that there exist $i,j \in B$ that are consecutive in $B$ (cyclically ordered) such that $c_{ij}$ is disjoint from $a$.

If there are at least three marked points of $\D_{2g+1}$ exterior to $a$, then it follows from Lemma~\ref{lemma:the fact} that $a$ and $B$ satisfy the claim.  The only remaining case for the claim is where $g=3$ (so $\D_{2g+1}=\D_7$) and $a$ surrounds 6 marked points.  In this case $i$ and $j$ can be taken to be any two marked points that lie inside $a$ and $c_B$ and are consecutive in $B$.

It remains to show that given $i,j$ consecutive in $B$ with $c_{ij}$ disjoint from $a$, we have $\premonoid{t}{r_B}=1$.  By Lemma~\ref{lemma:arcstab}, 
the element $T_{c_B}^2$ is the image in $\PureBraid_{2g+1}$ of a product $r_B'$ of $s_c \in F(S_{\Quotient})$ with $c \in \C_{ij}$ (here
we are using the definition of $s_{c_{ijk}}$ given before Lemma \ref{lemma:3-curves}).  It follows from Lemma~\ref{lemma:3-curves} that $\premonoid{t}{r_B'}$ is reducible along $c_{ij}$.  It then follows from Lemmas~\ref{lemma:quasipreserve} and~\ref{lemma:pi of r} that $\premonoid{t}{r_B'}$, hence $\premonoid{t}{r_B}$, is equal to 1.
\end{proof}


\subsection{Well-definedness with respect to \boldmath$\Sp_{2g}(\Z)$ relations }
\label{section:calc2}

Proposition~\ref{prop:q well def} implies that the free monoid $\widehat{F}(S_{\Sp}^{\pm 1})$ acts on $\Quotient_g$; we write this action as $(t,\eta) \mapsto \monoid{t}{\eta}$.  By definition, $\monoid{t}{\eta}$ is equal to $\genbygen{t}{w}$ where $w \in F(S_\Quotient)$ and $\eta = \overline{w}$ is the image of $w$ in $\Quotient_g$.

Let $\widehat R_{\Sp}$ denote $R_{\Sp} \cup \{tt^{-1} \mid t \in S_{\Sp}\}$, thought of as a subset of the free monoid on $S_{\Sp}^{\pm 1}$.  The next proposition says that the monoid action of $\widehat{F}(S_{\Sp}^{\pm 1})$ on $\Quotient_g$ respects the relations in $\widehat R_{\Sp}$; in other words, the monoid action descends to a group action of $\Sp_{2g}(\Z)$ on $\Quotient_{g}$.  

\begin{proposition}
 \label{prop:sp well def}
Let $g \geq 3$ and assume $\BTorelli_{2h+1}=\Theta_{2h+1}$ for $h < g$.  For all $r \in \widehat R_{\Sp}$ and $\eta \in \Quotient_g$, we have $\monoid{r}{\eta}=1$.
\end{proposition}

We begin with another lemma.  Let $r \in \widehat R_{\Sp}$ and $c \in \C(S_\Quotient)$.   We say that the pair $(r,c)$ satisfies the \emph{reducibility criterion} if either
\begin{itemize}\itemsep0pt
 \item[(1)] $c$ is disjoint from each curve of $\C(S_{\Sp})$ appearing in $r$, or 
 \item[(2)] there is a line segment that connects a marked point in $\D_{2g+1}$ to the boundary and that is disjoint from $c$ and each curve of $\C(S_{\Sp})$ appearing in $r$, or
 \item[(3)] there is a line segment that connects consecutive marked points in $\D_{2g+1}$ and that is disjoint from $c$ and each curve of $\C(S_{\Sp})$ appearing in $r$.
\end{itemize}
Notice that the first condition does not imply the third since $c$ might not surround consecutive marked points.

\begin{lemma}
\label{lemma:reductionlemma}
Fix some $r \in \widehat R_{\Sp}$ and $c \in \C(S_\Quotient)$ such that $(r,c)$ satisfies the reducibility criterion.  
Then $\monoid{r}{\overline{s}_c} = \overline{s}_c$.
\end{lemma}

\begin{proof}

Write $r = t_1 \cdots t_n$ where $t_i \in S_{\Sp}^{\pm 1}$.  We treat the three cases of the reducibility criterion separately.  First, if $c$ is disjoint from each element of $\C(S_{\Sp})$ appearing in $r$, then it follows from the definitions that $\monoid{r}{\overline s_c}=\overline s_c$.

The second case is when there is a line segment that joins a marked point $p_k$ to the boundary of $\D_{2g+1}$ and that is disjoint from $c$ and each curve of $
\C(S_{\Sp})$ that appears in $r$.  Let $d$ denote the convex simple closed curve surrounding all the marked points but $p_k$.
By the definition of the action of $\widehat{F}(S_{\Sp}^{\pm 1})$ on $\Quotient_g$, we have that  $\monoid{t_n}{\overline s_c} = \genbygen{t_n}{s_c}$ is reducible along $d$.  More specifically, $\monoid{t_n}{\overline s_c}$ is equal to $\rho(b_n)$, where $b_n\in\PureBraid_{2g+1}$ has a representative homeomorphism supported in the interior of $d$.  By Lemma~\ref{lemma:diskstab}, we can write $\monoid{t_n}{\overline s_c}$ as the image in $\Quotient_g$ of a product of Dehn twists (and inverse Dehn twists) about curves that surround two marked points and are disjoint from $d$.  It follows that $\monoid{t_{n-1}}{(\monoid{t_n}{\overline s_c})}$ is reducible along $d$ and is equal to $\rho(b_{n-1})$, where $b_{n-1}$ is represented by a homeomorphism supported in the interior of $d$.  Continuing inductively, we deduce that $\monoid{r}{\overline s_c}$ is reducible along $d$.  Since $\overline s_c$ is also reducible along $d$, we have that $\overline s_c(\monoid{r}{\overline s_c})^{-1}$ is reducible along $d$.  By Proposition~\ref{prop:consistent}, $\pi(\overline s_c)\pi(\monoid{r}{\overline s_c})^{-1} = 1$ in $\Sp_{2g}(\Z)$.  Then by Lemma~\ref{lemma:quasipreserve},  $\overline s_c(\monoid{r}{\overline s_c})^{-1}$ is equal to the identity in $\Quotient_{g}$, as desired.

The third case is when there is a straight line segment connecting consecutive marked points in $\D_{2g+1}$ and disjoint from $c$ and each curve of $\C(S_{\Sp})$ that appears in $r$.  Let $d = c_{k\ell}$ denote the boundary of a regular neighborhood of this line segment.  The argument is similar to the previous case.  The only difference is that when we factor the preimage of $\genbygen{t_n}{s_c}$ in $\PureBraid_{2g+1}$, we must use Dehn twists about curves that surround two or three marked points and are disjoint from $d$ (that such curves suffice follows from Lemma~\ref{lemma:arcstab}).  However, we can use the same argument, applying Lemma~\ref{lemma:3-curves} as necessary.
\end{proof}

\begin{proof}[{Proof of Proposition~\ref{prop:sp well def}}]

Examining the relators in $\widehat R_{\Sp}$ and the elements of $\C(S_\Quotient)$ one by one---see Figure~\ref{figure:sp rels} for a representative collection---we claim that, with a single exception, each relator $r \in \widehat R_{\Sp}$ satisfies the reducibility criterion with each $c \in \C(S_\Quotient)$.  When $g \geq 4$, one can always find a pair of consecutive marked points lying outside each curve in a given relator.  The only $c$ then that fails part (2) of the reducibility criterion is one surrounding those two points, but this curve satisfies part (1) of the reducibility criterion.  Thus, the claim is a finite check.  For each such non-exceptional choice of $r \in \widehat R_{\Sp}$ and $c \in \C(S_\Quotient)$,  Lemma \ref{lemma:reductionlemma} applies, and we have that $\monoid{r}{\overline{s}_c}=\overline{s}_c$.

The exceptional case is where $g=3$ and $r$ is the relator $r_{(vii)}$ corresponding to auxiliary relation $(vii)$ and $c \in \C(S_\Quotient)$ is $c_{47}$.   (This is the main place where the auxiliary generators in our presentation for $\Sp_{2g}(\Z)$ come into play; if we were to use Wajnryb's presentation without our auxiliary generators, the lantern relation would fail the reducibility criterion with every element of $\C(S_\Quotient)$ when $g=3$.)  It thus remains to show  $\monoid{r_{(vii)}}{\overline s_{c_{47}}}=\overline s_{c_{47}}$.

\Figure{figure:sp rels}{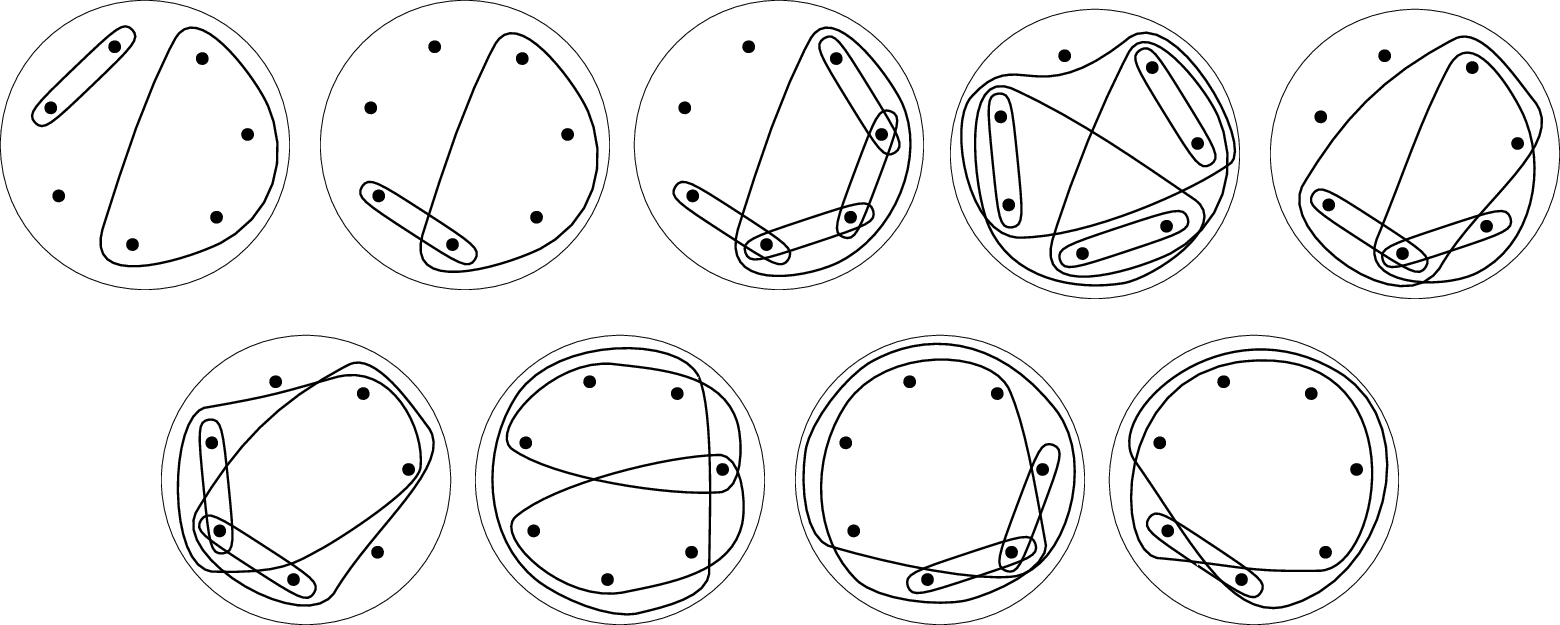}{Configurations of curves from $\C(S_{\Sp})$ arising in the relators for $\Sp_{2g}(\Z)$.  From the top left moving right: a disjointness relation, a braid relation, the chain relation, the lantern relation, and auxiliary relations (i), (ii), (vii), (viii), and (ix).}

Let $c_0$ denote $c_{1234567}$.  Using the relation in $\PureBraid_{n}$ from which we derived the odd twist relators for $\Quotient_g$ and the fact that $T_{c_0}$ is central in $\PB_7$, it follows that $\PB_7$ is generated by the $T_c$ with  $c \in \left(\C(S_{\Quotient})\cup \{c_{0} \}\right) \setminus \{c_{47}\}$.  Therefore, it suffices to show that $\monoid{r_{(vii)}}{\rho(T_{c_0})}=\rho(T_{c_0})$.

Each of the curves $u$, $u'$, and $v$ of $\C(S_{\Sp})$ appearing in $r_{(vii)}$ is disjoint from $c_{17}$, as is $c_0$.  As in the third case in the proof of Lemma~\ref{lemma:reductionlemma}, we can use Lemmas~\ref{lemma:arcstab} and~\ref{lemma:3-curves} with $c_{17}$ as the reducing curve to argue that $\monoid{r_{(vii)}}{\rho(T_{c_0})}=\rho(T_{c_0})$, as desired.
\end{proof}

\subsection{Completing the proof of Proposition~\ref{proposition:spaction}}
\label{section:cleanup}

By Propositions~\ref{prop:q well def} and~\ref{prop:sp well def}, there is an action of $\Sp_{2g}(\Z)$ on $\Quotient_g$ given by the formula
\[ \overline{t} \cdot \overline{s} = \genbygen{t}{s}. \]
It remains to check that this action has all three properties stipulated by Proposition~\ref{proposition:spaction}.
We already mentioned that property (1), namely, that $\pi(Z \cdot \eta) = Z \pi(\eta) Z^{-1}$ for $Z \in \Sp_{2g}(\Z)$ and $\eta \in \Quotient_{g}$, follows directly from Proposition~\ref{prop:consistent}.

Property (2) asserts that $\widehat{\pi}(\nu) \cdot \eta = \nu \eta \nu^{-1}$
for $\nu \in \QuotientEx_g$ and $\eta \in \Quotient_g$.  The half-twists about $a_1,\dots,a_{2g}$ are
the usual generators for the braid group $\Braid_{2g+1}$.  The half-twist $H_{a_i}$ maps to
$\overline{t}_{a_i} \in \Sp_{2g}(\Z)$, so to prove property (2) it is enough to show that
\[ \overline t_{a_i}^{\epsilon} \cdot \overline{s}_{c_{jk}} = \widehat\rho(H_{a_i}^{\epsilon}) \, \overline{s}_{c_{jk}}\, \widehat\rho(H_{a_i}^{-\epsilon}) \]
for all choices of $i,j,k$.  Since $\overline{s}_{c_{jk}}=\rho(T_{c_{jk}})$, the right-hand side is equal to $\overline{s}_{H_{a_i}^{\epsilon}(c_{jk})}=\rho(T_{H_{a_i}^{\epsilon}(c_{jk})})$, and so it remains to show that $(a_i)_{\epsilon}(c_{jk}) = H_{a_i}^{\epsilon}(c_{jk})$, where $(a_i)_{\epsilon}$ is interpreted as either $(a_i)_+$ or $(a_i)_-$.
For any choices of $i$, $j$, and $k$, we have $i(a_i,c_{jk})$ is either 0 or 2, and there is only one configuration in each case up to homeomorphisms of $\D_{2g+1}$.  In the
case $i(a_i,c_{jk})=0$, we have $(a_i)_\pm(c_{jk})=H_{a_i}^{\pm 1}(c_{jk})=c_{jk}$.  It remains to check the case $(i,j,k)=(1,2,3)$.  As mentioned in Section~\ref{section:constructaction}, we have in this case $(a_1)_\epsilon(c_{23}) = H_{a_1}^\epsilon(c_{23})$.

We now turn to property (3) of Proposition~\ref{proposition:spaction}, that the action of $(\Sp_{2g}(\Z))_{\lax{\vec{v}_{23}}}$ on $\Quotient_g$ preserves $\Omega_{23}$, the subgroup consisting of all elements that are reducible along $c_{23}$.  We will use the fact that $(\Sp_{2g}(\Z))_{\lax{\vec{v}_{23}}}$ is generated by the
set
$$\Xi = \{-I,t_{a_2}, t_{u'}, t_{b_3}, t_{a_4},t_{a_5},\ldots,t_{a_{2g}}\}.$$
That $\Xi$ generates is an immediate consequence of the semidirect product decomposition for the stabilizer in $\Sp_{2g}(\Z)$ of a primitive lax vector that is given in the proof of Lemma \ref{lemma:stabilizersurject}; this fact can also be proven in much the same way as the level 2 version, Lemma~\ref{lemma:sp pb stab} below.

Let $\Upsilon$ be the image in $\Quotient_g$ of the generating set for $(\PureBraid_{2g+1})_{c_{23}}$ from Lemma~\ref{lemma:arcstab}.  It is enough to show that for $x \in \Xi$ and $y \in \Upsilon$, the element $x \cdot y$ is reducible along $c_{23}$.  First, if $x=-I$, then it follows immediately from property (2), the fact that $T_{\partial \D_{2g+1}}$ is central in $\PureBraid_{2g+1}$ and the fact that $\pi \circ \rho(T_{\partial \D_{2g+1}}) = -I$ that $-I \cdot y = T_{\partial \D_{2g+1}} y T_{\partial \D_{2g+1}}^{-1} = y$.  Next, we may assume that $x \neq -I$; note then that $x \in S_{\Sp}$.  For the elements $y \in \Upsilon$ that lie in $S_{\Quotient}$, the reducibility along $c_{23}$ of $x \cdot y$ is obvious from the description of our action in Section~\ref{section:constructaction}.  For the others, it is an immediate consequence of Lemma~\ref{lemma:3-curves}.  This completes the proof.


\section{The proof of the Main Proposition}
\label{section:theproofmain}

In this section, we prove Proposition~\ref{proposition:main} by induction on $g$ using Propositions~\ref{prop:sp presentation} and~\ref{proposition:spaction}.  The base case is $g=2$, which we already said is known to be true.  So we can assume that $g \geq 3$ and that the quotient map $\Quotient_h \to \Sp_{2h}(\Z)[2]$ is an isomorphism for $h < g$.  Equivalently, we are assuming that $\BTorelli_{2h+1} = \Theta_{2h+1}$ for all $h < g$ and we want to prove that the quotient map $\pi : \Quotient_{g} \rightarrow \Sp_{2g}(\Z)[2]$ is an isomorphism.  The map $\pi$ is a surjection, so it is enough to construct a homomorphism $\phi : \Sp_{2g}(\Z)[2] \rightarrow \Quotient_{g}$ such that $\phi \circ \pi = 1$.

Let $X_g$ denote $\PartialIBases_g(\Z)$ when $g \geq 4$ and $\PartialIBasesEx_g(\Z)$ when $g=3$.  We will construct the map $\phi$ in two steps.  First in Lemmas~\ref{lemma:phiv wd} and~\ref{lemma:piphiv} we will use 
Proposition~\ref{proposition:spaction} to construct a homomorphism
$$\widetilde{\phi} : \BigFreeProd_{\lax{\vec{v}} \in X_g^{(0)}} \left(\Sp_{2g}(\Z)[2]\right)_{\lax{\vec{v}}} \rightarrow \Quotient_{g}$$
(recall that Proposition~\ref{proposition:spaction} requires the assumption that $\BTorelli_{2h+1} = \Theta_{2h+1}$ for all $h < g$).  Then we will show that $\widetilde{\phi}$ takes the edge and conjugation relators from Proposition~\ref{prop:sp presentation} to the identity (Lemmas~\ref{lemma:edge} and~\ref{lemma:conjugation}), so it induces
a homomorphism $\phi : \Sp_{2g}(\Z)[2] \rightarrow \Quotient_{g}$.  Finally, we will check that $\phi \circ \pi$ is equal to the identity (Lemma~\ref{lemma:inverse}), completing the proof.


\p{A stabilizer lemma} Before getting on with the construction of the inverse map $\phi$, we require a lemma.  Recall that
$c_{12}$ and $c_{45}$ are the curves in $\D_{2g+1}$ shown in Figure~\ref{figure:mainproof}.  Denote by  
$(\PureBraid_{2g+1})_{\{c_{23},c_{45}\}}$ and $(\PureBraid_{2g+1})_{\{c_{23},c_{12}\}}$ the corresponding stabilizers.  
We can define lax integral homology classes $\lax{\vec{v}_{12}}$ and $\lax{\vec{v}_{45}}$ analogously to the way we defined $\lax{\vec{v}_{23}}$ in Section~\ref{section:action setup}.  Denote by $(\Sp_{2g}(\Z)[2])_{\lax{\vec{v}_{23}}}$ and $(\Sp_{2g}(\Z)[2])_{\{\lax{\vec{v}_{23}},\lax{\vec{v}_{12}}\}}$ and $(\Sp_{2g}(\Z)[2])_{\{\lax{\vec{v}_{23}},\lax{\vec{v}_{45}}\}}$ the corresponding stabilizers.  Recall that $\rho$ is the projection $\PureBraid_{2g+1} \rightarrow \Quotient_g$.

\begin{lemma}
\label{lemma:sp pb stab}
The following restrictions of $\pi \circ \rho$ are all surjective:
\begin{alignat*}{3}
&(\PB_{2g+1})_{\{c_{23},c_{12}\}} &&\to (\Sp_{2g}(\Z)[2])_{\{\lax{\vec{v}_{23}},\lax{\vec{v}_{12}}\}} \\
&(\PB_{2g+1})_{c_{23}}  &&\to (\Sp_{2g}(\Z)[2])_{\lax{\vec{v}_{23}}} \\
&(\PB_{2g+1})_{\{c_{23},c_{45}\}} &&\to (\Sp_{2g}(\Z)[2])_{\{\lax{\vec{v}_{23}},\lax{\vec{v}_{45}}\}}.  \\ \end{alignat*}
\end{lemma}


\begin{proof}

Let $A_i = \{2i,2i+1\}$ and $B_i = \{1,\dots,2i\}$.  We can choose and orient a connected component of the preimage 
in $\Sigma_g^1$ of each of $c_{A_1},c_{B_1},\dots,c_{A_g},c_{B_g}$ in order to obtain a symplectic basis 
$(\vec{a}_1,\ldots,\vec{a}_g;\vec{b}_1,\ldots,\vec{b}_g)$ for $H_1(\Sigma_g^1)$.  Note that $c_{A_1}=c_{23}$, $c_{B_1}=c_{12}$,
and $c_{A_2}=c_{45}$, so we can choose the orientations such that $\vec{v}_{23}=\vec{a}_1$,  $\vec{v}_{12}=\vec{b}_1$, and $\vec{v}_{45}=\vec{a}_2$.
Next, let $E_i=\{2,3,2i,2i+1\}$ and $F_i=\{1,4,5,\dots,2i\}$.  The oriented lifts of $c_{E_i}$ and $c_{F_i}$ lie in 
the homology classes $\vec{a}_1\pm \vec{a}_i$ and $\vec{a}_1 \pm \vec{b}_i$.  The two signs here depend on 
the choice of the $\vec{a}_i$ and $\vec{b}_i$.  To simplify the notation, we assume both signs are positive.  We now proceed in three steps, corresponding to the three statements of the lemma.

\bigskip

{\em Step 1.} The map $(\PB_{2g+1})_{\{c_{23},c_{12}\}} \to (\Sp_{2g}(\Z)[2])_{\{\lax{\vec{v}_{23}},\lax{\vec{v}_{12}}\}}$ is surjective.

\medskip

The group $(\Sp_{2g}(\Z)[2])_{\{\lax{\vec{v}_{23}},\lax{\vec{v}_{12}}\}}$
preserves the submodule $\langle \vec{a}_1,\vec{b}_1 \rangle$ of $\Z^{2g}$ and thus also preserves its orthogonal complement $\langle \vec{a}_2,\vec{b}_2,\ldots,\vec{a}_g,\vec{b}_g \rangle$.  Since the mod $2$ reductions of $\vec{a}_1 = \vec{v}_{23}$ and $\vec{b}_1=\vec{v}_{12}$ are different, elements of $(\Sp_{2g}(\Z)[2])_{\{\lax{\vec{v}_{23}},\lax{\vec{v}_{12}}\}}$ take $\vec{a}_1$ to $\pm \vec{a}_1$ and $\vec{b}_1$ to $\pm \vec{b}_1$ for some choice of signs.  Since the algebraic intersection pairing must be preserved, these signs must be the same.  In summary, $(\Sp_{2g}(\Z)[2])_{\{\lax{\vec{v}_{23}},\lax{\vec{v}_{12}}\}} = (\Z/2) \oplus \Sp_{2g-2}(\Z)[2]$, where $\Z/2$ acts on $\langle \vec{a}_1,\vec{b}_1 \rangle$ by $\pm I$ and where $\Sp_{2g-2}(\Z)[2]$ acts on the orthogonal complement in the usual way.

Under this isomorphism, $\pi \circ \rho(T_{c_{123}}) = (-I,\text{id})$; indeed, the preimage in $\Sigma_g^1$ of the subdisk bounded by $c_{123}$ is homeomorphic to $\Sigma_1^1$ and $T_{c_{123}}$ lifts to a hyperelliptic involution of this subsurface.  It is therefore enough to show
that the composition of the restriction of $\pi \circ \rho$ to $(\PB_{2g+1})_{\{c_{23},c_{12}\}}$ with
the projection map $(\Sp_{2g}(\Z)[2])_{\{\lax{\vec{v}_{23}},\lax{\vec{v}_{12}}\}} \rightarrow \Sp_{2g-2}(\Z)[2]$
is surjective.  This map factors as
\[(\PB_{2g+1})_{\{c_{23},c_{12}\}} \stackrel{\xi_{2g+1}}{\longrightarrow} \PB_{2g-1} \stackrel{\beta_{2g-1}}{\longrightarrow} \Sp_{2g-2}(\Z)[2],\]
where the map $\xi_{2g+1}$ is obtained by collapsing the disk bounded by $c_{123}$ to a single marked point; this makes sense because 
$(\PB_{2g+1})_{\{c_{23},c_{12}\}} \subseteq (\PB_{2g+1})_{c_{123}}$.  The map $\xi_{2g+1}$ is surjective because every homeomorphism of $\D_{2g-1}$ can be homotoped such that it fixes a disk surrounding the first marked point, and we have already stated that $\beta_{2g-1}$ is surjective.  This completes the proof of the first statement.

\bigskip

{\em Step 2.} The map $(\PB_{2g+1})_{c_{23}} \to (\Sp_{2g}(\Z)[2])_{\lax{\vec{v}_{23}}}$
is surjective.

\medskip

Consider $Y \in (\Sp_{2g}(\Z)[2])_{\lax{\vec{v}_{23}}}$.  The Dehn twist 
$T_{c_{123}}$ lies in $(\PB_{2g+1})_{c_{23}}$ and takes $\vec{a}_1$ to $-\vec{a}_1$, so without loss of generality we can assume that $Y(\vec{a}_1) = \vec{a}_1$. Since $Y$ preserves the algebraic intersection pairing and $\hat\imath(\vec{a}_1,\vec{b}_1)=1$, the  $\vec{b}_1$-coordinate of $Y(\vec{b}_1)$ is $1$, so
\[Y(\vec{b}_1) = \ell_1\vec{a}_1+\vec{b}_1+\ell_2\vec{a}_2 + m_2\vec{b}_2 + \cdots + \ell_g\vec{a}_g+m_g\vec{b}_g \quad \quad (\ell_i,m_i \in \Z).\]
Since $Y \in \Sp_{2g}(\Z)[2]$, each $\ell_i$ and $m_i$ is even.  For $2 \leq i \leq g$, 
set $n_i = m_i\ell_i-\ell_i-m_i$.  Define $Z$ to equal
\[  
\tau_{\vec{a}_1}^{\ell_1}\left(\tau_{\vec{a}_1}^{n_2} \tau_{\vec{a}_1+\vec{a}_2}^{\ell_2}   \tau_{\vec{b}_2}^{-m_2} \tau_{\vec{a}_1+\vec{b}_2}^{m_2}  \right) \cdots \left(\tau_{\vec{a}_1}^{n_g} \tau_{\vec{a}_1+\vec{a}_g}^{\ell_g}   \tau_{\vec{b}_g}^{-m_g} \tau_{\vec{a}_1+\vec{b}_g}^{m_g}  \right) \in (\Sp_{2g}(\Z)[2])_{\vec{a}_1}.
\]
The key property of $Z$ is that $ZY(\vec{b}_1) = \vec{b}_1$.
Since $\tau_{\vec{a}_i}^2$, $\tau_{\vec{b}_i}^2$, $\tau_{\vec{a}_1+\vec{a}_i}^2$, and 
$\tau_{\vec{a}_1+\vec{b}_i}^2$ are the images under $\pi \circ \rho$ of the Dehn twists about $c_{A_i}$, $c_{B_i}$, 
$c_{E_i}$, and $c_{F_i}$, we have an explicit $b \in (\PB_{2g+1})_{c_{23}}$ with $\pi \circ \rho(b) = Z$.  Using
this $b$, we can modify $Y$ so that $Y(\vec{b}_1) = \vec{b}_1$, so 
$Y \in (\Sp_{2g}(\Z)[2])_{\{\lax{\vec{a}_1},\lax{\vec{b}_1}\}}$.  We have thus reduced the second statement to the first.

\bigskip

{\em Step 3.} The map $(\PB_{2g+1})_{\{c_{23},c_{45}\}} \to (\Sp_{2g}(\Z)[2])_{\{\lax{\vec{v}_{23}},\lax{\vec{v}_{45}}\}}$
is surjective.

\medskip

Consider $Y \in (\Sp_{2g}(\Z)[2])_{\{\lax{\vec{v}_{23}},\lax{\vec{v}_{45}}\}}$.  As in Step 2, we may assume that $Y(\vec{a}_1) = \vec{a}_1$.  Since $Y$ preserves the algebraic intersection pairing and $\hat\imath(\vec{a}_1,\vec{b}_1) = 1$ and $\hat\imath(\vec{a}_2,\vec{b}_1)=0$, the $\vec{b}_1$-coordinate of $Y(\vec{b}_1)$ is $1$ and the $\vec{b}_2$-coordinate is $0$.  We thus have
\[Y(\vec{b}_1) = \ell_1\vec{a}_1+\vec{b}_1+\ell_2\vec{a}_2 + \ell_3\vec{a}_3+m_3\vec{b}_3 + \cdots + \ell_g\vec{a}_g+m_g\vec{b}_g \quad \quad (\ell_i,m_i \in \Z).\]
Just like before, each $\ell_i$ and $m_i$ is even.  Define $n_i = m_i\ell_i-\ell_i-m_i$ and $Z$ to equal
\[  
\tau_{\vec{a}_1}^{\ell_1} \left(\tau_{\vec{a}_1}^{-\ell_2}  \tau_{\vec{a}_1+\vec{a}_2}^{\ell_2} \right)\left(\tau_{\vec{a}_1}^{n_3} \tau_{\vec{a}_1+\vec{a}_3}^{\ell_3}   \tau_{\vec{b}_3}^{-m_3} \tau_{\vec{a}_1+\vec{b}_3}^{m_3}  \right) \cdots \left(\tau_{\vec{a}_1}^{n_g} \tau_{\vec{a}_1+\vec{a}_g}^{\ell_g}   \tau_{\vec{b}_g}^{-m_g} \tau_{\vec{a}_1+\vec{b}_g}^{m_g}  \right) \in (\Sp_{2g}(\Z)[2])_{\{\vec{a}_1,\vec{a}_2\}}.
\]
The key property of $Z$ is that $Z(Y(\vec{b}_1)) = \vec{b}_1$.
As before, each square of a transvection appearing in $Z$ is the image of a Dehn twist about a curve disjoint from $c_{23}$ 
and $c_{45}$, so there is a $b \in (\PB_{2g+1})_{\{c_{23},c_{45}\}}$ with $\pi \circ \rho(b) = Z$.  Using this
$b$, we can modify $Y$ so that $Y(\vec{b}_1) = \vec{b}_1$, so
$Y \in (\Sp_{2g}(\Z)[2])_{\{\vec{a}_1,\vec{b}_1,\vec{a}_2\}}$.

The collapsing map $\D_{2g+1} \to \D_{2g-1}$ described in the first step induces a collapsing map $\Sigma_g^1 \to \Sigma_{g-1}^1$ whereby a torus with one boundary component (the preimage of the disk bounded by $c_{123}$) is collapsed to a point.  There is an induced splitting $H_1(\Sigma_g^1) \cong H_1(\Sigma_1^1) \oplus H_1(\Sigma_{g-1}^1) \cong \Z^2 \oplus \Z^{2g-2}$.  Under  this identification, $\vec{v}_{45}$ lies in the $\Z^{2g-2}$ factor.  The map $(\PB_{2g+1})_{\{c_{23},c_{12}\}} \to (\Z/2) \oplus \Sp_{2g-2}(\Z)[2]$ described in Step 1 thus restricts to a map $(\PB_{2g+1})_{\{c_{23},c_{12},c_{45}\}} \to (\Z/2) \oplus (\Sp_{2g-2}(\Z)[2])_{\lax{\vec{v}_{45}}}$.  There is a commutative diagram
\[
\xymatrix{
(\PB_{2g+1})_{\{c_{23},c_{12},c_{45}\}} \ar[r] \ar[d]_{\xi_{2g+1}} & (\Z/2) \oplus (\Sp_{2g-2}(\Z)[2])_{\lax{\vec{v}_{45}}} \ar[d] \\
(\PB_{2g-1})_{d_{23}} \ar[r] & (\Sp_{2g-2}(\Z)[2])_{\lax{\vec{v}_{45}}} 
}
\]
where $\xi_{2g+1}$ is the restriction of the map described in Step 1, where $d_{23} \subseteq \D_{2g-1}$ is the image of $c_{45}$ under the collapsing map, and where the rightmost vertical map is projection onto the second factor.  The leftmost vertical map is surjective as in Step 1, and the rightmost vertical map is obviously surjective.  Recall that $Y$ lies in $(\Sp_{2g}(\Z)[2])_{\{\vec{a}_1,\vec{b}_1,\vec{a}_2\}} \cong (\Z/2) \oplus (\Sp_{2g-2}(\Z)[2])_{\lax{\vec{v}_{45}}}$.  Since $T_{c_{123}} \in (\PB_{2g+1})_{\{c_{23},c_{12},c_{45}\}}$ maps to the generator of the first factor, we have reduced the problem to the surjectivity of the bottom horizontal map.  This is equivalent to Step 2, so we are done.
\end{proof}

\p{Construction of \boldmath$\widetilde\phi$}
We are now ready to define $\widetilde\phi$.  For each $\lax{\vec{v}} \in X_{g}^{(0)}$, we need to construct a homomorphism
\[ \widetilde{\phi}_{\lax{\vec{v}}} : (\Sp_{2g}(\Z)[2])_{\lax{\vec{v}}} \rightarrow \Quotient_{g}.\]  
We start by dealing with the special case $\lax{\vec{v}} = \lax{\vec{v}_{23}}$.  Recall that $\Omega_{23}$ is the image in $\Quotient_g$ of $(\PureBraid_{2g+1})_{c_{23}}$.  By Lemma~\ref{lemma:sp pb stab}, the map $\pi|_{\Omega_{23}}$ is a surjection onto $(\Sp_{2g}(\Z)[2])_{\lax{\vec{v}_{23}}}$.  Each element of $\Omega_{23}$ is reducible by definition, so Theorem~\ref{theorem:reducibletotwists} implies that $\pi|_{\Omega_{23}}$ is injective.  We define $\widetilde{\phi}_{\lax{\vec{v}_{23}}} = \pi|_{\Omega_{23}}^{-1}$. 

We now consider a general $\lax{\vec{v}} \in X_{g}^{(0)}$.  Here we use the action of $\Sp_{2g}(\Z)$ on $\Quotient_g$ provided by Proposition~\ref{proposition:spaction}.  The group $\Sp_{2g}(\Z)$ acts transitively on the vertices of $X_g$ (indeed, any vertex is represented by a one-element partial symplectic basis as in Section~\ref{section:action and connectivity} and $\Sp_{2g}(\Z)$ clearly acts transitively on these), so there exists some $Z \in \Sp_{2g}(\Z)$ such that $Z(\lax{\vec{v}_{23}}) = \lax{\vec{v}}$.  We then define
$$\widetilde{\phi}_{\lax{\vec{v}}}(Y) = Z \cdot \widetilde{\phi}_{\lax{\vec{v}_{23}}}(Z^{-1} Y Z) \quad \quad \left(Y \in \left(\Sp_{2g}(\Z)[2]\right)_{\lax{\vec{v}}}\right).$$
Clearly $\widetilde{\phi}_{\lax{\vec{v}}}$ is a homomorphism.

\begin{lemma}
\label{lemma:phiv wd}
The map $\widetilde{\phi}_{\lax{\vec{v}}}$ does not depend on the choice of $Z$.
\end{lemma}
\begin{proof}
It is enough to show that if $Z \in (\Sp_{2g}(\Z))_{\lax{\vec{v}_{23}}}$ then 
\[
\widetilde{\phi}_{\lax{\vec{v}_{23}}}(Z U Z^{-1}) = Z \cdot \widetilde{\phi}_{\lax{\vec{v}_{23}}}(U)
\]
To prove this, first notice that by the definition of $\widetilde{\phi}_{\lax{\vec{v}_{23}}}$, both $\widetilde{\phi}_{\lax{\vec{v}_{23}}}(Z U Z^{-1})$ and $\widetilde{\phi}_{\lax{\vec{v}_{23}}}(U)$ lie in $\Omega_{23}$.  By Proposition~\ref{proposition:spaction}(3),  $Z \cdot \widetilde{\phi}_{\lax{\vec{v}_{23}}}(U)$ also lies in $\Omega_{23}$.  Since $\pi|_{\Omega_{23}}$ is injective, it remains to show that $\widetilde{\phi}_{\lax{\vec{v}_{23}}}(Z U Z^{-1})$ and $Z \cdot \widetilde{\phi}_{\lax{\vec{v}_{23}}}(U)$ have the same image under $\pi$.  We have
\[ \pi(\widetilde{\phi}_{\lax{\vec{v}_{23}}}(Z U Z^{-1})) = ZUZ^{-1} = Z\pi(\widetilde{\phi}_{\lax{\vec{v}_{23}}}(U))Z^{-1} = \pi(Z \cdot \widetilde{\phi}_{\lax{\vec{v}_{23}}}(U)), \]
where the first and second equalities use the fact that $\pi \circ \widetilde\phi_{\lax{\vec{v}_{23}}}$ equals the identity and the third equality uses Proposition~\ref{proposition:spaction}(1).  
\end{proof}

We will require the following easy consequence of Proposition~\ref{proposition:spaction}(1).

\begin{lemma}
\label{lemma:piphiv}
For any $\lax{\vec{v}} \in X_{g}^{(0)}$, we have $\pi \circ \widetilde{\phi}_{\lax{\vec{v}}} = \mathrm{id}$.
\end{lemma}

\p{Well-definedness of \boldmath$\phi$} The individual maps $\widetilde{\phi}_{\lax{\vec{v}}}$ together define the map $\widetilde \phi$ as in the start of the section.  In order to check that $\widetilde \phi$ descends to a well-defined homomorphism $\phi:\Sp_{2g}(\Z)[2] \to \Quotient_g$, we must check that $\widetilde \phi$ respects the edge and conjugation relations for $\Sp_{2g}(\Z)[2]$ as in Proposition~\ref{prop:sp presentation}.  First we deal with the edge relations.  

\begin{lemma}
\label{lemma:edge}
If $\lax{\vec{v}},\lax{\vec{w}} \in X_g^{(0)}$ are joined by an edge $e$ and $Y \in (\Sp_{2g}(\Z)[2])_{e}$, then \[ \widetilde{\phi}_{\lax{\vec{v}}}(Y) = \widetilde{\phi}_{\lax{\vec{w}}}(Y).\]
\end{lemma}

\begin{proof}

Recall that all simplices of additive type have dimension at least 2, and so we only need to consider standard edges and edges of intersection type.  Let $e_1 = \{ \lax{\vec{v}_{23}}, \lax{\vec{v}_{45}} \}$ and $e_2 = \{ \lax{\vec{v}_{23}}, \lax{\vec{v}_{12}} \}$.  The edge $e_1$ is a standard edge in $X_g$ and $e_2$ is an edge of intersection type.  

The keys to this lemma are the following two facts, the first of which can be proved in the same way as Cases 1 and 2 of Corollary~\ref{corollary:movesimplices} and the second of which is a consequence of the classification of surfaces, cf. \cite[Section 1.3]{FarbMargalitPrimer}:
\begin{enumerate}
\item the group $\Sp_{2g}(\Z)$ acts transitively on the set of standard edges and on the set of edges of intersection type in $X_g$, and
\item for $k \in \{0,2\}$, the group $\Braid_{2g+1}$ acts transitively on the set of ordered pairs of distinct homotopy classes of simple closed curves that intersect $k$ times and surround two marked points each.
\end{enumerate}
Interchanging $\lax{\vec{v}}$ and $\lax{\vec{w}}$ if necessary, the first fact provides a $Z \in \Sp_{2g}(\Z)$ such that $Z(\lax{\vec{v}_{23}}) = \lax{\vec{v}}$ and such that $Z(e_i) = e$ for some $i \in \{1,2\}$.  The second fact gives a $b \in \Braid_{2g+1}$ that interchanges the curves from Figure~\ref{figure:mainproof} corresponding to the endpoints of $e_i$.

Let $\beta = \widehat{\rho}(b) \in \QuotientEx_{g}$.  We have that  $Z \widehat{\pi}(\beta)(\lax{\vec{v}_{23}}) = \lax{\vec{w}}$.  Finally, since $Y$ stabilizes $e$ and $Z(e_i)=e$, it follows that $W = Z^{-1}Y Z \in (\Sp_{2g}(\Z)[2])_{e_i}$.
We claim that
\[ 
\beta^{-1} \widetilde{\phi}_{\lax{\vec{v}_{23}}}(W) \beta = \widetilde{\phi}_{\lax{\vec{v}_{23}}}\left(\widehat{\pi}(\beta)^{-1} W \widehat{\pi}(\beta)\right).
\]
By Lemma~\ref{lemma:sp pb stab} and Theorem~\ref{theorem:reducibletotwists}, $\pi$ restricts to isomorphisms
\begin{align*}
\rho\left((\PureBraid_{2g+1})_{\{c_{23},c_{45}\}}\right)  \to (\Sp_{2g}(\Z)[2])_{e_1} \ \ \ \ \text{ and}\ \ \ \ 
\rho\left((\PureBraid_{2g+1})_{\{c_{23},c_{12}\}}\right) \to (\Sp_{2g}(\Z)[2])_{e_2},
\end{align*}
and hence $\widetilde{\phi}_{\lax{\vec{v}_{23}}}(W)=\rho(a)$ where $a$ lies in $(\PureBraid_{2g+1})_{\{c_{23}, c_{45}\}}$ or $(\PureBraid_{2g+1})_{\{c_{23},c_{12}\}}$.  Because $b$ swaps $c_{23}$ with either $c_{45}$ or $c_{12}$, the braid $b^{-1}ab$ lies in $(\PureBraid_{2g+1})_{\{c_{23}\}}$, and so $\beta^{-1} \widetilde{\phi}_{\lax{\vec{v}_{23}}}(W) \beta$ lies in $\Omega_{23}$.  Using the fact that $\pi \circ \widetilde{\phi}_{\lax{\vec{v}_{23}}} = \textrm{id}$,  we have
\begin{align*}
\pi\left(\beta^{-1} \widetilde{\phi}_{\lax{\vec{v}_{23}}}(W) \beta\right) = \widehat\pi\left(\beta^{-1} \widetilde{\phi}_{\lax{\vec{v}_{23}}}(W) \beta\right) &= \widehat\pi\left(\beta^{-1}\right)  \pi\left(\widetilde{\phi}_{\lax{\vec{v}_{23}}}(W)\right) \widehat\pi (\beta) \\ 
&\ \ \ = \widehat{\pi}(\beta)^{-1} W \widehat{\pi}(\beta). 
\end{align*}
As $\widetilde{\phi}_{\lax{\vec{v}_{23}}}$ and $\pi|_{\Omega_{23}}$ are inverses and $\beta^{-1} \widetilde{\phi}_{\lax{\vec{v}_{23}}}(W) \beta$ lies in $\Omega_{23}$, the claim follows.
The lemma follows easily from the claim and Proposition~\ref{proposition:spaction}(2).  
\end{proof}

The follow lemma states that $\widetilde \phi$ respects the conjugation relations of $\Sp_{2g}(\Z)[2]$.  It follows immediately from Proposition~\ref{proposition:spaction}(2) and Lemma~\ref{lemma:piphiv}.

\begin{lemma}
\label{lemma:conjugation}
For any $Y \in (\Sp_{2g}(\Z)[2])_{\lax{\vec{v}}}$ and $U \in (\Sp_{2g}(\Z)[2])_{\lax{\vec{w}}}$, we have
\[ 
\widetilde{\phi}_{\lax{\vec{w}}}(U) 
\widetilde{\phi}_{\lax{\vec{v}}}(Y) 
\widetilde{\phi}_{\lax{\vec{w}}}(U)^{-1} 
= 
 \widetilde{\phi}_{U(\lax{\vec{v}})}\left((U Y U^{-1})\right). \]
\end{lemma}

\p{Completing the proof of Proposition~\ref{proposition:main}} Since $\widetilde\phi$ respects the edge and conjugation relations, it induces a map $\phi : \Sp_{2g}(\Z)[2] \to \Quotient_g$.  It remains to check that $\phi$ is a left inverse of the projection $\pi : \Quotient_g \to \Sp_{2g}(\Z)[2]$.

\begin{lemma}
\label{lemma:inverse}
We have $\phi \circ \pi = \mathrm{id}$.
\end{lemma}

\begin{proof}

The group $\PureBraid_{2g+1}$ is generated by the conjugates of $T_{c_{23}}$ in $\Braid_{2g+1}$; see Section~\ref{section:presentations}.  Therefore, $\Quotient_g$ is generated by elements of the form $\eta_b = \widehat\rho(b) \rho(T_{c_{23}}) \widehat\rho(b)^{-1}$ 
with $b \in \Braid_{2g+1}$.  Thus, it suffices to check that $\phi ( \pi(\eta_b)) = \eta_b$, where $b$ is an arbitrary element of $\Braid_{2g+1}$.  This follows from Proposition~\ref{proposition:spaction}(1), Proposition~\ref{proposition:spaction}(2), Lemma~\ref{lemma:piphiv}, and the fact that $\rho(T_{c_{23}}) \in \Omega_{23}$.
\end{proof}

\end{document}